\documentclass[preprint, a4paper]{elsarticle}

\usepackage[utf8]{inputenc} 
\usepackage[T1]{fontenc}
\usepackage{a4}
\usepackage{amsfonts, amssymb, amsmath, amsthm}

\usepackage{enumitem}

\usepackage{mathtools} 
\usepackage{tensor} 
\usepackage{scalerel} 

\usepackage{perpage}
\MakePerPage{footnote}

\usepackage[ngerman, english]{babel}
\useshorthands{"}
\makeatletter
 \defineshorthand[ngerman]{"/}{\mbox{-}\bbl@allowhyphens}
\makeatother
\addto\extrasenglish{\languageshorthands{ngerman}}

\usepackage[x11names]{xcolor}
\usepackage[hyperfootnotes=false, unicode=true, linktocpage=true]{hyperref}
\hypersetup{
	pdfauthor={Seerp Roald Koudenburg},
	pdftitle={A categorical approach to the maximum theorem},
  pdfkeywords={Berge's maximum theorem, extreme value theorem, continuity, probabilistic approach space, Kan extension, double category},
  linkbordercolor={NavajoWhite1},
  citebordercolor={LightCyan2},
  urlbordercolor={DarkSeaGreen2}
}

\usepackage{tikz} 
\usetikzlibrary{decorations.markings, matrix, arrows, fit, calc}

\tikzstyle{map} = [->, font=\scriptsize]
\tikzstyle{linj} = [left hook->, font=\scriptsize]
\tikzstyle{rinj} = [right hook->, font=\scriptsize]
\tikzstyle{sur} = [->>, font=\scriptsize]
\tikzstyle{cell} = [double,double equal sign distance,-implies, shorten >= 3.75pt, shorten <= 3.75pt, font=\scriptsize]
\tikzstyle{eq} = [double,double equal sign distance]
\tikzstyle{ps} = [shorten >= 2pt]

\tikzstyle{iso} = [above, sloped, inner sep=1.5pt]
\tikzstyle{nat} = [above, sloped, inner sep=2pt]
\tikzstyle{desc} = [fill=white, inner sep=2pt]
\tikzstyle{dots} = [black, font=]

\tikzstyle{small} = [font=\scriptsize]

\tikzstyle{textbaseline} = [baseline=-2.8pt]

\tikzstyle{barred} = [decoration={markings, mark=at position 0.5 with {\draw[-] (0,-1.5pt) -- (0,1.5pt);}}, postaction ={decorate}]

\tikzstyle{math4} = [matrix of math nodes, row sep={3.5em,between origins}, column sep={4em,between origins}, text height=1.5ex, text depth=0.25ex, nodes in empty cells]
\tikzstyle{math35} = [math4, row sep={3.25em,between origins}, column sep={3.5em,between origins}]
\tikzstyle{minimath} = [matrix of math nodes, row sep={3em,between origins}, column sep={3.25em,between origins}, font=\scriptsize, text height=1ex, text depth=0.25ex, nodes in empty cells]
\tikzstyle{scheme} = [textbaseline, x=1.6em, y=1.6em, yshift=-2.4em, font=\scriptsize, text depth=0ex, every node/.style={overlay}, execute at end picture = { \useasboundingbox ($(current bounding box.north west) + (0,0.4em)$) rectangle ($(current bounding box.south east) - (0,0.4em)$); }]

\makeatletter
\def\slashedarrowfill@#1#2#3#4#5{%
  $\m@th\thickmuskip0mu\medmuskip\thickmuskip\thinmuskip\thickmuskip
   \relax#5#1\mkern-7mu%
   \cleaders\hbox{$#5\mkern-2mu#2\mkern-2mu$}\hfill
   \mathclap{#3}\mathclap{#2}%
   \cleaders\hbox{$#5\mkern-2mu#2\mkern-2mu$}\hfill
   \mkern-7mu#4$%
}
\def\rightslashedarrowfill@{%
  \slashedarrowfill@\relbar\relbar\mapstochar\rightarrow}
\newcommand\xslashedrightarrow[2][]{%
  \ext@arrow 0055{\rightslashedarrowfill@}{#1}{#2}}
\makeatother

\def\slashedrightarrow{\xslashedrightarrow{}}

\newcommand{\dashcirc}{\multimap}
\newcommand{\circdash}{\mathbin{\text{\reflectbox{$\multimap$}}}}

\newtheorem{theorem}{Theorem}[section]
\newtheorem{lemma}[theorem]{Lemma}

\newtheorem{proposition}[theorem]{Proposition}

\theoremstyle{definition}
\newtheorem{definition}[theorem]{Definition}

\theoremstyle{remark}
\newtheorem{remark}[theorem]{Remark}
\newtheorem{example}[theorem]{Example}

\providecommand{\defref}[1]{Definition~\ref{#1}}
\providecommand{\exref}[1]{Example~\ref{#1}}

\providecommand{\lemref}[1]{Lemma~\ref{#1}}
\providecommand{\propref}[1]{Proposition~\ref{#1}}

\providecommand{\thmref}[1]{Theorem~\ref{#1}}

\providecommand{\secref}[1]{Section~\ref{#1}}

\newcommand\defeq{\mathrel{\vcentcolon\Leftrightarrow}}
\newcommand\eqdef{\mathrel{\Leftrightarrow\vcentcolon}}
\renewcommand{\implies}{\Rightarrow}
\renewcommand{\iff}{\Leftrightarrow}

\providecommand{\dfn}{\coloneqq}

\providecommand{\of}{\circ}
\providecommand{\iso}{\cong}

\providecommand{\xrar}[1]{\xrightarrow{#1}}

\providecommand{\into}{\hookrightarrow}

\providecommand{\eps}{\varepsilon}

\newcommand{\dash}{\textup{--}}

\providecommand{\tens}{\otimes}

\providecommand{\mf}[1]{\mathfrak{#1}}

\providecommand{\tdiff}{\ominus}

\providecommand{\brcs}[1]{\lbrace #1 \rbrace}
\providecommand{\bigbrcs}[1]{\bigl\lbrace #1 \bigr\rbrace}
\providecommand{\brks}[1]{\lbrack #1 \rbrack}
\providecommand{\bigbrks}[1]{\bigl\lbrack #1 \bigr\rbrack}

\providecommand{\pars}[1]{\left(#1\right)}
\providecommand{\bigpars}[1]{\bigl(#1\bigr)}
\providecommand{\Bigpars}[1]{\Bigl(#1\Bigr)}

\providecommand{\set}[1]{\brcs{#1}}

\providecommand{\isect}{\cap}
\providecommand{\Isect}{\bigcap}
\providecommand{\union}{\cup}
\providecommand{\Union}{\bigcup}

\DeclareMathOperator\ForallOp{\mbox{\Large $\mathsurround0pt\forall$}}
\newcommand\Forall{\ForallOp\displaylimits}

\providecommand{\downset}{\mathord\downarrow}
\providecommand{\upset}{\mathord\uparrow}
\DeclareMathOperator{\Dn}{Dn}

\DeclareMathOperator{\DnDir}{DnDir}

\providecommand{\rev}[1]{#1^\circ}
\DeclareMathOperator{\Gr}{Gr}

\newcommand{\tn}{\mathbin\&}

\providecommand{\natarrow}{\Rightarrow}

\providecommand{\map}[3]{#1\colon#2\to#3}

\providecommand{\nat}[3]{#1\colon#2\natarrow#3}

\providecommand{\hmap}[3]{#1\colon#2\slashedrightarrow#3}

\providecommand{\inv}[1]{{#1}^{-1}}

\DeclareMathOperator{\im}{im}
\DeclareMathOperator{\id}{id}

\providecommand{\ladj}{\dashv}

\providecommand{\catvar}[1]{\mathcal{#1}}

\providecommand{\Dbl}{\mathsf{Dbl}}

\providecommand{\2}{\mathsf 2}

\providecommand{\K}{\catvar K}

\providecommand{\V}{\catvar V}

\providecommand{\Set}{\mathsf{Set}}

\providecommand{\Cat}{\mathsf{Cat}}
\providecommand{\enCat}[1]{#1\text-\Cat}

\providecommand{\Dbl}{\mathsf{Dbl}}
\providecommand{\lDbl}{\Dbl_\textup l}

\providecommand{\ThinEq}{\mathsf{ThinEquip}}
\providecommand{\lThinEq}{\ThinEq_\textup l}
\providecommand{\nThinEq}{\ThinEq_\textup n}
\providecommand{\strThinEq}{\ThinEq_\textup{str}}

\providecommand{\Top}{\mathsf{Top}}
\providecommand{\vTop}[1]{#1\text-\mathsf{Top}}
\providecommand{\PsTop}[1]{#1\text-\mathsf{PsTop}}
\providecommand{\PreTop}[1]{#1\text-\mathsf{PreTop}}
\providecommand{\ModTop}[1]{#1\text-\mathsf{ModTop}}
\providecommand{\ModPreTop}[1]{#1\text-\mathsf{ModPreTop}}

\providecommand{\App}{\mathsf{App}}

\providecommand{\Cls}[1]{#1\text-\mathsf{Cls}}
\providecommand{\PsCls}[1]{#1\text-\mathsf{PsCls}}
\providecommand{\PreCls}[1]{#1\text-\mathsf{PreCls}}
\providecommand{\ModCls}[1]{#1\text-\mathsf{ModCls}}
\providecommand{\ModPreCls}[1]{#1\text-\mathsf{ModPreCls}}

\providecommand{\Rel}{\mathsf{Rel}}
\providecommand{\enRel}[1]{#1\text-\Rel}
\providecommand{\MetRel}{\mathsf{MetRel}}

\DeclareMathOperator{\Mod}{\mathsf{Mod}}

\providecommand{\Prof}{\mathsf{Prof}}
\providecommand{\enProf}[1]{#1\text-\Prof}

\providecommand{\ModCat}[1]{#1\text-\mathsf{ModCat}}

\providecommand{\Gph}[1]{#1\text-\mathsf{Gph}}
\providecommand{\UGph}[1]{#1\text-\mathsf{UGph}}
\providecommand{\LGph}[1]{#1\text-\mathsf{LGph}}

\providecommand{\ModGph}[1]{#1\text-\mathsf{ModGph}}
\providecommand{\ModRGph}[1]{#1\text-\mathsf{ModRGph}}
\providecommand{\ModUGph}[1]{#1\text-\mathsf{ModUGph}}

\providecommand{\hc}{\odot}

\providecommand{\lhom}{\triangleleft}
\providecommand{\rhom}{\triangleright}

\begin{document}
	\title{A categorical approach to the maximum theorem}
	\author{Seerp Roald Koudenburg}
	\ead{roald@metu.edu.tr}
	\address{Mathematics Research and Teaching Group\\Middle East Technical University, Northern Cyprus Campus\\ 99738 Kalkanl\i, G\"uzelyurt, TRNC, via Mersin 10, Turkey}
	\begin{abstract}
		Berge's maximum theorem gives conditions ensuring the continuity of an optimised function as a parameter changes. In this paper we state and prove the maximum theorem in terms of the theory of monoidal topology and the theory of double categories.
		
		This approach allows us to generalise (the main assertion of) the maximum theorem, which is classically stated for topological spaces, to pseudotopological spaces and pretopological spaces, as well as to closure spaces, approach spaces and probabilistic approach spaces, amongst others. As a part of this we prove a generalisation of the extreme value theorem.
	\end{abstract}
	\maketitle
	
	\setcounter{section}{-1} 
	\section{Introduction} \label{introduction}
	Berge's maximum theorem \cite{Berge59}, which is used in mathematical economics for instance, concerns a relation $\hmap JAB$ between topological spaces, which we regard as a subset $J \subseteq A \times B$, as well as a continuous map $\map dA{\brks{-\infty, \infty}}$ into the extended real line, as depicted in the following diagram.
	\begin{displaymath}
		\begin{tikzpicture}
			\matrix(m)[math35]{A & & B \\ & \brks{-\infty, \infty} & \\};
			\path[map]	(m-1-1) edge node[below left] {$d$} (m-2-2)
													edge[barred] node[above] {$J$} (m-1-3)
									(m-1-3) edge[dashed] node[below right] {$l$} (m-2-2);
		\end{tikzpicture}
	\end{displaymath}
	We may ``extend $d$ along $J$'' by ``optimising $d$ for each $y \in B$'', thus obtaining a map $\map lB{\brks{-\infty, \infty}}$ given by the suprema
	\begin{equation} \label{optimised}
		l(y) = \sup_{x \in \rev Jy} d(x),
	\end{equation}
	where $\rev Jy = \set{x \in A \mid (x, y) \in J}$ denotes the preimage of $y$ under $J$. The main assertion of the maximum theorem states that the optimised function $l$ is continuous as soon as the relation $J$ is `hemi"/continuous' and $\rev Jy \neq \emptyset$ for each $y \in B$. Among the conditions included in hemi"/continuity is the compactness of the preimages $\rev Jy$ so that, by the extreme value theorem, hemi"/continuity of $J$ implies that the suprema defining $l$ are attained as maxima---a consequence that is used in the classical proof of the maximum theorem.
	
	Regarding the ordered set $\brks{-\infty, \infty}$ as a category allows us to think of the suprema in \eqref{optimised} as being limits. In fact, we may consider the full optimised function $l$ as the `left Kan extension of $d$ along $J$', a construction that is fundamental to category theory. Recently it has been shown that, mostly in purely categorical settings, structure on a `morphism' $\map D{\mathcal A}{\mathcal M}$ carries over to Kan extensions of $D$ under certain conditions---the monoidal structure on a functor for instance, see \cite{Mellies-Tabareau08}, \cite{Koudenburg15a} and \cite{Weber16}. The maximum theorem can be thought of as fitting in the same scheme of results: it shows that the continuity of the map $d$ carries over to its left Kan extension $l$. In view of this, one might hope to discover a purely categorical result that, in the topological setting, recovers the classical maximum theorem while, when considered in other settings, allows us to obtain generalisations of the maximum theorem. This paper realises this hope to a large extent.
	
	Besides recognising optimised functions as Kan extensions, the second ingredient of our categorical approach to the maximum theorem is to regard topological structures as algebraic structures---a point of view that forms the basis of the study of `monoidal topology' \cite{Hofmann-Seal-Tholen14}. In the fundamental example for instance, one regards topologies on a set $A$ as closure operations, i.e.\ relations $\hmap c{PA}A$ between the powerset $PA$ of $A$ and $A$ itself: one defines $(S, x) \in c$ precisely if $x \in \bar S$, the closure of $S \subseteq A$. The axioms for a topology on $A$ then translate to three axioms on the `closure relation' $c$ and, by weakening or removing some of these axioms, generalisations of the notion of topological space are recovered, such as that of pretopological space~\cite{Choquet48} and closure space.
	
	The closure relation $\hmap c{PA}A$ above can be equivalently thought of as a map $\map c{PA \times A}{\set{\bot, \top}}$ taking values in the set $\set{\bot, \top}$ of truth values. A second way of generalising the notion of topological space, which is fundamental to monoidal topology, is to replace the set of truth values by a different set of values $\V$. In this way for instance, by considering $\brks{0, \infty}$"/valued closure relations $\map\delta{PA \times A}{\brks{0, \infty}}$, one recovers the notion of approach space \cite{Lowen97}, consisting of a set $A$ equipped with a point"/set distance $\delta(S, x) \in \brks{0, \infty}$ for each subset $S \subseteq A$ and point \mbox{$x \in A$}. Likewise, by allowing closure relations to take `distance distribution functions' \mbox{$\map\phi{\brks{0, \infty}}{\brks{0, 1}}$} as values, one obtains the notion of probabilistic approach space \cite{Lai-Tholen17b}.
	
	Besides closure operations, the notion of topology can be described algebraically in terms of ultrafilter convergence as well \cite{Barr70}: topologies on a set $A$ correspond precisely to convergence relations $\hmap\alpha{UA}A$ satisfying certain axioms, where $UA$ denotes the set of ultrafilters on $A$. As with closure operations, by weakening these axioms, or by considering $\V$"/valued convergence relations $\map\alpha{UA \times A}\V$, one recovers generalisations of the notion of topological space, such as the notions of pretopological space and (probabilistic) approach space, as well as that of pseudotopological space \cite{Choquet48}, amongst others. In our approach to the maximal theorem we will consider both closure relations and ultrafilter convergence relations, as well as the relationship between them. In our study of the latter we closely follow \cite{Lai-Tholen17a}.
	
	The language allowing us to naturally describe the relations between the two ingredients of our approach---Kan extensions and algebraic descriptions of topological structures---is that of double categories, in the sense of e.g.\ \cite{Grandis-Pare04}. The notion of double category extends that of category by considering two types of morphisms instead of the usual single type: e.g.\ between sets we will consider both functions $\map fAC$ as well as $\V$"/valued relations $\map J{A \times B}\V$. Throughout this paper the language of double categories will lead us in the right direction. At the start for instance, when we consider approach spaces (equipped with $\brks{0, \infty}$"/valued closure relations), it naturally leads us to consider Kan extensions that are `weighted' by $\brks{0, \infty}$"/valued relations $\map J{A \times B}\brks{0, \infty}$, instead of Kan extensions along ordinary relations $\hmap JAB$ as described above. Later it naturally leads to the generalisation of the notion of hemi"/continuous relation, as well as to the proper generalisation of Kan extensions ``whose suprema are attained by maxima''. Finally, the language of double categories leads us to the generalisations of the maximum theorem and of the extreme value theorem themselves.
	
	In closing this introduction we remark on some present restrictions of our approach. In recent work on the maximum theorem (e.g.\ \cite{Feinberg-Kasyanov15}), as well as in recent textbooks (e.g.\ \cite{Aliprantis-Border06}), the name `maximum theorem' is designated to a result that extends and generalises in two ways the main assertion described above:
	\begin{enumerate}[label=-]
		\item more generally, it concerns optimisations $\map kB\brks{-\infty, \infty}$ of the form 		\begin{displaymath}
			k(y) = \sup_{x \in \rev Jy} e(x, y)
		\end{displaymath}
		for each $y \in B$, where $\hmap JAB$ is a hemi"/continuous relation with graph $\Gr J \subseteq A \times B$ and $\map e{\Gr J}\brks{-\infty, \infty}$ is a continuous function;

		\item besides continuity of the optimised function $k$, it also proves the `upper hemi"/continuity' of the `solution relation' $\hmap{J^*}AB$ that is defined by 
		\begin{displaymath}
			(x, y) \in J^* \quad \defeq \quad (x, y) \in J \quad \text{and} \quad e(x, y) = ky.
		\end{displaymath}
	\end{enumerate}
	Investigating ways of incorporating these generalisations in the categorical approach presented here have to be left as a further study.
	
	\section*{Outline}
	We start in \secref{thin equipments} by recalling the language and basic theory of double categories, mostly from \cite{Grandis-Pare99} and \cite{Grandis-Pare04}. Guided by the classical setting of functions $\map fAC$ and relations $\hmap JAB$ between sets, we restrict to double categories whose cells, which describe the relations between the two types of morphism, are uniquely determined by their boundaries, and in which every `vertical morphism' $\map fAC$ induces two corresponding `horizontal morphisms' $\hmap{f_*}AC$ and $\hmap{f^*}CA$. Such double categories we will call `thin equipments'. Our main examples are the thin equipments $\enRel\V$, of relations \mbox{$\map J{A \times B}\V$} taking values in a `quantale' $\V$: loosely speaking, any ordered set $\V$ with ``enough structure to replace the set of truth values''. After recalling some examples of quantales, such as the quantale $\Delta$ of distance distribution functions, we recall the notion of `monoid' in a thin equipment. Monoids in $\enRel{\set{\bot, \top}}$ are ordered sets while monoids in $\enRel{\brks{0, \infty}}$ and $\enRel\Delta$ respectively recover the notions of generalised metric space \cite{Lawvere73} and probabilistic metric space \cite{Menger42}.
	
	In \secref{Kan extension section} the double categorical notion of Kan extension, introduced in \cite{Koudenburg14}, is considered in thin equipments. After describing Kan extensions between monoids in $\enRel\V$, we consider the classical situation of a Kan extension into $\brks{-\infty, \infty}$ whose suprema are attained as maxima, and generalise it in terms of a `Beck"/Chevalley condition' for Kan extensions. The main result of this section shows that, in a thin equipment, Kan extensions satisfying the Beck"/Chevalley condition are precisely the `absolute Kan extensions' of \cite{Grandis-Pare08}.
	
	Given a `monad' $T$ on a thin equipment, we start \secref{T-graphs} by recalling from e.g.\ \cite{Hofmann-Seal-Tholen14} the notions of `$T$"/graph' and `$T$"/category', as well as some related notions. For the ultrafilter monad $U$ extended to ordinary relations these notions recover those of pseudotopological space and topological space, as well as that of pretopological space. Extending $U$ to $\V$"/valued relations recovers to the notion of $\V$"/valued topological space \cite{Lai-Tholen17b} which, by taking $\V = \brks{0, \infty}$ and $\V = \Delta$, includes the notions of approach space and probabilistic approach space respectively. Likewise, by taking the powerset monad extended to $\V$"/valued relations we obtain the notion of $\V$"/valued (pre"/)closure space and several of its generalisations. As a variant on the main theorem of \cite{Lai-Tholen17a}, which shows that $\V$"/valued topological spaces correspond precisely to $\V$"/valued closure spaces whose closure relations $\map c{PA \times A}\V$ `preserve finite joins', the main result of this section establishes a similar correspondence between $\V$"/valued pretopological spaces and `finite"/join"/preserving' $\V$"/valued preclosure spaces.
	
	In \secref{modular T-graphs} we consider objects in a thin equipment that are equipped with compatible monoid and $T$"/graph structures. Following \cite{Tholen09} we call such objects `modular $T$"/graphs'. We prove that the correspondences described in \secref{T-graphs} lift to give correspondences between modular $\V$"/valued (pre"/)topological spaces and finite"/join"/preserving modular $\V$"/valued (pre"/)closure spaces.
	
	The generalisations of the maximum theorem given in \secref{maximum theorem} apply to Kan extensions $\map lBM$, between $T$"/graphs, that satisfy one of the following conditions. Either $l$ satisfies the Beck"/Chevalley condition, in the sense of \secref{Kan extension section}, or the $T$"/graph $M$ is `$T$"/cocomplete', as described in \secref{T-complete T-graphs}. The latter condition extends the $T$"/cocompleteness property considered in \cite{Hofmann-Seal-Tholen14}. Loosely speaking, the $T$"/graph structure of a $T$"/cocomplete modular $T$"/graph is completely determined by its `generic points': a modular approach space $A$ for example, equipped with both a generalised metric $A(x, y) \in \brks{0, \infty}$, where $x, y \in A$, and a $\brks{0, \infty}$"/valued ultrafilter convergence $\hmap\alpha{UA}A$, is $U$"/cocomplete whenever for every ultrafilter $\mf x$ on $A$ a generic point $x_0 \in A$ is chosen such that
	\begin{displaymath}
		\alpha(\mf x, y) = A(x_0, y)
	\end{displaymath}
	for all $y \in A$. We will describe how every `completely distributive' quantale $\V$ itself admits two $U$"/cocomplete modular $\V$"/valued topological space structures.
	
	In \secref{horizontal T-morphisms} the notions of lower and upper hemi"/continuity, for ordinary relations between topological spaces, are generalised to the notions of `$T$"/open' and `$T$"/closed' horizontal morphism $\hmap JAB$ between $T$"/graphs $A$ and $B$. Restricting ourselves to the extensions $P$ and $U$ of the powerset and ultrafilter monads to $\V$"/relations, we describe the relationship between $P$"/openness and $U$"/openness, as well as that between $P$"/closedness and $U$"/closedness, in terms of the correspondences between $\V$"/valued (pre"/)closure spaces and $\V$"/valued (pre"/)topological spaces given in \secref{T-graphs}.
	
	Finally in \secref{maximum theorem} we state and prove four generalisations of the classical maximum theorem, in terms of Kan extensions between $T$"/graphs. These come in pairs, one pair for `left' Kan extensions and the other for `right' Kan extensions: each pair either assumes that the Kan extension satisfies the Beck"/Chevalley condition, in the sense of \secref{Kan extension section}, or has a $T$"/cocomplete target, in the sense of \secref{T-complete T-graphs}. Besides showing how to recover the classical maximum theorem we describe its generalisations to preclosure spaces, approach spaces and probabilistic approach spaces.
	
	Generalising the classical extreme value theorem, in the last section we obtain conditions  ensuring the Beck"/Chevalley condition for Kan extensions between modular $\V$"/valued pseudotopological spaces.
	
	\section{Thin equipments} \label{thin equipments}
	In this preliminary section we consider the notion of thin equipment, which forms the main setting for this paper. As this notion is modeled to describe the interaction between functions $\map fAC$ and relations $\hmap JAB$ between sets, we start by briefly setting out some notation for relations. We think of a relation $\hmap JAB$ as a subset $J \subseteq A \times B$, and shorten $(x, y) \in J$ to $xJy$. We will write
	\begin{displaymath}
		JS = \set{y \in B \mid \exists x \in S\colon xJy}
	\end{displaymath} 
	for the $J$"/image of $S \subseteq A$; also we write $Jx \dfn J\set x$ for all $x \in A$. The \emph{reverse} $\hmap{\rev J}BA$ of $J$ is defined by
	\begin{displaymath}
		y\rev Jx \qquad \defeq \qquad xJy,
	\end{displaymath}
	allowing us to denote by $\rev JT$ the $J$"/preimage of $T \subseteq B$.
	
	Relations between $A$ and $B$ are ordered by inclusion; in fact, to describe the interplay between functions and relations it is useful to depict by a cell
	\begin{displaymath}
		\begin{tikzpicture}
			\matrix(m)[math35]{A & B \\ C & D \\};
			\path[map]	(m-1-1) edge[barred] node[above] {$J$} (m-1-2)
													edge node[left] {$f$} (m-2-1)
									(m-1-2) edge node[right] {$g$} (m-2-2)
									(m-2-1) edge[barred] node[below] {$K$} (m-2-2);
			\draw				($(m-1-1)!0.5!(m-2-2)$) node[rotate=-90] {$\leq$};
		\end{tikzpicture}
	\end{displaymath}
  the property that $(fx)K(gy)$ for every $xJy$. For example these cells allow us to formalise, in the definition of thin equipment below, the relation between a function $\map fAC$ and the two relations $\hmap{f_*}AC$ and $\hmap{f^*}CA$ that it induces, that are defined by
  \begin{displaymath}
  	x(f_*)y \qquad \defeq \qquad fx = y \qquad \eqdef \qquad y(f^*)x.
  \end{displaymath}
  
  Notice that cells like the one above can be composed both vertically and horizontally: any two vertically adjacent cells combine as on the left below while any two horizontally adjacent cells combine as on the right. Here $\hc$ denotes the usual composition of relations: $x(J \hc H)z$ precisely if $xJy$ and $yHz$ for some $y \in B$.
	\begin{equation} \label{composition of cells}
		\begin{tikzpicture}[textbaseline]
			\matrix(m)[math35]{A & B \\ C & D \\ G & H \\};
			\path[map]	(m-1-1) edge[barred] node[above] {$J$} (m-1-2)
													edge node[left] {$f$} (m-2-1)
									(m-1-2) edge node[right] {$g$} (m-2-2)
									(m-2-1) edge[barred] node[below] {$K$} (m-2-2)
									        edge node[left] {$h$} (m-3-1)
									(m-2-2) edge node[right] {$k$} (m-3-2)
									(m-3-1) edge[barred] node[below] {$L$} (m-3-2);
			\draw				($(m-1-1)!0.5!(m-2-2)$) node[rotate=-90] {$\leq$}
									($(m-2-1)!0.5!(m-3-2)$) node[rotate=-90] {$\leq$};
		\end{tikzpicture} \mapsto \begin{tikzpicture}[textbaseline]
			\matrix(m)[math35]{A & B \\ G & H \\};
			\path[map]	(m-1-1) edge[barred] node[above] {$J$} (m-1-2)
													edge node[left] {$hf$} (m-2-1)
									(m-1-2) edge node[right] {$kg$} (m-2-2)
									(m-2-1) edge[barred] node[below] {$K$} (m-2-2);
			\draw				($(m-1-1)!0.5!(m-2-2)$) node[rotate=-90] {$\leq$};
		\end{tikzpicture} \qquad \begin{tikzpicture}[textbaseline]
			\matrix(m)[math35]{A & B & E \\ C & D & F \\};
			\path[map]	(m-1-1) edge[barred] node[above] {$J$} (m-1-2)
													edge node[left] {$f$} (m-2-1)
									(m-1-2) edge[barred] node[above] {$H$} (m-1-3)
													edge node[right] {$g$} (m-2-2)
									(m-1-3) edge node[right] {$l$} (m-2-3)
									(m-2-1) edge[barred] node[below] {$K$} (m-2-2)
									(m-2-2) edge[barred] node[below] {$M$} (m-2-3);
			\draw				($(m-1-1)!0.5!(m-2-2)$) node[rotate=-90] {$\leq$}
									($(m-1-2)!0.5!(m-2-3)$) node[rotate=-90] {$\leq$};
		\end{tikzpicture} \mapsto \begin{tikzpicture}[textbaseline]
			\matrix(m)[math35]{A & E \\ C & F \\};
			\path[map]	(m-1-1) edge[barred] node[above] {$J \hc H$} (m-1-2)
													edge node[left] {$f$} (m-2-1)
									(m-1-2) edge node[right] {$g$} (m-2-2)
									(m-2-1) edge[barred] node[below] {$K \hc M$} (m-2-2);
			\draw				($(m-1-1)!0.5!(m-2-2)$) node[rotate=-90] {$\leq$};
		\end{tikzpicture}
	\end{equation}
  
  The preceding describes the prototypical thin equipment $\Rel$, of functions and relations between sets. It naturally gives rise to the following general definition of thin equipment.
  \begin{definition}
  	A \emph{thin equipment} $\K$ consists of a pair of categories $\K_\textup v = (\K_\textup v, \of, \id)$ and $\K_\textup h = (\K_\textup h, \hc, 1)$, on the same collection of objects, that is equipped with a collection $\K_\textup c$ of square"/shaped cells
  	\begin{equation} \label{cell}
			\begin{tikzpicture}[textbaseline]
				\matrix(m)[math35]{A & B \\ C & D, \\};
				\path[map]	(m-1-1) edge[barred] node[above] {$J$} (m-1-2)
														edge node[left] {$f$} (m-2-1)
										(m-1-2) edge node[right] {$g$} (m-2-2)
										(m-2-1) edge[barred] node[below] {$K$} (m-2-2);
				\draw				($(m-1-1)!0.5!(m-2-2)$) node[rotate=-90] {$\leq$};
			\end{tikzpicture}
		\end{equation}
  	each of which is uniquely determined by its boundary morphisms $f, g \in \K_\textup v$ and $J, K \in \K_\textup h$. This data is required to satisfy the following axioms:
  	\begin{enumerate}[label=-]
  		\item $\K_\textup c$ is closed under vertical and horizontal composition as depicted in \eqref{composition of cells} above;
  		\item $\K_\textup c$ contains identity cells as shown below, one for each $f \in \K_\textup v$ and one for each $J \in \K_\textup h$;
	  		\begin{displaymath}
  				\begin{tikzpicture}
						\matrix(m)[math35]{A & A \\ C & C \\};
						\path[map]	(m-1-1) edge[barred] node[above] {$1_A$} (m-1-2)
																edge node[left] {$f$} (m-2-1)
												(m-1-2) edge node[right] {$f$} (m-2-2)
												(m-2-1) edge[barred] node[below] {$1_C$} (m-2-2);
						\draw				($(m-1-1)!0.5!(m-2-2)$) node[rotate=-90] {$\leq$};
					\end{tikzpicture} \qquad \qquad \qquad \qquad \begin{tikzpicture}
						\matrix(m)[math35]{A & B \\ A & B \\};
						\path[map]	(m-1-1) edge[barred] node[above] {$J$} (m-1-2)
																edge node[left] {$\id_A$} (m-2-1)
												(m-1-2) edge node[right] {$\id_B$} (m-2-2)
												(m-2-1) edge[barred] node[below] {$J$} (m-2-2);
						\draw				($(m-1-1)!0.5!(m-2-2)$) node[rotate=-90] {$\leq$};
					\end{tikzpicture}
  			\end{displaymath}
  		\item the ordering on morphisms in $\K_\textup h$ that is induced by $\K_\textup c$ is separated, that is the existence of both cells below implies $J = K$;
  			\begin{equation} \label{separate}
  				\begin{tikzpicture}[textbaseline]
						\matrix(m)[math35]{A & B \\ A & B \\};
						\path[map]	(m-1-1) edge[barred] node[above] {$J$} (m-1-2)
																edge node[left] {$\id_A$} (m-2-1)
												(m-1-2) edge node[right] {$\id_B$} (m-2-2)
												(m-2-1) edge[barred] node[below] {$K$} (m-2-2);
						\draw				($(m-1-1)!0.5!(m-2-2)$) node[rotate=-90] {$\leq$};
					\end{tikzpicture} \qquad \qquad \qquad \qquad \begin{tikzpicture}[textbaseline]
						\matrix(m)[math35]{A & B \\ A & B \\};
						\path[map]	(m-1-1) edge[barred] node[above] {$K$} (m-1-2)
																edge node[left] {$\id_A$} (m-2-1)
												(m-1-2) edge node[right] {$\id_B$} (m-2-2)
												(m-2-1) edge[barred] node[below] {$J$} (m-2-2);
						\draw				($(m-1-1)!0.5!(m-2-2)$) node[rotate=-90] {$\leq$};
					\end{tikzpicture}	
  			\end{equation}
  		\item for each morphism $\map fAC$ in $\K_\textup v$ there are two morphisms $\hmap{f_*}AC$ and $\hmap{f^*}CA$ in $\K_\textup h$ such that the cells below exist.
  	\end{enumerate}
  	\begin{displaymath}
			\begin{tikzpicture}
				\matrix(m)[math35]{A & C \\ C & C \\};
				\path[map]	(m-1-1) edge[barred] node[above] {$f_*$} (m-1-2)
														edge node[left] {$f$} (m-2-1)
										(m-1-2) edge node[right] {$\id_C$} (m-2-2)
										(m-2-1) edge[barred] node[below] {$1_C$} (m-2-2);
				\draw				($(m-1-1)!0.5!(m-2-2)$) node[rotate=-90] {$\leq$};
			\end{tikzpicture} \quad \quad \begin{tikzpicture}
				\matrix(m)[math35]{A & A \\ A & C \\};
				\path[map]	(m-1-1) edge[barred] node[above] {$1_A$} (m-1-2)
														edge node[left] {$\id_A$} (m-2-1)
										(m-1-2)	edge node[right] {$f$} (m-2-2)
										(m-2-1) edge[barred] node[below] {$f_*$} (m-2-2);
				\draw				($(m-1-1)!0.5!(m-2-2)$) node[rotate=-90] {$\leq$};
			\end{tikzpicture} \qquad \qquad \begin{tikzpicture}
				\matrix(m)[math35]{C & A \\ C & C \\};
				\path[map]	(m-1-1) edge[barred] node[above] {$f^*$} (m-1-2)
										(m-1-2)	edge node[right] {$f$} (m-2-2)
										(m-1-1) edge node[left] {$\id_C$} (m-2-1)
										(m-2-1) edge[barred] node[below] {$1_C$} (m-2-2);
				\draw				($(m-1-1)!0.5!(m-2-2)$) node[rotate=-90] {$\leq$};
			\end{tikzpicture} \quad \quad \begin{tikzpicture}
				\matrix(m)[math35]{A & A \\ C & A \\};
				\path[map]	(m-1-1) edge[barred] node[above] {$1_A$} (m-1-2)
										(m-1-2) edge node[right] {$\id_A$} (m-2-2)
										(m-1-1) edge node[left] {$f$} (m-2-1)
										(m-2-1) edge[barred] node[below] {$f^*$} (m-2-2);
				\draw				($(m-1-1)!0.5!(m-2-2)$) node[rotate=-90] {$\leq$};
			\end{tikzpicture}
		\end{displaymath} 
  \end{definition}
  
  We call the morphisms of $\K_\textup v$ the \emph{vertical morphisms} of $\K$ and those of $\K_\textup h$ the \emph{horizontal morphisms}. Cells with identities as vertical morphisms, such as in \eqref{separate}, are called \emph{horizontal cells}; if either cell in \eqref{separate} exists then we write $J \leq K$ or $K \leq J$ respectively.
  
  For $\map fAC$ in $\K_\textup v$ we call the horizontal morphism $\hmap{f_*}AC$ the \emph{companion} of $f$ and $\hmap{f^*}CA$ the \emph{conjoint} of $f$. Notice that the companion and conjoint of $f$ are uniquely determined by the existence of the four cells above, as a consequence of the separated ordering on horizontal morphisms. It follows that $(g \of f)_* = f_* \hc g_*$ and $(g \of f)^* = g^* \hc f^*$ for composable morphisms $f$ and $g$, while $(\id_A)_* = 1_A = (\id_A)^*$; in short the assignments $f \mapsto f_*$ and $f \mapsto f^*$ are functorial. Given morphisms $\map fAC$, $\hmap KCD$ and $\map gBD$ we write
  \begin{displaymath}
  	K(f, g) \dfn f_* \hc K \hc g^*
  \end{displaymath}
	and call $K(f, g)$ the \emph{restriction of $K$ along $f$ and $g$}; notice that $1_C(f, \id) = f_*$ and $1_C(\id, f) = f^*$. In terms of restrictions the functoriality of companions and conjoints means that $K(f, g)(h, k) = K(h \of f, k \of g)$ and $K(\id, \id) = K$.
	
  When drawing cells we will often depict identity morphisms by the equal sign (=). Although the cells of a thin equipment are uniquely determined by their boundaries, often it will be useful to give them names. In those cases we will use greek letters $\phi$, $\psi$,\dots, as well as denoting vertical and horizontal composition of cells by $\of$ and $\hc$, while vertical and horizontal unit cells will be denoted by $1_f$ and $\id_J$ respectively.
  
   Besides the direct definition given above, a thin equipment can equivalently be defined as a flat strict double category, in the sense of Section~1 of \cite{Grandis-Pare99}, whose horizontal bicategory is locally skeletal and in which every vertical morphism has both a companion and conjoint (also called horizontal adjoint), the latter in the sense of Section~1 of \cite{Grandis-Pare04}. The term `equipment' originates from the term `proarrow equipment' used by Wood in \cite{Wood82} for structures closely related to ``double categories $\K$ with all companions and conjoints'': one can think of such $\K$ as equipping their underlying vertical bicategories with the `proarrows' of their underlying horizontal bicategories.
  
	\begin{example}
		Instead of the classical relations between sets $A$ and $B$ we will also consider \emph{metric relations} $\hmap JAB$, given by functions $\map J{A \times B}{\brks{0,\infty}}$. Composition of metric relations $\hmap JAB$ and $\hmap HBE$ is given by ``shortest path distance''
		\begin{displaymath}
			(J \hc H)(x, z) = \inf_{y \in B} J(x, y) + H(y, z).
		\end{displaymath}
		Together with functions between sets, metric relations form a thin equipment $\MetRel$ in which a cell as in \eqref{cell} exists precisely if $J(x, y) \geq K(fx, gy)$ for all $x \in A$ and $y \in B$.
	\end{example}
  
  \begin{example}
  	Generalising the previous example, relations between sets can take values in any `quantale' as follows. A \emph{quantale} $\V = (\V, \tens, k)$ is a complete lattice $\V$ equipped with a (not necessarily commutative) monoid structure $\tens$ with unit $k$, such that $\tens$ preserves suprema on both sides. Given a quantale $\V$, a \emph{$\V$"/relation} $\hmap JAB$ between sets $A$ and $B$ is a function $\map J{A \times B}\V$. The composite of $\V$"/relations $\hmap JAB$ and $\hmap HBE$ is given by ``matrix multiplication''
  	\begin{displaymath}
  		(J \hc H)(x, z) = \sup_{y \in B} J(x, y) \tens H(y, z);
  	\end{displaymath}
  	the identity $\V$-relations $\hmap{1_A}AA$ for this composition are given by $1_A(x, y) = k$ if $x = y$ and $1_A(x, y) = \bot$ if $x \neq y$, where $\bot = \sup \emptyset$ is the bottom element of $\V$. Functions and $\V$-relations between sets combine to form a thin equipment $\enRel\V$, in which a cell as in \eqref{cell} exists precisely if $J(x, y) \leq K(fx, gy)$ for all $x \in A$ and $y \in B$. Since the ordering on $\V$ is separated the ordering on parallel $\V$"/relations is separated as well. The companion $\hmap{f_*}AC$ and conjoint $\hmap{f^*}CA$ of a function $\map fAC$ are the $\V$"/relations given by $f_*(x, y) = k = f^*(y, x)$ if $fx = y$ and $f_*(x, y) = \bot = f^*(y, x)$ if $fx \neq y$. The restriction $K(f, g)$ of a $\V$"/relation $\hmap KCD$ along functions $\map fAC$ and $\map gBD$ is indeed given by restriction: $K(f, g)(x, y) = K(fx, gy)$ for all $x \in A$ and $y \in B$.
  	
  	If $\V$ is the two-chain $\2 = \brcs{\bot \leq \top}$ of truth values, equipped with the monoid structure $(\wedge, \top)$ given by conjunction, then $\enRel\2$ is isomorphic to the thin equipment $\Rel$ of relations, under the identification of ordinary relations $J \subseteq A \times B$ with $\2$"/relations $\map J{A \times B}\2$. If $\V$ is the \emph{Lawvere quantale} $\brks{0, \infty}$, equipped with the opposite order  $\geq$ and the monoid structure $(+, 0)$, then $\enRel{\brks{0, \infty}}$ coincides with the thin equipment $\MetRel$ of the previous example. Similarly the completion $\brks{-\infty, \infty}$ of $\mathbb R$, either with the natural order $\leq$ or with the reversed order $\geq$, forms a quantale under addition. As is customary, when referring to infima and suprema in $(\brks{0, \infty}, \geq)$ and $(\brks{-\infty, \infty}, \geq)$ we will always consider the natural order $\leq$.
  	
  	In the same vein the unit interval $\brks{0, 1}$, with its natural order, admits several monoid structures $\tn$ that make it into a quantale: one can take the usual multiplication $\tn = \times$ of real numbers, the frame operation $p \tn q = \min\set{p, q}$ or the \L ukasiewicz operation $p \tn q = \max\set{p+q-1,0}$. Notice that, besides preserving suprema in both variables, each of these multiplications is commutative and has unit $k = 1$: monoid structures on $\brks{0, 1}$ with these properties are known as \emph{left"/continuous t"/norms}.
  \end{example}
  
  \begin{example} \label{Delta}
  	A \emph{distance distribution function} is a function $\map \phi{\brks{0, \infty}}{\brks{0, 1}}$ satifying the left"/continuity condition $\sup_{s < t} \phi (s) = \phi(t)$ for all $t \in \brks{0, \infty}$. As a consequence $\phi$ preserves the natural order $\leq$, while $\phi(0) = 0$. Any left"/continuous t"/norm $\tn$ on $\brks{0,1}$ induces a quantale structure on the set $\Delta$ of all distance distribution functions: $\Delta$ inherits a pointwise ordering from $\brks{0,1}$ while its multiplication is given by the convolution product
  	\begin{displaymath}
  		(\phi \tens \psi)(t) = \sup_{r + s \leq t} \phi(r) \tn \psi(s), 
  	\end{displaymath}
  	for all $t \in \brks{0, \infty}$. The resulting quantales $\Delta_{\tn}$ share their unit $k$, which is given by $k(t) = 1$ for $t > 0$ and $k(0) = 0$, while their orderings fail to be linear.
  \end{example}
  
  \begin{example}
  	Any frame $\V$, that is a lattice such that $v \mapsto \min\set{v, w}$ preserves suprema for all $w \in \V$, can be regarded as a quantale with $v \tens w = \min\set{v, w}$.
  \end{example}
  
  By using companions and conjoints any general cell in a thin equipment corresponds to a couple of horizontal cells as follows.
	\begin{lemma} \label{horizontal cells}
		In a thin equipment consider morphisms as in the boundary of the cell below. The cell below exists if and only if $J \hc g_* \leq K(f, \id)$ if and only if $f^* \hc J \leq K(\id, g)$.
		\begin{displaymath}
			\begin{tikzpicture}
				\matrix(m)[math35]{A & B \\ C & D \\};
				\path[map]	(m-1-1) edge[barred] node[above] {$J$} (m-1-2)
														edge node[left] {$f$} (m-2-1)
										(m-1-2) edge node[right] {$g$} (m-2-2)
										(m-2-1) edge[barred] node[below] {$K$} (m-2-2);
				\draw				($(m-1-1)!0.5!(m-2-2)$) node[rotate=-90] {$\leq$};
			\end{tikzpicture}
		\end{displaymath}
	\end{lemma}
	\begin{proof}
		By composing the cell above with the cells defining the companions and conjoints of $f$ and $g$ we obtain the horizontal cells that exhibit the inequalities. In the same way the cell can be recovered from either horizontal cell that exhibits one of the inequalities.
	\end{proof}
  
  As in any double category (see e.g.\ Section 11 of \cite{Shulman08}) one can consider monoids in a thin equipment, as follows.
  \begin{definition} \label{monoid}
  	Let $\K$ be a thin equipment.
  	\begin{enumerate}[label=-]
  		\item A \emph{monoid} $A = (A, \bar A)$ in $\K$ is an object $A$ equipped with a horizontal morphism $\hmap{\bar A}AA$ (which we will often denote by $A$) satisfying the associativity and unit axioms $\bar A \hc \bar A \leq \bar A$ and $1_A \leq \bar A$.
  		\item A vertical morphism $\map fAC$ between monoids is called a \emph{monoid homomorphism} if the cell on the left below exists.
  		\begin{displaymath}
  			\begin{tikzpicture}
					\matrix(m)[math35]{A & A \\ C & C \\};
					\path[map]	(m-1-1) edge[barred] node[above] {$\bar A$} (m-1-2)
															edge node[left] {$f$} (m-2-1)
											(m-1-2) edge node[right] {$f$} (m-2-2)
											(m-2-1) edge[barred] node[below] {$\bar C$} (m-2-2);
					\draw				($(m-1-1)!0.5!(m-2-2)$) node[rotate=-90] {$\leq$};
				\end{tikzpicture} \qquad \qquad \qquad \qquad \begin{tikzpicture}
					\matrix(m)[math35]{A & B \\ C & D \\};
					\path[map]	(m-1-1) edge[barred] node[above] {$J$} (m-1-2)
															edge node[left] {$f$} (m-2-1)
											(m-1-2) edge node[right] {$g$} (m-2-2)
											(m-2-1) edge[barred] node[below] {$K$} (m-2-2);
					\draw				($(m-1-1)!0.5!(m-2-2)$) node[rotate=-90] {$\leq$};
				\end{tikzpicture}
			\end{displaymath}
		\item A horizontal morphism $\hmap JAB$ between monoids is called a \emph{bimodule} if $\bar A \hc J \hc \bar B \leq J$.
		\item A cell between monoid homomorphisms and bimodules, as on the right above, is simply a cell in $\K$ between the underlying vertical and horizontal morphisms.
  	\end{enumerate}
  	The structure on $\K$ lifts to make monoids, their homomorphisms and bimodules, as well as the cells between those, into a thin equipment $\Mod(\K)$. The unit bimodule of a monoid $A$ is its structure morphism $\hmap{\bar A}AA$, while the companion and conjoint of a monoid morphism $\map fAC$ are the bimodules given by the restrictions $f_* = \bar C(f, \id)$ and $f^* = \bar C(\id, f)$. The restriction $K(f, g)$ of a bimodule $\hmap KCD$ along homomorphisms $\map fAC$ and $\map gBD$ coincides with the restriction $K(f, g)$ of the underlying horizontal morphism $K$ in $\K$, along the vertical morphisms underlying $f$ and $g$.
  \end{definition}
  
  \begin{example} \label{V-profunctors}
  	Being an ordered set we may regard any quantale $\V = (\V, \tens, k)$ as a category; the monoid structure $(\tens, k)$ then makes $\V$ into a monoidal category. In these terms monoids in $\enRel\V$ are precisely \emph{$\V$"/enriched categories}, in the usual sense of e.g.\ \cite{Kelly82}, while their homomorphisms are \emph{$\V$"/functors}. A bimodule $\hmap JAB$ is a \emph{$\V$"/bimodule} in the sense of Section 3 of \cite{Lawvere73}: a $\V$"/relation $\hmap JAB$ such that
  	\begin{displaymath}
  		A(x_1, x_2) \tens J(x_2, y_1) \tens B(y_1, y_2) \leq J(x_1, y_2)
  	\end{displaymath}
  	for all $x_1, x_2 \in A$ and $y_1, y_2 \in B$. Also called \emph{$\V$"/distributors}, we will call such bimodules \emph{$\V$"/profunctors}. We write $\enProf\V \dfn \Mod(\enRel\V)$.
  	
  	We remark that $\V$, as a monoidal category, is biclosed: the suprema preserving maps $x \tens \dash$ and $\dash \tens y$, where $x, y \in \V$, have right adjoints $x \dashcirc \dash$ and $\dash \circdash y$ defined by
  	\begin{equation} \label{inner homs}
  		y \leq x \dashcirc z \qquad \iff \qquad x \tens y \leq z \qquad \iff \qquad x \leq z \circdash y
  	\end{equation}
  	for all $x, y, z \in \V$ or, equivalently,
  	\begin{displaymath}
  		x \dashcirc z = \sup \set{v \in \V \mid x \tens v \leq z} \qquad \text{and} \qquad z \circdash y = \sup \set{v \in \V \mid v \tens y \leq z}.
  	\end{displaymath}
  	Both $\dashcirc$ and $\circdash$ can be used to enrich $\V$ over itself, resulting in two $\V$"/categories $\V_\dashcirc$ and $\V_{\circdash}$ with hom"/objects $\V_{\dashcirc}(x, y) = x \dashcirc y$ and $\V_{\circdash}(x, y) = x \circdash y$ respectively. If the monoid structure on $\V$ is commutative then the right adjoints $x \dashcirc \dash$ and $\dash \circdash x$ coincide.
  	
  	We will use the fact that the adjoints $x \dashcirc \dash$ and $\dash \circdash y$ induce right adjoints to the sup"/maps $J \hc \dash$ and $\dash \hc H$, for any $\V$"/relations $\hmap JAB$ and $\hmap HBE$. Denoting these adjoints by $J \lhom \dash$ and $\dash \rhom H$ respectively, they are defined by
  	\begin{displaymath}
  		H \leq J \lhom K \qquad \iff \qquad J \hc H \leq K \qquad \iff \qquad J \leq K \rhom H
  	\end{displaymath}
  	for all $\hmap JAB$, $\hmap HBE$ and $\hmap KAE$ or, equivalently,
  	\begin{flalign*}
  		&& (J \lhom K)(y, z) &= \inf_{x \in A} J(x, y) \dashcirc K(x, z)& \\
  		\text{and} && (K \rhom H)(x, y) &= \inf_{z \in E} K(x, z) \circdash H(y, z),&
  	\end{flalign*}
  	for all $x \in A$, $y \in B$ and $z \in E$.
	\end{example}
	\begin{example} \label{modular relation}
  	Monoids $A$ in $\enRel\2$, that is categories enriched in the set $\2$ of truth values, can be identified with \emph{ordered sets}, whose order relations $\hmap{\bar A}AA$ are transitive and reflexive, while homomorphisms of such monoids are order preserving maps. A $\2$"/profunctor $\hmap JAB$ between ordered sets is a \emph{modular relation} satisfying
  	\begin{displaymath}
  		x_1 \leq x_2, \quad x_2Jy_1 \quad \text{and} \quad y_1 \leq y_2 \quad \implies \quad x_1Jy_2,
  	\end{displaymath}
  	for all $x_1, x_2 \in A$ and $y_1, y_2 \in B$. The value of $y \dashcirc z$ in $\2$ is the Boolean truth value of the implication $y \to z$, so that the two ways of enriching $\2$ over itself simply recover the natural and reversed orderings of $\2$.
  \end{example}
  
  \begin{example}	\label{non-expansive relation}
  	If $\V$ is the Lawvere quantale $\brks{0, \infty}$ then a monoid $A$ in $\enRel\V$, that is a $\brks{0, \infty}$"/category, is a \emph{generalised metric space} in Lawvere's sense \cite{Lawvere73}, whose distance function $\map A{A \times A}{\brks{0, \infty}}$ satisfies
  	\begin{displaymath}
  		A(x, y) + A(y, z) \geq A(x, z) \qquad \text{and} \qquad A(x, x) = 0
  	\end{displaymath}
  	for all $x$, $y$ and $z \in A$, but which need not be symmetric. A $\brks{0, \infty}$"/functor $\map fAC$ is a \emph{non"/expansive map}, that satisfies $A(x,y) \geq C(fx, fy)$ for all $x, y \in A$, while a $\brks{0, \infty}$"/profunctor $\hmap JAB$ is a \emph{modular metric relation} $\map J{A \times B}{\brks{0, \infty}}$, satisfying
  	\begin{displaymath}
  		A(x_1, x_2) + J(x_2, y_1) + B(y_1, y_2) \geq J(x_1, y_2)
  	\end{displaymath}
  	for all $x_1, x_2 \in A$ and $y_1, y_2 \in B$. In $\V = \brks{0, \infty}$ the number $z \circdash y$ is the \emph{truncated difference} $z \tdiff y = \max \set{z - y, 0}$, so that the two ways of enriching $\brks{0, \infty}$ over itself equip it with the (non"/symmetric) metrics $\brks{0, \infty}_\dashcirc(x, y) = y \tdiff x$ and $\brks{0, \infty}_{\circdash}(x, y) = x \tdiff y$.
  \end{example}
  
  \begin{example}
  	The notion of metric space is further generalised by enriching over the extended real numbers $(\brks{-\infty, \infty}, \geq)$ instead, thus allowing negative distances as well. Willerton in \cite{Willerton15} uses such $\brks{-\infty, \infty}$"/categories in giving a category theoretic perspective of the Legendre"/Fenchel transform, while Lawvere in \cite{Lawvere84} takes a categorical approach to entropy using categories enriched over $(\brks{-\infty, \infty}, \leq)$.
  \end{example}
  
  \begin{example}
  	Analogous to the relation between metric spaces and $\brks{0, \infty}$"/categories, Flagg notes in \cite{Flagg97} (or see Section III.2.1 of \cite{Hofmann-Seal-Tholen14}) that \emph{probabilistic metric spaces}, originally introduced by Menger in \cite{Menger42}, can be regarded as categories enriched in the quantales $\Delta_{\tn}$ (\exref{Delta}) of distance distribution functions. Instead of real"/valued distances, any pair $(x, y)$ of points in a probabilistic metric space $A$ is equipped with a distance distribution function $A(x, y) \in \Delta$. For each $t \in \brks{0, \infty}$, the value $A(x, y)(t) \in \brks{0, 1}$ is to be thought of as the ``probability that the distance between $x$ and $y$ is less than $t$''.
  \end{example}
  
  We close this section by restricting to the setting of thin equipments the notions of lax functor between double categories and (vertical) transformation of such functors, both introduced in Section~7 of \cite{Grandis-Pare99}.
  \begin{definition} \label{functor and transformation}
  	A \emph{lax functor} $\map F\K\mathcal L$ between thin equipments $\K$ and $\mathcal L$ consists of a functor $\map{F_\textup v}{\K_\textup v}{\mathcal L_\textup v}$ (which will be denoted $F$) as well as an assignment of horizontal morphisms
  	\begin{displaymath}
  		\hmap JAB \qquad \mapsto \qquad \hmap{FJ}{FA}{FB}
  	\end{displaymath}
  that preserves horizontal composition laxly, that is
  	\begin{displaymath}
  		1_{FA} \leq F1_A \qquad \text{and} \qquad FJ \hc FH \leq F(J \hc H)
  	\end{displaymath}
  	for any object $A$ and composable morphisms $J$ and $H$ in $\K_\textup h$, such that the existence of any cell in $\K$ as on the left below implies the existence of the middle cell in $\mathcal L$.
  	\begin{displaymath}
			\begin{tikzpicture}
				\matrix(m)[math35]{A & B \\ C & D \\};
				\path[map]	(m-1-1) edge[barred] node[above] {$J$} (m-1-2)
														edge node[left] {$f$} (m-2-1)
										(m-1-2) edge node[right] {$g$} (m-2-2)
										(m-2-1) edge[barred] node[below] {$K$} (m-2-2);
				\draw				($(m-1-1)!0.5!(m-2-2)$) node[rotate=-90] {$\leq$};
			\end{tikzpicture} \qquad \qquad \qquad \begin{tikzpicture}
				\matrix(m)[math35]{FA & FB \\ FC & FD \\};
				\path[map]	(m-1-1) edge[barred] node[above] {$FJ$} (m-1-2)
														edge node[left] {$Ff$} (m-2-1)
										(m-1-2) edge node[right] {$Fg$} (m-2-2)
										(m-2-1) edge[barred] node[below] {$FK$} (m-2-2);
				\draw				($(m-1-1)!0.5!(m-2-2)$) node[rotate=-90] {$\leq$};
			\end{tikzpicture} \qquad \qquad \qquad \begin{tikzpicture}
				\matrix(m)[math35]{FA & FB \\ GA & GB \\};
				\path[map]	(m-1-1) edge[barred] node[above] {$FJ$} (m-1-2)
														edge node[left] {$\xi_A$} (m-2-1)
										(m-1-2) edge node[right] {$\xi_B$} (m-2-2)
										(m-2-1) edge[barred] node[below] {$GJ$} (m-2-2);
				\draw				($(m-1-1)!0.5!(m-2-2)$) node[rotate=-90] {$\leq$};
			\end{tikzpicture}
  	\end{displaymath}
  	
  	A \emph{transformation} $\nat\xi FG$ of lax functors $F$ and $\map G\K\mathcal L$ is a natural transformation $\nat{\xi_\textup v}{F_\textup v}{G_\textup v}$ (which will be denoted $\xi$) such that for every horizontal morphism $\hmap JAB$ in $\K$ the \emph{naturality cell} on the right above exists in $\mathcal L$.
  \end{definition}
  
  Thin equipments, lax functors and their transformations form a $2$"/category that we will denote $\lThinEq$; it is a full sub"/2"/category of the $2$"/category $\lDbl$ of double categories, lax functors and their transformations.
  
  A lax functor $\map F\K\mathcal L$ is called \emph{normal} if it preserves horizontal units strictly, that is $F1_A = 1_{FA}$ for all $A \in \K$; a \emph{strict functor} $\map F\K\mathcal L$ is a lax functor that preserves both units and horizontal compositions strictly. Notice that a lax functor $F$ is normal if and only if it preserves companions and conjoints, in the sense that $F(f_*) = (Ff)_*$ and $F(f^*) = (Ff)^*$ for all $\map fAC$ in $\K$. On the other hand any lax functor preserves restrictions, as the following restriction of Proposition~6.8 of \cite{Shulman08} to thin equipments shows.
  \begin{proposition}[Shulman] \label{lax functors preserve restrictions}
  	For any lax functor $\map F\K\mathcal L$ and morphisms $\map fAC$, $\hmap KCD$ and $\map gBD$ in $\K$ we have $F\bigpars{K(f, g)} = (FK)(Ff, Fg)$.
  \end{proposition}
  \begin{proof}
  	To obtain $F\bigpars{K(f, g)} \leq (FK)(Ff, Fg)$ we apply $F$ to the composite of cells on the left below and compose the result with the appropriate cells among those that define $(Ff)_*$ and $(Fg)^*$.
  	\begin{displaymath}
  		\begin{tikzpicture}
  			\matrix(m)[math35]{A & C & D & B \\ C & C & D & D \\};
  			\path[map]	(m-1-1) edge[barred] node[above] {$f_*$} (m-1-2)
  													edge node[left] {$f$} (m-2-1)
  									(m-1-2) edge[barred] node[above] {$K$} (m-1-3)
  									(m-1-3) edge[barred] node[above] {$g^*$} (m-1-4)
  									(m-1-4) edge node[right] {$g$} (m-2-4)
  									(m-2-2) edge[barred] node[below] {$K$} (m-2-3);
  			\path				(m-1-2) edge[eq] (m-2-2)
  									(m-1-3) edge[eq] (m-2-3)
  									(m-2-1) edge[eq] (m-2-2)
  									(m-2-3) edge[eq] (m-2-4);
				\draw				($(m-1-1)!0.5!(m-2-2)$) node[rotate=-90] {$\leq$}
										($(m-1-3)!0.5!(m-2-4)$) node[rotate=-90] {$\leq$};
  		\end{tikzpicture} \qquad \begin{tikzpicture}
  			\matrix(m)[math35]{FA & FA & FC & FD & FB & FB \\ FA & FC & FC & FD & FD & FB \\};
  			\path[map]	(m-1-1) edge[barred] node[above] {$F1_A$} (m-1-2)
  									(m-1-2) edge[barred] node[above] {$(Ff)_*$} (m-1-3)
  													edge node[left, inner sep=1.5pt] {$Ff$} (m-2-2)
  									(m-1-3) edge[barred] node[above] {$FK$} (m-1-4)
  									(m-1-4) edge[barred] node[above] {$(Fg)^*$} (m-1-5)
  									(m-1-5) edge[barred] node[above] {$F1_B$} (m-1-6)
  									 				edge node[left, inner sep=1.5pt] {$Fg$} (m-2-5)
  									(m-2-1) edge[barred] node[below] {$F(f_*)$} (m-2-2)
  									(m-2-3) edge[barred] node[below] {$FK$} (m-2-4)
  									(m-2-5) edge[barred] node[below] {$F(g^*)$} (m-2-6);
  			\path				(m-1-1) edge[eq] (m-2-1)
  									(m-1-3) edge[eq] (m-2-3)
  									(m-1-4) edge[eq] (m-2-4)
  									(m-1-6) edge[eq] (m-2-6)
  									(m-2-2) edge[eq] (m-2-3)
  									(m-2-4) edge[eq] (m-2-5);
				\draw				($(m-1-1)!0.5!(m-2-2)$) node[rotate=-90] {$\leq$}
										($(m-1-2)!0.5!(m-2-3)$) node[rotate=-90] {$\leq$}
										($(m-1-4)!0.5!(m-2-5)$) node[rotate=-90] {$\leq$}
										($(m-1-5)!0.5!(m-2-6)$) node[rotate=-90] {$\leq$};
  		\end{tikzpicture}
  	\end{displaymath}
  	The inverse $(FK)(Ff, Fg) \leq F\bigpars{K(f, g)}$ is obtained by composing the composite on the right above, whose leftmost and rightmost cells are `$F$"/images' of cells defining $f_*$ and $g^*$ respectively, with the lax structure cells $1_{FA} \leq F1_A$, $1_{FB} \leq F1_B$ and $F(f_*) \hc FK \hc F(g^*) \leq F(f_* \hc K \hc g^*)$.
  \end{proof}
  
  We write $\strThinEq \subset \nThinEq \subset \lThinEq$ for the locally full sub"/2"/categories generated by the strict and normal functors respectively. The following is Proposition~11.12 of \cite{Shulman08} restricted to thin equipments.
	\begin{proposition}[Shulman] \label{2-functor Mod}
		The assignment $\K \mapsto \Mod(\K)$ of \defref{monoid} extends to a strict $2$-functor $\map\Mod\lThinEq\nThinEq$, which restricts to a $2$"/functor $\strThinEq \to \strThinEq$.
	\end{proposition}
	\begin{proof}[Sketch of the proof.]
		The image
		\begin{displaymath}
			\map{\Mod F}{\Mod(\K)}{\Mod(\mathcal L)}
		\end{displaymath}
		of a lax functor $\map F\K\mathcal L$ between thin equipments simply applies $F$ indexwise; e.g.\ it maps a monoid $A = (A, \bar A)$ in $\K$ to the monoid $(\Mod F)(A) \dfn (FA, F\bar A)$ in $\mathcal L$. Notice that $\Mod F$ is normal, while it is strict whenever $F$ is. The naturality cells of a transformation $\nat\xi FG$ ensure that, for every monoid $A$ in $\K$, the component $\map{\xi_A}{FA}{GA}$ is a homomorphism of monoids, so that these components combine to form a transformation $\nat{\Mod\xi}{\Mod F}{\Mod G}$.
	\end{proof}

  \section{Kan extensions in thin equipments} \label{Kan extension section}
  Using thin equipments as environment, in this section we describe the first ingredient of our categorical approach to the maximum theorem: the notion of Kan extension, which generalises that of optimised function. In the definition below we start by restricting the notion of left Kan extension in a general double category, that was introduced in \cite{Koudenburg14} under the name `pointwise left Kan extension', to thin equipments. Afterwards we will describe and study a type of Kan extension that generalises those optimised functions given by suprema that are attained as maxima, as described in the Introduction.
  \begin{definition}
  	Let $\map dAM$ and $\hmap JAB$ be morphisms in a thin equipment $\K$. The cell $\eta$ in the right"/hand side below defines $\map lBM$ as the \emph{left Kan extension of $d$ along $J$} if every cell in $\K$, of the form as on the left"/hand side, factors through $\eta$ as shown.
  	\begin{equation} \label{left Kan extension}
  		\begin{tikzpicture}[textbaseline]
				\matrix(m)[math35]{A & B & C \\ M & & M \\};
				\path[map]	(m-1-1) edge[barred] node[above] {$J$} (m-1-2)
														edge node[left] {$d$} (m-2-1)
										(m-1-2) edge[barred] node[above] {$H$} (m-1-3)
										(m-1-3) edge node[right] {$g$} (m-2-3);
				\path				(m-2-1) edge[eq] (m-2-3);
				\draw				($(m-1-1)!0.5!(m-2-3)$) node[rotate=-90] {$\leq$};
			\end{tikzpicture} = \begin{tikzpicture}[textbaseline]
				\matrix(m)[math35]{A & B & C \\ M & M & M \\};
				\path[map]	(m-1-1) edge[barred] node[above] {$J$} (m-1-2)
														edge node[left] {$d$} (m-2-1)
										(m-1-2) edge[barred] node[above] {$H$} (m-1-3)
														edge node[right] {$l$} (m-2-2)
										(m-1-3) edge node[right] {$g$} (m-2-3);
				\path				(m-2-1) edge[eq] (m-2-2)
										(m-2-2) edge[eq] (m-2-3);
				\draw				($(m-1-1)!0.5!(m-2-2)$) node {$\eta$}
										($(m-1-2)!0.5!(m-2-3)$) node[rotate=-90] {$\leq$};
			\end{tikzpicture}
  	\end{equation}
  	
  	Horizontally dual, the cell $\eps$ in the right"/hand side below defines $\map rAM$ as the \emph{right Kan extension of $\map eBM$ along $\hmap JAB$} if every cell in $\K$, of the form as on the left"/hand side, factors through $\eps$ as shown.
  	\begin{displaymath}
  		\begin{tikzpicture}[textbaseline]
				\matrix(m)[math35]{C & A & B \\ M & & M \\};
				\path[map]	(m-1-1) edge[barred] node[above] {$H$} (m-1-2)
														edge node[left] {$f$} (m-2-1)
										(m-1-2) edge[barred] node[above] {$J$} (m-1-3)
										(m-1-3) edge node[right] {$e$} (m-2-3);
				\path				(m-2-1) edge[eq] (m-2-3);
				\draw				($(m-1-1)!0.5!(m-2-3)$) node[rotate=-90] {$\leq$};
			\end{tikzpicture} = \begin{tikzpicture}[textbaseline]
				\matrix(m)[math35]{C & A & B \\ M & M & M \\};
				\path[map]	(m-1-1) edge[barred] node[above] {$H$} (m-1-2)
														edge node[left] {$f$} (m-2-1)
										(m-1-2) edge[barred] node[above] {$J$} (m-1-3)
														edge node[right] {$r$} (m-2-2)
										(m-1-3) edge node[right] {$e$} (m-2-3);
				\path				(m-2-1) edge[eq] (m-2-2)
										(m-2-2) edge[eq] (m-2-3);
				\draw				($(m-1-2)!0.5!(m-2-3)$) node {$\eps$}
										($(m-1-1)!0.5!(m-2-2)$) node[rotate=-90] {$\leq$};
			\end{tikzpicture}
  	\end{displaymath}
  \end{definition}
  
  For a quantale $\V$ and a $\V$"/functor $\map jAB$ the following proposition implies that, in the thin equipment $\enProf\V$ of $\V$"/profunctors, left Kan extensions along the companion $\hmap{j_*}AB$ coincide with $\V$"/enriched left Kan extensions along $j$ in the usual sense (see e.g.\ Section 4 of \cite{Kelly82}). The same holds for right Kan extensions along the conjoint $\hmap{j^*}BA$.
  \begin{proposition}\label{Kan extension expression}
  	Let $\V$ be a quantale, let $\map dAM$ and $\map lBM$ be $\V$"/functors and let $\hmap JAB$ be a $\V$"/profunctor. The $\V$"/functor $l$ is the left Kan extension of $d$ along $J$ in $\enProf\V$ precisely if
  	\begin{equation} \label{defining equation}
  		M(ly, z) = \inf_{x \in A} J(x,y) \dashcirc M(dx, z)
  	\end{equation}
  	for all $y \in B$ and $z \in M$. In particular if $M = \V_{\dashcirc}$ (\exref{V-profunctors}) then $l$ is given by
  	\begin{displaymath}
  		ly = \sup_{x \in A} dx \tens J(x,y)
  	\end{displaymath}
  	while if $M = \V_{\circdash}$ and $\V$ is commutative then $l$ is given by
  	\begin{displaymath}
  		ly = \inf_{x \in A} dx \circdash J(x,y).
  	\end{displaymath}
  	Dually a $\V$"/functor $\map rAM$ is the right Kan extension of a $\V$"/functor $\map eBM$ along $\hmap JAB$ precisely if
  	\begin{displaymath}
  		M(z, rx) = \inf_{y \in B} M(z, ey) \circdash J(x, y)
  	\end{displaymath}
  	for all $x \in A$ and $z \in M$. If $M = \V_{\circdash}$ then $r$ is given by
  	\begin{displaymath}
  		rx = \sup_{y \in B} J(x,y) \tens ey;
  	\end{displaymath}
  	if $M = \V_{\dashcirc}$ and $\V$ is commutative then $r$ is given by
  	\begin{displaymath}
  		rx = \inf_{y \in B} J(x,y) \dashcirc ey.
  	\end{displaymath}
  \end{proposition}
  \begin{proof}[Sketch of the proof]
  	We sketch the proof for the left Kan extension $\map lBM$ of $\map dAM$ along $\hmap JAB$; the proof for right Kan extensions is analogous. For the `if'"/part first notice that the existence of the cell $\eta$ follows from the fact that
  	\begin{displaymath}
  		k \leq M(ly, ly) = \inf_{x \in A} J(x, y) \dashcirc M(dx, ly) \leq J(x, y) \dashcirc M(dx, ly)
  	\end{displaymath}
  	for all $x \in A$ and $y \in B$, combined with the definition of $\dashcirc$, see \eqref{inner homs}. To see that it satisfies the univeral property \eqref{left Kan extension} notice that, by the definitions of $\hc$ and $\dashcirc$, the cell on the left"/hand side of \eqref{left Kan extension} exists precisely if $H(y, z) \leq J(x,y) \dashcirc M(dx, gz)$ for all $x \in A$, $y \in B$ and $z \in C$, so that $H(y, z) \leq M(ly, gz)$ follows.
  	
  	For the `precisely'"/part assume that $l$ satisfies the universal property \eqref{left Kan extension}, let $y \in B$ and $z \in M$, and let $v \in \V$ be such that $v \leq J(x,y) \dashcirc M(dx, z)$ for all $x \in A$. To show that $v \leq M(ly, z)$ consider on the left"/hand side of \eqref{left Kan extension} the cell with $C = *$, the unit $\V$"/category with single object $*$ and hom"/object $*(*, *) = k$, $\map gCM$ given by $g(*) = z$ and $\hmap HBC$ given by $H(s,*) = v$ if $s = y$ and $H(s,*) = \bot$ otherwise. That this cell exists follows from the assumption on $v$, while factorising it through $\eta$ gives $v \leq M(ly, z)$.
  	
  	In the case that $M = \V_{\dashcirc}$ the equation defining $l$ reduces to
  	\begin{multline*}
  		ly \dashcirc z = \inf_{x \in A} \bigpars{J(x,y) \dashcirc (dx \dashcirc z)} \\ = \inf_{x \in A} \bigpars{\bigpars{dx \tens J(x,y)} \dashcirc z} = \bigpars{\sup_{x \in A} dx \tens J(x,y)} \dashcirc z,
  	\end{multline*}
  	for all $y \in B$ and $z \in M$, where we use that $\dash \dashcirc z$ transforms suprema into infima. Using the fact that $\V$ is separated we conclude that $ly$ and $\sup_{x \in A} dx \tens J(x,y)$ coincide for all $y \in B$.
  	
  	Finally if $M = \V_{\circdash}$ and $\V$ is commutative then the equation defining $l$ reduces to
  	\begin{multline*}
  		ly \circdash z = \inf_{x \in A} \bigpars{(dx \circdash z) \circdash J(x, y)} \\
  		= \inf_{x \in A} \bigpars{\bigpars{dx \circdash J(x,y)} \circdash z} = \bigpars{\inf_{x \in A} dx \circdash J(x,y)} \circdash z
  	\end{multline*}
  	for all $y \in B$ and $z \in M$, where we use that $\dash \circdash z$ preserves infima. As before $ly = \inf_{x \in A} dx \circdash J(x,y)$ follows.
  \end{proof}
  
  \begin{example} \label{Kan extensions in ordered sets}
  	Let $A$, $B$ and $M$ be ordered sets, $\map dAM$ a monotone map and $\hmap JAB$ a modular relation (\exref{modular relation}). Taking $\V = 2$ in the previous proposition, the defining equation \eqref{defining equation} of the left Kan extension $\map lBM$ of $d$ along $J$ reduces to
  	\begin{displaymath}
  		ly = \sup_{x \in \rev Jy} dx \qquad \text{where} \qquad \rev Jy = \set{x \in A \mid xJy},
  	\end{displaymath}
  	so that $l$ exists whenever these suprema exist. Dually the right Kan extension $\map rAM$ of a monotone map $\map eBM$ along $\hmap JAB$ is given by the infima
  	\begin{displaymath}
  		rx = \inf_{y \in Jx} ey \qquad \text{where} \qquad Jx = \set{y \in B \mid xJy}.
  	\end{displaymath}
  	
  	For general quantales $\V$ notice that, if $d$ is a $\V$"/functor $A \to \V_{\dashcirc}$, then by the previous proposition $\map lBM$ above can be regarded as a $\V$"/enriched Kan extension as well, along the $\V$"/profunctor $\hmap{J_\V}AB$ that is given by $J_\V(x,y) = k$ if $xJy$ and $J_\V(x,y) = \bot$ otherwise. Likewise if $\V$ is commutative and $e$ is a $\V$"/functor $B \to \V_{\dashcirc}$ then $\map rAM$ above is the \emph{left} Kan extension, in $\enProf\V$, of $e$ along the $\V$"/profunctor $\hmap{(\rev J)_\V}BA$.
  \end{example}
  
  Having introduced Kan extensions we now consider the notion of exact cell. The corresponding notion for general double categories, that was introduced in \cite{Koudenburg14} under the name `pointwise exact cell', generalises the classical notion of `exact square' of functors, as studied by Guitart in \cite{Guitart80}.
  \begin{definition}
  	The cell $\phi$ on the left below is called \emph{left exact} if, for any cell $\eta$ as in the middle that defines $l$ as a left Kan extension, the vertical composite $\eta \of \phi$ defines $l \of g$ as a left Kan extension. Dually $\phi$ is called \emph{right exact} if, for any cell $\eps$ below that defines $r$ as a right Kan extension, the composite $\eps \of \phi$ defines $r \of f$ as a right Kan extension.
  	\begin{displaymath}
  		\begin{tikzpicture}
				\matrix(m)[math35]{A & B \\ C & D \\};
				\path[map]	(m-1-1) edge[barred] node[above] {$J$} (m-1-2)
														edge node[left] {$f$} (m-2-1)
										(m-1-2) edge node[right] {$g$} (m-2-2)
										(m-2-1) edge[barred] node[below] {$K$} (m-2-2);
				\draw				($(m-1-1)!0.5!(m-2-2)$) node {$\phi$};
			\end{tikzpicture} \qquad \quad \qquad \begin{tikzpicture}
				\matrix(m)[math35]{C & D \\ M & M \\};
				\path[map]	(m-1-1) edge[barred] node[above] {$K$} (m-1-2)
														edge node[left] {$d$} (m-2-1)
										(m-1-2) edge node[right] {$l$} (m-2-2);
				\path				(m-2-1) edge[eq] (m-2-2);
				\draw				($(m-1-1)!0.5!(m-2-2)$) node {$\eta$};
			\end{tikzpicture} \qquad \quad \qquad \begin{tikzpicture}
				\matrix(m)[math35]{C & D \\ M & M \\};
				\path[map]	(m-1-1) edge[barred] node[above] {$K$} (m-1-2)
														edge node[left] {$r$} (m-2-1)
										(m-1-2) edge node[right] {$e$} (m-2-2);
				\path				(m-2-1) edge[eq] (m-2-2);
				\draw				($(m-1-1)!0.5!(m-2-2)$) node {$\eps$};
			\end{tikzpicture}
  	\end{displaymath}
  \end{definition}
  
  Notice that if the cell $\eta$ above is itself left exact then it defines $l$ as the \emph{absolute} left Kan extension of $d$ along $K$: for any morphism $\map kMN$ the composite $1_k \of \eta$ defines $k \of l$ as a left Kan extension. Likewise if $\eps$ above is right exact then it defines $r$ as an absolute right Kan extension.
  
  The following result restates Corollary 4.5 of \cite{Koudenburg14} in the setting of thin equipments. For each cell $\phi$ as on the left above we will call the hypotheses of the parts (a) and (b) below the left and right \emph{Beck"/Chevalley conditions} for $\phi$ respectively.
  \begin{proposition}
  	For a cell $\phi$ in a thin equipment, as on the left above, the following hold; compare \lemref{horizontal cells}.
  	\begin{enumerate}[label=\textup{(\alph*)}]
  		\item If $f^* \hc J = K(\id, g)$ then $\phi$ is left exact.
  		\item If $J \hc g_* = K(f, \id)$ then $\phi$ is right exact.
  	\end{enumerate}
  \end{proposition}
  
  As we shall see shortly, the following theorem describes a type of left Kan extension that generalises those optimised functions given by suprema that are attained by maxima. We will call its hypothesis the \emph{Beck"/Chevalley condition} for left Kan extensions. Horizontally dual, we say that a right Kan extension satisfies the \emph{Beck"/Chevalley condition} whenever its defining cell $\eps$ satisfies the right Beck"/Chevalley condition or, equivalently, $\eps$ satisfies the universal property that is horizontally dual to that described in the theorem below. In a general double category, right Kan extensions defined by cells satisfying such a universal property were introduced by Grandis and Par\'e in \cite{Grandis-Pare08}, where they were called `absolute right Kan extensions'.
  \begin{theorem} \label{Beck-Chevalley}
  	In a thin equipment consider a cell $\eta$ as in the right"/hand side below. It satisfies the left Beck"/Chevalley condition if and only if any cell as on the left"/hand side factors through $\eta$ as shown.
  	\begin{displaymath}
			\begin{tikzpicture}[textbaseline]
				\matrix(m)[math35]{A & B & E \\ M & & N \\};
				\path[map]	(m-1-1) edge[barred] node[above] {$J$} (m-1-2)
														edge node[left] {$d$} (m-2-1)
										(m-1-2) edge[barred] node[above] {$H$} (m-1-3)
										(m-1-3) edge node[right] {$g$} (m-2-3)
										(m-2-1) edge[barred] node[below] {$K$} (m-2-3);
				\draw				($(m-1-1)!0.5!(m-2-3)$) node[rotate=-90] {$\leq$};
			\end{tikzpicture} = \begin{tikzpicture}[textbaseline]
				\matrix(m)[math35]{A & B & E \\ M & M & N \\};
				\path[map]	(m-1-1) edge[barred] node[above] {$J$} (m-1-2)
														edge node[left] {$d$} (m-2-1)
										(m-1-2) edge[barred] node[above] {$H$} (m-1-3)
														edge node[right] {$l$} (m-2-2)
										(m-1-3) edge node[right] {$g$} (m-2-3)
										(m-2-2) edge[barred] node[below] {$K$} (m-2-3);
				\path				(m-2-1) edge[eq] (m-2-2);
				\draw				($(m-1-1)!0.5!(m-2-2)$) node {$\eta$}
										($(m-1-2)!0.5!(m-2-3)$) node[rotate=-90] {$\leq$};
			\end{tikzpicture}
  	\end{displaymath}
  	In particular, in this case $l$ is the absolute left Kan extension of $d$ along $J$.
  \end{theorem}
  \begin{proof}
  	The `only if'"/part. Suppose that $\eta$ satisfies the left Beck"/Chevalley condition, that is $d^* \hc J = l^*$. We have to show that any cell as on the left"/hand side above factors through $\eta$ as shown. By \lemref{horizontal cells} we may equivalently prove that $l^* \hc H \leq K(\id, g)$. This is shown below, where the identity is the assumption on $\eta$ and the inequality follows from applying \lemref{horizontal cells} to the cell on the left"/hand side above.
  	\begin{displaymath}
  		l^* \hc H = d^* \hc J \hc H \leq K(\id, g)
  	\end{displaymath}
  	
  	The `if'"/part. Assuming that $\eta$ satisfies the universal property above, we have to show that $d^* \hc J = l^*$. Applying \lemref{horizontal cells} to $\eta$ gives $d^* \hc J \leq l^*$. For the reverse inequality apply the same lemma to the factorisation through $\eta$ in
  	\begin{displaymath}
  		\begin{tikzpicture}[textbaseline]
					\matrix(m)[math35]{A & A & B \\ M & A & B \\};
					\path[map]	(m-1-2) edge[barred] node[above] {$J$} (m-1-3)
											(m-1-1) edge node[left] {$d$} (m-2-1)
											(m-2-1) edge[barred] node[below] {$d^*$} (m-2-2)
											(m-2-2) edge[barred] node[below] {$J$} (m-2-3);
					\path				(m-1-3) edge[eq] (m-2-3)
											(m-1-1) edge[eq] (m-1-2)
											(m-1-2)	edge[eq] (m-2-2);
					\draw				($(m-1-1)!0.5!(m-2-2)$) node[rotate=-90] {$\leq$};
				\end{tikzpicture} = \begin{tikzpicture}[textbaseline]
					\matrix(m)[math35]{A & B & & B \\ M & M & A & B, \\};
					\path[map]	(m-1-1) edge[barred] node[above] {$J$} (m-1-2)
															edge node[left] {$d$} (m-2-1)
											(m-1-2) edge node[right] {$l$} (m-2-2)
											(m-2-3) edge[barred] node[below] {$J$} (m-2-4)
											(m-2-2) edge[barred] node[below] {$d^*$} (m-2-3);
					\path				(m-1-2) edge[eq] (m-1-4)
											(m-1-4)	edge[eq] (m-2-4)
											(m-2-1) edge[eq] (m-2-2);
					\draw				($(m-1-1)!0.5!(m-2-2)$) node {$\eta$}
											($(m-1-2)!0.5!(m-2-4)$) node[rotate=-90] {$\leq$};
				\end{tikzpicture}
  	\end{displaymath}
  	which exists by the assumption on $\eta$.
  \end{proof}
  
  \begin{example} \label{Beck-Chevalley for ordered sets}
  	Given ordered sets $A$, $B$ and $M$, let $\map dAM$ be a monotone map and $\hmap JAB$ a modular relation. If the left Kan extension $\map lBM$ of $d$ along $J$ exists then, as we saw in \exref{Kan extensions in ordered sets}, $l$ is given by
  	\begin{displaymath}
  		ly = \sup_{x \in \rev Jy} dx.
  	\end{displaymath}
  	It is easily checked that the Beck"/Chevalley condition for $l$ states that for every $y \in B$ there is $x \in \rev Jy$ with $dx = ly$, that is the suprema above are attained as maxima.
  \end{example}
  \begin{example} \label{Beck-Chevalley for metric spaces}
  	Given generalised metric spaces $A$, $B$ and $M$, let $\map dAM$ be a non"/expanding map and let $\hmap JAB$ be a modular metric relation (\exref{non-expansive relation}). By the proposition below the left Kan extension $\map lBM$ of $d$ along $J$, if it exists, satisfies the Beck"/Chevalley condition precisely when
  	\begin{displaymath}
  		\inf_{x \in A} M(ly, dx) + J(x, y) = 0
  	\end{displaymath}
  	for all $y \in B$.
  \end{example}
  
  If the map $d$ in \exref{Beck-Chevalley for ordered sets} above is a continuous map $\map dA\brks{-\infty, \infty}$, then the Beck"/Chevalley condition holds whenever the pre"/images $\rev Jy$, for each $y \in B$, are non"/empty and compact in $A$: this is a direct consequence of Weierstra\ss' extreme value theorem, see e.g.\ Corollary 2.35 of \cite{Aliprantis-Border06} or \thmref{extreme value theorem for ordered closure spaces} below. In \secref{extreme value theorem section} we will generalise the extreme value theorem to left Kan extensions of morphisms $\map dA{\V_{\dashcirc}}$ of `$\V$"/valued topological spaces', a notion that is recalled in the next section.
  
  \begin{proposition}\label{Beck-Chevalley condition}
  	Let $\V$ be a quantale, let $\map dAM$ be a $\V$"/functor and $\hmap JAB$ a $\V$"/profunctor. The left Kan extension $\map lBM$ of $d$ along $J$, if it exists, satisfies the Beck"/Chevalley condition precisely when
  	\begin{displaymath}
  		k \leq \sup_{x \in A} M(ly, dx) \tens J(x, y)
  	\end{displaymath}
  	for all $y \in B$.
  \end{proposition}
  \begin{proof}
  	For the `when'"/part we have to show that $d^* \hc J = l^*$ follows from the inequality above. Firstly $d^* \hc J \leq l^*$ is obtained by applying \lemref{horizontal cells} to the universal cell defining $l$. That the reverse inequality follows from the inequality above is shown by
  	\begin{align*}
  		l^*(z, y) &= M(z, ly) = M(z, ly) \tens k \leq M(z, ly) \tens \bigpars{\sup_{x \in A} M(ly, dx) \tens J(x,y)} \\
  		&= \sup_{x\in A} \bigpars{M(z, ly) \tens M(ly, dx) \tens J(x, y)} \\
  		&\leq \sup_{x \in A} M(z, dx) \tens J(x, y) = (d^* \hc J)(z, y),
  	\end{align*}
  	where $y \in B$ and $z \in M$. The `precisely when'"/part follows easily from evaluating $l^* = d^* \hc J$ at $(ly, y)$, where $y \in B$, and using that $k \leq M(ly, ly)$ by the unit axiom for $M$.
  \end{proof}
  
  Kan extensions satisfying the Beck"/Chevalley condition are preserved by any strict functor, as follows.
  \begin{proposition} \label{image of Kan extension satisfying Beck-Chevalley}
  	Let $\map F\K\mathcal L$ be a normal lax functor between thin equipments. Given morphisms $\map dAM$ and $\hmap JAB$ in $\K$ suppose that their left Kan extension $\map lBM$ exists, and that it satisfies the Beck"/Chevalley condition. The image $Fl$ is the left Kan extension of $Fd$ along $FJ$, and it satisfies the Beck"/Chevalley condition, precisely if $F(d^*) \hc FJ = F(d^* \hc J)$.
  \end{proposition} 
	\begin{proof}
		Since normal lax functors preserve companions we have
		\begin{displaymath}
			(Fd)^* \hc FJ = F(d^*) \hc FJ \leq F(d^* \hc J) = F(l^*) = (Fl)^*,
		\end{displaymath}
		where the second equality is the $F$"/image of the Beck"/Chevalley condition for $l$. The result follows directly from noticing that the Beck"/Chevalley condition for $Fl$ means that $(Fd)^* \hc FJ = (Fl)^*$.
	\end{proof}
  
  \section{\texorpdfstring{$T$"/}{T-}graphs} \label{T-graphs}
  Here we describe the second ingredient of our categorical approach to the maximum theorem: expressing topological structures as algebraic structures. More precisely, taking the view"/point of the study of monoidal topology \cite{Hofmann-Seal-Tholen14}, we will regard topological structures as `graphs' or `categories' over a monad on a thin equipment.
  
  \begin{definition}
  	A \emph{lax monad} $T$ on a thin equipment $\K$ is simply a monad $T = (T, \mu, \iota)$ on $\K$ in the $2$"/category $\lThinEq$, consisting of a lax endofunctor $\map T\K\K$ equipped with multiplication and unit transformations $\nat\mu{T^2}T$ and $\nat\iota{\id_\K}T$ that satisfy the usual associativity and unit axioms. We call $T$ \emph{normal} or \emph{strict} whenever its underlying endofunctor is normal or strict.
  \end{definition}
  Notice that any lax monad $T$ on a thin equipment $\K$ restricts to a monad $T_\textup v$ on the vertical category $\K_\textup v$. In particular, in the case $\K = \enRel\V$ for some quantale $\V$, the lax monad $T$ can be thought of as being a ``lax extension'' of the $\Set$"/monad $T_\textup v$ to the thin $2$"/category of $\V$"/relations. The latter is the traditional view"/point taken in monoidal topology; that such lax extensions are equivalent to lax monads on $\enRel\V$, in our sense, is shown in Section~III.1.13 of \cite{Hofmann-Seal-Tholen14}.
  
  For general constructions of lax extensions of $\Set$"/monads to $\V$"/relations we refer to \cite{Clementino-Hofmann04} (or see Section~IV.2.4 of \cite{Hofmann-Seal-Tholen14}) and \cite{Seal05}. Here we restrict ourselves to the extensions of the powerset monad and the ultrafilter monad, which are recalled in the examples below. 
  \begin{example} \label{powerset monad}
  	We denote by $PA = \set{S \subseteq A}$ the powerset of a set $A$. The assignment $A \mapsto PA$ extends to the \emph{powerset monad} on $\Set$ that is given by
		\begin{displaymath}
			\map{Pf}{PA}{PC}\colon S \mapsto fS; \quad\! \map{\mu_A}{P^2A}{PA}\colon \Sigma \mapsto \Union \Sigma; \quad\! \map{\iota_A}A{PA}\colon x \mapsto \set{x},
		\end{displaymath}
		where $\map fAC$ is any function. It was shown by Clementino and Hofmann (Section~6.3 of \cite{Clementino-Hofmann04}) that $P$ extends to a lax monad on $\enRel\V$, by mapping a $\V$"/relation $\hmap JAB$ to
		\begin{displaymath}
			(PJ)(S, T) = \inf_{t \in T}\sup_{s \in S} J(s, t),
		\end{displaymath}
		for any $S \in PA$ and $T \in PB$. In case $\V = 2$, so that we can regard $J$ and $PJ$ as ordinary relations, this reduces to
		\begin{displaymath}
			S (PJ) T \quad \iff \quad T \subseteq JS.
		\end{displaymath}
		Notice that $P$ is not normal.
  \end{example}
  
  The following example describes the lax extensions of the ultrafilter monad $U$. To be able to extend the ultrafilter monad $U$ to $\V$"/relations the quantale $\V$ needs to be `completely distributive', as follows. Writing $\Dn\V$ for the set of \emph{downsets} $S \subseteq \V$, satisfying
  \begin{displaymath}
  	u \leq v \quad \text{and} \quad v \in S \quad \implies \quad u \in S
  \end{displaymath}
  for all $u, v \in \V$, the quantale $\V$ is called \emph{completely distributive} if $\map{\sup}{\Dn\V}\V$ has a left adjoint $\Downarrow$. In that case let the \emph{totally below relation} $\ll$ on $\V$ be defined by $u \ll v \defeq u \in {\Downarrow} v$; equivalently
	\begin{displaymath}
		u \ll v \qquad \iff \qquad \Forall\displaylimits_{S \subseteq \V} \bigpars{v \leq \sup S \implies \exists s \in S\colon u \leq s}.
	\end{displaymath}
	Writing $\map{\downarrow}\V{\Dn\V}$ for the map that sends $v \in \V$ to the \emph{principal downset} ${\downarrow}v = \set{u \in \V \mid u \leq v}$, it follows from the chain of adjunctions ${\Downarrow} \ladj \sup \ladj {\downarrow}$ that $v = \sup\set{u \in \V \mid u \ll v}$ for all $v \in \V$; for details see e.g.\ Section~II.1.11 of \cite{Hofmann-Seal-Tholen14}.
	
The two"/chain quantale $\2 = \set{\bot \leq \top}$ is completely distributive, with the totally below relation given by $u \ll v \iff v = \top$, and so are the quantales $(\brks{0, \infty}, \geq)$ and $(\brks{-\infty, \infty}, \geq)$, both with $u \ll v \iff u > v$. That the quantales $\Delta_{\tn}$ of distance distribution functions (\exref{Delta}) are completely distributive is shown in Section~2.1 of \cite{Hofmann-Reis13}.
	\begin{example} \label{ultrafilter monad}
		For a set $A$ we denote by $UA$ the set of \emph{ultrafilters} on $A$; see e.g.\ Section~II.1.13 of \cite{Hofmann-Seal-Tholen14}. The assignment $A \mapsto UA$ extends to the \emph{ultrafilter monad} $U = (U, \mu, \iota)$ on $\Set$ defined by
		\begin{displaymath}
			T \in (Uf)(\mf x) \defeq \inv f T \in \mf x; \qquad S \in \mu_A(\mf X) \defeq S^\sharp \in \mf X; \qquad S \in \iota_A(x) \defeq x \in S,
		\end{displaymath}
		where $\map fAC$, $\mf x \in UA$, $T \subseteq C$, $\mf X \in U^2A$, $S \subseteq A$ and $x \in A$; here $S^\sharp$ is the set of all ultrafilters on $S \subseteq A$:
		\begin{displaymath}
			\mf x \in S^\sharp \quad \defeq \quad S \in \mf x.
		\end{displaymath}
		
		In Section~8 of \cite{Clementino-Tholen03} Clementino and Tholen show that $U$ extends to a lax monad on $\enRel\V$ provided that $\V$ is completely distributive, by mapping a $\V$"/relation $\hmap JAB$ to
		\begin{displaymath}
			(UJ)(\mf x, \mf y) = \inf_{\substack{S \in \mf x \\ T \in \mf y}} \sup_{\substack{s \in S \\ t \in T}} J(s, t),
		\end{displaymath}
		for all $\mf x \in UA$ and $\mf y \in UB$; see \cite{Lai-Tholen17a} for an alternative proof. In case $\V = 2$, so that we can regard $J$ and $UJ$ as ordinary relations, the definition of $UJ$ reduces to
		\begin{displaymath}
			\mf x (UJ) \mf y \quad \iff \quad \Forall_{S \in \mf x} JS \in \mf y \quad \iff \quad \Forall_{T \in \mf y} \rev JT \in \mf x,
		\end{displaymath}
		which recovers Barr's original extension of the ultrafilter monad \cite{Barr70}. Returning to general $\V$, it was shown in Section~6.4 of \cite{Clementino-Hofmann04} that $UJ$ can be equivalently given by
		\begin{equation} \label{equivalent definition UJ}
			(UJ)(\mf x, \mf y) = \sup \set{v \in \V \mid \mf x (UJ_v) \mf y}
		\end{equation}
		where $\hmap{J_v}AB$ is the (ordinary) relation defined by
		\begin{displaymath}
			x J_v y \quad \defeq \quad v \leq J(x,y)
		\end{displaymath}
		for all $x \in A$ and $y \in B$.
		
		It is easily checked that the above described extension $U$ of the ultrafilter monad to $\V$"/relations is normal. Moreover $U$ is a strict monad in the cases $\V = \2$ (see Sections~III.1.11"/12 of \cite{Hofmann-Seal-Tholen14}) and $\V = (\brks{0, \infty}, \geq)$ (see Proposition~III.2.4.3 of \cite{Hofmann-Seal-Tholen14}). In Section~6.4 of \cite{Clementino-Hofmann04} it is shown that $U$ is not strict when $\V = (\brks{-\infty, \infty}, \geq)$; unfortunately I do not know whether $U$ is a strict monad for any of the quantales $\Delta_{\tn}$ of distance distribution functions (\exref{Delta}).
	\end{example}
	
	Having described the main examples of monads $T$ on thin equipments, in the definition below we recall, from e.g.\ Sections~III.1.6 and III.4.1 of \cite{Hofmann-Seal-Tholen14}, the notions of `graph' and `category' over such a monad. The examples that follow describe how these notions allow us to regard topological structures as algebraic structures.
  \begin{definition} \label{graph}
  	Let $T = (T, \mu, \iota)$ be a lax monad on a thin equipment $\K$.
  	\begin{enumerate}[label=-]
  		\item A \emph{$T$"/graph} $A = (A, \alpha)$ consists of an object $A$ equipped with a horizontal morphism $\hmap \alpha{TA}A$, such that the \emph{unitor cell} on the left below exists.
  		\begin{displaymath}
  			\begin{tikzpicture}
  				\matrix(m)[math35]{A & A \\ TA & A \\};
  				\path[map]	(m-1-1) edge node[left] {$\iota_A$} (m-2-1)
  										(m-2-1) edge[barred] node[below] {$\alpha$} (m-2-2);
  				\path				(m-1-1) edge[eq] (m-1-2)
  										(m-1-2) edge[eq] (m-2-2);
  				\draw				($(m-1-1)!0.5!(m-2-2)$) node[rotate=-90] {$\leq$};	
  			\end{tikzpicture} \qquad \qquad \begin{tikzpicture}
  				\matrix(m)[math35]{TA & TA & A \\ TA & & A \\};
  				\path[map]	(m-1-1) edge[barred] node[above] {$T1_A$} (m-1-2)
  										(m-1-2) edge[barred] node[above] {$\alpha$} (m-1-3)
  										(m-2-1) edge[barred] node[below] {$\alpha$} (m-2-3);
  				\path				(m-1-3) edge[eq] (m-2-3)
  										(m-1-1) edge[eq] (m-2-1);
  				\draw				($(m-1-2)!0.5!(m-2-2)$) node[rotate=-90] {$\leq$};
  			\end{tikzpicture} \qquad \qquad \begin{tikzpicture}
  				\matrix(m)[math35]{T^2A & A \\ TA & A \\};
  				\path[map]	(m-1-1) edge[barred] node[above, yshift=3pt, xshift=-3pt] {$(T\alpha)(\id, \iota_A)$} (m-1-2)
  														edge node[left] {$\mu_A$} (m-2-1)
  										(m-2-1) edge[barred] node[below] {$\alpha$} (m-2-2);
  				\path				(m-1-2) edge[eq] (m-2-2);
  				\draw				($(m-1-1)!0.5!(m-2-2)$) node[rotate=-90] {$\leq$};
  			\end{tikzpicture}
  		\end{displaymath}
  		\item A $T$"/graph $A$ is called \emph{left unitary} if the middle cell above exists and \emph{right unitary} if the cell on the right exists; it is called \emph{unitary} if both cells exist.
  		\item A \emph{$T$"/category} is a $T$"/graph $A = (A, \alpha)$ such that the \emph{associator cell} on the left below exists.
  		\begin{displaymath}
  			\begin{tikzpicture}
  				\matrix(m)[math35]{T^2A & TA & A \\ TA & & A \\};
  				\path[map]	(m-1-1) edge[barred] node[above] {$T\alpha$} (m-1-2)
  														edge node[left] {$\mu_A$} (m-2-1)
  										(m-1-2) edge[barred] node[above] {$\alpha$} (m-1-3)
  										(m-2-1) edge[barred] node[below] {$\alpha$} (m-2-3);
  				\path				(m-1-3) edge[eq] (m-2-3);
  				\draw				($(m-1-2)!0.5!(m-2-2)$) node[rotate=-90] {$\leq$};
  			\end{tikzpicture} \qquad \qquad \qquad \begin{tikzpicture}
					\matrix(m)[math35]{TA & A \\ TC & C \\};
					\path[map]	(m-1-1) edge[barred] node[above] {$\alpha$} (m-1-2)
															edge node[left] {$Tf$} (m-2-1)
											(m-1-2) edge node[right] {$f$} (m-2-2)
											(m-2-1) edge[barred] node[below] {$\gamma$} (m-2-2);
					\draw				($(m-1-1)!0.5!(m-2-2)$) node[rotate=-90] {$\leq$};
				\end{tikzpicture}
  		\end{displaymath}
  		\item A morphism $\map fAC$ between $T$"/graphs is called a \emph{$T$"/morphism} (or again simply \emph{morphism}) if the cell on the right above exists.
  	\end{enumerate}
  \end{definition}
  Notice that any $T$"/category $A$ is a unitary $T$"/graph: the required cells are obtained by composing the associator cell of $A$ both with the ``$T$"/image'' of the unitor cell and with the horizontal cell $\iota^*_A \leq \alpha$ that corresponds to the unitor cell, under \lemref{horizontal cells}. Moreover any $T$"/graph is left unitary when $T$ is normal: in that case we have $T1_A \hc \alpha = 1_{TA} \hc \alpha = \alpha$.
  	
  Together with the morphisms between them $T$"/graphs form a category which we denote by $\Gph T$. Left unitary $T$"/graphs, unitary $T$"/graphs and $T$"/categories generate full subcategories
  \begin{equation} \label{subcategories}
  	\enCat T \into \UGph T \into \LGph T \into \Gph T.
  \end{equation}
  For a lax extension of a $\Set$"/monad $T$ to $\enRel\V$, where $\V$ is a quantale, $T$"/graphs are traditionally called \emph{$(T, \V)$"/graphs} and their morphisms \emph{$(T, \V)$"/functors}; in this case we shall write $\Gph{(T, \V)} \dfn \Gph T$ and $\UGph{(T, \V)} \dfn \UGph T$. Likewise categories over such monads $T$ are called \emph{$(T, \V)$"/categories}, and we write $\enCat{(T, \V)} \dfn \enCat T$.
  \begin{example} \label{closure space}
  	Let $\V$ be a quantale and let $P$ be the powerset monad extended to $\V$"/relations, see \exref{powerset monad}. A $(P, \V)$"/graph is a set $A$ equipped with a $\V$"/relation $\map\delta{PA \times A}\V$ satisfying the reflexivity axiom
  	\begin{enumerate}
  		\item[(R)] $k \leq \delta(\set x, x)$
  	\end{enumerate}
  	for all $x \in A$. It is easily checked that $A$ is unitary precisely if it is left unitary, which in turn is equivalent to the extensionality axiom
  	\begin{enumerate}
  		\item[(E)] $S \subseteq T \quad \implies \quad \delta(S, x) \leq \delta(T, x)$
  	\end{enumerate}
  	for all $S, T \in PA$.
  	
  	Seal proved in Section~5.4 of \cite{Seal05} that a $\V$"/relation $\hmap\delta{PA}A$ equips the set $A$ with a $(P, \V)$"/category structure precisely if it satisfies the two axioms above as well as the transitivity axiom
		\begin{enumerate}
			\item[(T)] $v \tens \delta(S^{(v)}, x) \leq \delta(S, x)$,
		\end{enumerate}
		for all $x \in A$, $S \in PA$ and $v \in \V$; here $S^{(v)} \dfn \set{y \in A \mid v \leq \delta(S, y)}$. Called \emph{closeness spaces} by Seal, we follow \cite{Lai-Tholen17a} and call $(P, \V)$"/categories \emph{$\V$"/valued closure spaces}; we write $\Cls\V \dfn \enCat{(P, \V)}$. Similarly we shall call $(P, \V)$"/graphs \emph{$\V$"/valued pseudoclosure spaces} and unitary $(P, \V)$"/graphs \emph{$\V$"/valued preclosure spaces}, and write $\PsCls\V \dfn \Gph{(P, \V)}$ and $\PreCls\V \dfn \UGph{(P, \V)}$. In each of these categories a morphism $\map f{(A, \delta)}{(C, \zeta)}$ is a \emph{continuous map}, satisfying
		\begin{enumerate}
			\item[(C)] $\delta(S, x) \leq \zeta(fS, fx)$
		\end{enumerate}
		for all $S \in PA$ and $x \in A$.
		
		Of course $\2$"/valued closure spaces $A$ can be identified with ordinary closure spaces $(A, S \mapsto \bar S)$ via $x \in \bar S \iff \delta(S, x) = \top$, while morphisms $\map fAC$ are \emph{continuous} in the usual sense: $f\bar S \subseteq \overline{fS}$ for all $S \in PA$. In Exercise~III.2.G of \cite{Hofmann-Seal-Tholen14} $\brks{0, \infty}$"/valued closure spaces are called \emph{metric closure spaces}.
  \end{example}
  
  \begin{example} \label{pseudotopological space}
  	Let $\V$ be a completely distributive quantale and let $U$ be the ultrafilter monad extended to $\V$"/relations, see \exref{ultrafilter monad}. A $(U, \V)$"/graph is a set $A$ equipped with a \emph{$\V$"/valued convergence relation} $\map\alpha{UA \times A}\V$, which is required to satisfy the reflexivity axiom
  	\begin{enumerate}
  		\item[(R)] $k \leq \alpha(\iota x, x)$
  	\end{enumerate}
  	for all $x \in A$; here $\iota x$ is the principal ultrafilter generated by $x$. In the case $\V = \2$ this recovers the notion of \emph{pseudotopological space} introduced by Choquet \cite{Choquet48}; we shall call $(U, \V)$"/graphs \emph{$\V$"/valued pseudotopological spaces} and write $\PsTop\V \dfn \Gph{(U, \V)}$. A morphism $\map f{(A, \alpha)}{(C, \gamma)}$ of $\V$"/valued pseudotopological spaces is a \emph{continuous map} satisfying
  	\begin{enumerate}
  		\item[(C)] $\alpha(\mf x, x) \leq \gamma\bigpars{(Uf)(\mf x), fx}$
  	\end{enumerate}
  	for all $\mf x \in UA$ and $x \in A$.
  \end{example}
  
  While unitary $(U, \V)$"/graphs and $(U, \V)$"/categories can be described directly in terms of ultrafilter convergence, we shall follow the approach taken by Lai and Tholen in \cite{Lai-Tholen17a} and instead describe them in terms of $\V$"/valued preclosure spaces and $\V$"/valued closure spaces respectively. These descriptions generalise the classical description of topological spaces in terms of closure operations. We will use this approach throughout: for instance in \secref{horizontal T-morphisms} we will describe ``horizontal $U$"/morphisms'' $\hmap JAB$ between $(U, \V)$"/categories in terms of the corresponding $\V$"/valued closure space structures on $A$ and $B$.
  
  The functor $\enCat{(U, V)} \to \Cls\V$ that allows us to regard $(U, \V)$"/categories as $\V$"/valued closure spaces is induced by an `algebraic morphism' $\hmap\eps PU$ between the powerset monad $P$ and the ultrafilter monad $U$, in the sense of Section~7 of \cite{Tholen16}, as follows. The first assertion of the proposition below is the first assertion of Proposition~3.4 of \cite{Lai-Tholen17a}.
  \begin{proposition} \label{algebraic morphism}
  	Let $\V$ be a completely distributive quantale and let $P$ and $U$ be the extensions of the powerset and ultrafilter monads to the thin equipment $\enRel\V$ of $\V$"/relations. The family of $\V$"/relations $\hmap{\eps_A}{PA}{UA}$, where $A$ ranges over all sets, that is defined by
  	\begin{displaymath}
  		\eps_A(S, \mf x) = \begin{cases}
  			k & \text{if $S \in \mf x$;} \\
  			\bot & \text{otherwise,}
  		\end{cases}
  	\end{displaymath}
  	for all $S \in PA$ and $\mf x \in UA$, forms an \emph{algebraic morphism} $\hmap\eps PU$. This means that the cells
  	\begin{displaymath}
  		\begin{tikzpicture}
  			\matrix(m)[math35]{PA & UA \\ PC & UC \\};
  			\path[map]	(m-1-1) edge[barred] node[above] {$\eps_A$} (m-1-2)
  													edge node[left] {$Pf$} (m-2-1)
  									(m-1-2) edge node[right] {$Uf$} (m-2-2)
  									(m-2-1) edge[barred] node[below] {$\eps_C$} (m-2-2);
  			\draw				($(m-1-1)!0.5!(m-2-2)$) node[rotate=-90] {$\leq$};
			\end{tikzpicture} \qquad \qquad \begin{tikzpicture}
  			\matrix(m)[math35]{A & A \\ PA & UA \\};
  			\path[map]	(m-1-1) edge node[left] {$\iota_A^P$} (m-2-1)
  									(m-1-2) edge node[right] {$\iota_A^U$} (m-2-2)
  									(m-2-1) edge[barred] node[below] {$\eps_A$} (m-2-2);
  			\path				(m-1-1) edge[eq] (m-1-2);
  			\draw				($(m-1-1)!0.5!(m-2-2)$) node[rotate=-90] {$\leq$};
			\end{tikzpicture} \qquad \qquad \begin{tikzpicture}
  				\matrix(m)[math35]{P^2A & PUA & U^2A \\ PA & & UA \\};
  				\path[map]	(m-1-1) edge[barred] node[above] {$P\eps_A$} (m-1-2)
  														edge node[left] {$\mu_A^P$} (m-2-1)
  										(m-1-2) edge[barred] node[above] {$\eps_{UA}$} (m-1-3)
  										(m-1-3) edge node[right] {$\mu_A^U$} (m-2-3)
  										(m-2-1) edge[barred] node[below] {$\eps_A$} (m-2-3);
  				\draw				($(m-1-2)!0.5!(m-2-2)$) node[rotate=-90] {$\leq$};
  			\end{tikzpicture}
  	\end{displaymath}
  	exist in $\enRel\V$, where $A$ is any set and $\map fAC$ is any function, while
  	\begin{displaymath}
  		PJ \hc \eps_B \leq \eps_A \hc UJ \qquad \text{and} \qquad P\eps_A \hc P\alpha = P(\eps_A \hc \alpha)
  	\end{displaymath}
  	for all $\V$"/relations $\hmap JAB$ and $\hmap\alpha{UA}A$.
  	
  	Furthermore the following identities hold:
		\begin{enumerate}[label=\textup{(\alph*)}]
			\item $\eps_A(\iota^P_A, \id) = \iota^U_{A*}$ for all $A$;
			\item $PJ \hc \eps_B = \eps_A \hc UJ$ for all $\hmap JAB$;
			\item $\eps_A \lhom (\eps_A \hc H) = H$ for all $\hmap H{UA}B$ with $(UH)(\id, \iota^U_B) \leq H(\mu_A^U, \id)$.
		\end{enumerate}
  \end{proposition}
  Since the proof is somewhat long and technical we defer it to the end of this section. The first assertion of the following proposition is the second assertion of Proposition~3.4 of \cite{Lai-Tholen17a}.
  \begin{proposition} \label{algebraic functor}
  	The assignment $(A, \hmap\alpha{UA}A) \mapsto (A, \eps_A \hc \alpha)$, where
  	\begin{displaymath}
  		\hmap{\eps_A \hc \alpha}{PA}A\colon(S, x) \mapsto \sup_{\mf x \in S^\sharp} \alpha(\mf x, x),
  	\end{displaymath}
  	extends to \emph{algebraic functors} $\map{\tilde \eps}{\enCat{(U, \V)}}{\Cls\V}$ and $\map{\tilde \eps}{\Gph{(U, \V)}}{\PreCls\V}$ that leave morphisms unchanged.
  \end{proposition}
  \begin{proof}[Sketch of the proof]
  	A routine calculation. To show that a $(U, \V)$"/graph $(A, \alpha)$ is mapped to a $\V$"/valued preclosure space, that is $\eps_A \hc \alpha$ forms a left unitary $(P, \V)$"/graph structure on $A$ (see \exref{closure space}), remember that the extension $U$ of the ultrafilter monad is normal, so that $P1_A \hc \eps_A = \eps_A$ by the previous proposition.
  \end{proof}
  
  We now follow \cite{Lai-Tholen17b} in calling a $\V$"/valued closure space $(A, \delta)$ (\exref{closure space}) a \emph{$\V$"/valued topological space} whenever its structure relation $\map\delta{PA \times A}\V$ preserves finite joins:
		\begin{equation} \label{preserve finite joins}
			\delta(\emptyset, x) = \bot \qquad \text{and} \qquad \delta(S \union T, x) = \sup\bigbrcs{\delta(S, x), \delta(T, x)}
		\end{equation}
		for all $x \in A$ and $S, T \in PA$. In particular $\2$"/valued topological spaces can be identified with topological spaces, while $\brks{0, \infty}$"/valued topological spaces coincide with Lowen's original \emph{approach spaces} \cite{Lowen97}, consisting of sets $A$ equipped with a \emph{point"/set distance} $\map\delta{PA \times A}{\brks{0, \infty}}$. Taking $\V = \Delta_{\tn}$, the quantale of distance distribution functions (\exref{Delta}), $\Delta_{\tn}$"/valued topological spaces are called \emph{$\tn$"/probabilistic approach spaces} in \cite{Lai-Tholen17b}.
		
		Writing $\vTop\V$ for the full subcategory of $\Cls\V$ generated by $\V$"/valued topological spaces, the main result of \cite{Lai-Tholen17a} is as follows.
  \begin{theorem}[Lai and Tholen] \label{V-valued topological space}
  	Let $\V$ be a completely distributive quantale. The algebraic functor $\map{\tilde \eps}{\enCat{(U, \V)}}{\Cls\V}$ embeds $\enCat{(U, \V)}$ into $\Cls\V$ as a full coreflective subcategory, which is precisely the category $\vTop\V$ of $\V$"/valued topological spaces. Its right adjoint $\map R{\Cls\V}{\enCat{(U, \V)}}$ is given on objects by $R(A, \delta) = (A, \eps_A \lhom \delta)$, where
  	\begin{displaymath}
  		\hmap{\eps_A \lhom \delta}{UA}A\colon (\mf x, x) \mapsto \inf_{S \in \mf x} \delta(S, x),
  	\end{displaymath}
  	while it leaves morphisms unchanged.
  \end{theorem}
  
  \begin{example} \label{ultrafilter convergence}
		Taking $\V = \2$ in the theorem above recovers Barr's presentation \cite{Barr70}
		\begin{displaymath}
			\enCat{(U, \2)} \iso \Top
		\end{displaymath}
		of topological spaces in terms of ultrafilter convergence. Instead of closure operations, in terms of topologies this isomorphism is induced by the correspondence of $(U, 2)$"/category structures $\hmap\alpha{UA}A$ and topologies $\tau$ on a set $A$, given by the inverse assignments $\alpha \mapsto \tau$ and $\tau \mapsto \alpha$ that are defined by
		\begin{displaymath}
			S \in \tau \quad \defeq \quad \Forall_{\mf x \alpha x} \pars{x \in S \implies S \in \mf x}
				\qquad \text{and} \qquad	\mf x \alpha x \quad \defeq \quad \Forall_{S \in \tau} \pars{x \in S \implies S \in \mf x},
		\end{displaymath}
		where $S \subseteq A$, $\mf x \in UA$ and $x \in A$.
		
		Taking $\V = \brks{0, \infty}$ in the theorem above recovers Clementino and Hofmann's presentation (Section~3.2 of \cite{Clementino-Hofmann03})
		\begin{displaymath}
			\enCat{(U, \brks{0, \infty})} \iso \App
		\end{displaymath}
		of Lowen's approach spaces in terms of metric ultrafilter convergence.
	\end{example}
	
	Turning to $(U, \V)$"/graphs, we follow \cite{Lai-Tholen17b} in calling a $\V$"/valued pseudoclosure space $A = (A, \delta)$ (\exref{closure space}) a \emph{$\V$"/valued pretopological space} whenever its structure relation $\map\delta{PA \times A}\V$ preserves finite joins; see \eqref{preserve finite joins}. Notice that this implies that $A$ is unitary, i.e.\ $A$ is a $\V$"/valued preclosure space (see \exref{closure space}). Choosing $\V = \2$ recovers the classical notion of \emph{pretopological space} \cite{Choquet48}; see Example~III.4.1.3(2) of \cite{Hofmann-Seal-Tholen14}.
	
	Writing $\PreTop\V$ for the subcategory of $\PreCls\V$ consisting of $\V$"/valued pretopological spaces, the following theorem is a variation on \thmref{V-valued topological space}.
	\begin{theorem} \label{V-valued pretopological space}
		For a completely distributive quantale $\V$ consider the restriction
		\begin{displaymath}
			\UGph{(U, \V)} \into \Gph{(U, \V)} \xrar{\tilde \eps}\PreCls\V,
		\end{displaymath}
		of the functor $\tilde \eps$ that is given in \propref{algebraic functor}. Again denoted by $\tilde \eps$, it embeds $\UGph{(U, \V)}$ into $\PreCls\V$ as a full coreflective subcategory, which is precisely the category $\PreTop\V$ of $\V$"/valued pretopological spaces. As in \thmref{V-valued topological space} the right adjoint $R$ to $\tilde\eps$ is given by $R(A, \delta) = (A, \eps_A \lhom \delta)$.
	\end{theorem}
	\begin{proof}
		That the composite $\tilde \eps$ maps into $\PreTop\V$ follows directly from the fact that, for an ultrafilter $\mf x$ on a set $A$, we have $\emptyset \notin \mf x$ while $S \union T \in \mf x$ precisely if $S \in \mf x$ or $T \in \mf x$. We start by checking that the assignment $RA = (A, \eps_A \lhom \delta)$, for $\V$"/valued preclosure spaces $A = (A, \delta)$, induces a functor $\map R{\PreCls\V}{\UGph{(U, \V)}}$.
		
		Writing
		\begin{displaymath}
			\hmap{\alpha \dfn \eps_A \lhom \delta}{UA}A\colon (\mf x, x) \mapsto \inf_{S \in \mf x}\delta(S, x),
		\end{displaymath}
		we have to show that $RA = (A, \alpha)$ is a unitary $(U, \V)$"/graph. That $\alpha$ is reflexive, that is $\alpha(\iota^U_A x, x) = \inf_{S \in \iota^U_A x} \delta(S, x) \geq k$ for all $x \in A$, follows easily from the fact that $A$ is reflexive and unitary; hence $RA$ is a $(U, \V)$"/graph. That $RA$ is left unitary is immediate from the fact that the ultrafilter monad $U$ is normal, so that only right unitariness remains: we have to show that
		\begin{displaymath}
			(U\alpha)\bigpars{\mf X, \iota_A^U(x)} \leq \alpha\bigpars{\mu_A^U(\mf X), x} = \inf_{T \in \mu_A^U(\mf X)} \delta(T, x)
		\end{displaymath}
		for all $\mf X \in U^2A$ and $x \in A$. To see this notice that for every $T \in \mu_A^U(\mf X)$, that is $T^\sharp \in \mf X$, we have
		\begin{multline*}
			(U\alpha)\bigpars{\mf X, \iota_A^U(x)} = \inf_{\substack{\sigma \in \mf X \\ R \in \iota_A^U(x)}} \sup_{\substack{\mf y \in \sigma \\ r \in R}}\inf_{P \in \mf y} \delta(P, r) \\
			\leq \sup_{\substack{\mf y \in T^\sharp \\ r \in \set x}} \inf_{P \in \mf y} \delta(P, r) = \sup_{\mf y \in T^\sharp} \inf_{P \in \mf y} \delta(P, x) \leq \delta(T, x).
		\end{multline*}
		
		To see that $A \mapsto RA$ extends to morphisms consider a continuous map $\map fAC$ between $\V$"/valued preclosure spaces $A = (A, \delta)$ and $C = (C, \zeta)$. Then
		\begin{displaymath}
			\alpha(\mf x, x) = \inf_{S \in \mf x} \delta(S, x) \leq \inf_{S \in \mf x} \zeta(fS, fx) = \inf_{T \in (Uf)(\mf x)} \zeta(T, fx) = (\eps_C \lhom \zeta)\bigpars{(Uf)(\mf x), fx}
		\end{displaymath}
		for all $\mf x \in UA$ and $x \in A$, showing that $f$ forms a morphism $RA \to RC$ of unitary $(U, \V)$"/graphs.
		
		The arguments proving that $R$ is a right adjoint of $\tilde \eps$, as well as that $(\tilde \eps \of R)(A) = A$ for any $\V$"/valued pretopological space $A$, are identical to the ones given in the proof of Theorem~3.6 of \cite{Lai-Tholen17a} (whose statement is reproduced above as \thmref{V-valued topological space}). To complete the proof it thus suffices to prove that $(R \of \tilde \eps)(A) = A$ for any unitary $(U, \V)$"/graph $A = (A, \alpha)$. But this follows immediately from \propref{algebraic morphism}(c).
	\end{proof}

	We close this section with the proof of \propref{algebraic morphism}.
	\begin{proof}[Proof of \propref{algebraic morphism}]
		That the family of $\V$"/relations $\hmap{\eps_A}{PA}{UA}$, whose definition is recalled below, forms an algebraic morphism $\hmap\eps PU$ is proved in Proposition~3.4 of \cite{Lai-Tholen17a}. Thus it remains to prove parts (a), (b) and (c).
		\begin{displaymath}
  		\eps_A(S, \mf x) = \begin{cases}
  			k & \text{if $S \in \mf x$;} \\
  			\bot & \text{otherwise}
  		\end{cases}
  	\end{displaymath}
		
		Part (a): $\eps_A(\iota^P_A, \id) = \iota^U_{A*}$ for all $A$. This is a direct consequence of the fact that, for any $x \in A$ and ultrafilter $\mf x$ on $A$, one has $\set x \in \mf x$ if and only if $\mf x = \iota^U_A(x)$, the principal ultrafilter on $x$.
		
		Part (b): $PJ \hc \eps_B = \eps_A \hc UJ$ for all $\hmap JAB$. The inequality $\leq$ is a consequence of $\hmap\eps PU$ being an algebraic morphism; we claim that the converse inequality holds as well. Indeed for any $R \in PA$ and $\mf y \in UB$ we have
		\begin{align*}
			(\eps_A \hc UJ)(R, \mf y) &= \sup_{\mf x \in UA} \Bigpars{\eps_A(R, \mf x) \tens \inf_{\substack{S \in \mf x \\ T \in \mf y}} \sup_{\substack{s \in S \\ t \in T}} J(s, t)} = \sup_{\mf x \in R^\sharp} \inf_{\substack{S \in \mf x \\ T \in \mf y}} \sup_{\substack{s \in S \\ t \in T}} J(s, t) \\
			&\leq \inf_{T \in \mf y} \sup_{\substack{s \in R \\ t \in T}} J(s, t) \overset{\text{(i)}}= \sup_{T \in \mf y} \inf_{t \in T} \sup_{s \in R} J(s, t) \\
			&= \sup_{T \in PB} \Bigpars{\bigpars{\inf_{t \in T} \sup_{s \in R} J(s, t)} \tens \eps_B(T, \mf y)} = (PJ \hc \eps_B)(R, \mf y),
		\end{align*}
		where the equality denoted (i) follows from the lemma below. We conclude that $PJ \hc \eps_B = \eps_A \hc UJ$ for all $\hmap JAB$.
		
		Part (c): $\eps_A \lhom (\eps_A \hc H) = H$ for all $\hmap H{UA}B$ with $(UH)(\id, \iota^U_B) \leq H(\mu_A^U, \id)$. In Theorem~3.6 of \cite{Lai-Tholen17a} this was proved in case that $\hmap{H = \alpha}{UA}A$ is the $\V$"/valued convergence relation of a $(U, \V)$"/category $A$. We will modify the proof given there slightly so that it generalises to any $\V$"/relation $\hmap H{UA}B$ satisfying $(UH)(\id, \iota^U_B) \leq H(\mu_A^U, \id)$. First notice that $H \leq \eps_A \lhom (\eps_A \hc H)$ follows from the adjunction $\eps_A \hc \dash \ladj \eps_A \lhom \dash$. For the reverse inequality let $\mf x \in UA$ and $y \in B$; to prove that
		\begin{displaymath}
			\bigpars{\eps_A \lhom (\eps_A \hc H)}(\mf x, y) = \inf_{S \in \mf x} \sup_{\mf y \in S^\sharp} H(\mf y, y) \leq H(\mf x, y)
		\end{displaymath}
		it suffices to show that $v \ll \inf_{S \in \mf x} \sup_{\mf y \in S^\sharp} H(\mf y, y)$ implies $v \leq H(\mf x, y)$ for all $v \in \V$. The hypothesis here means that for all $S \in \mf x$ there is $\mf y \in S^\sharp$ with $v \leq H(\mf y, y)$. As a consequence the sets
		\begin{displaymath}
			X_S = \set{\mf y \in S^\sharp \mid v \leq H(\mf y, y)},
		\end{displaymath}
		where $S$ ranges over $\mf x$, form a proper filter base and we can choose an ultrafilter $\mf X \in U^2A$ containing all of them. As $S^\sharp \supseteq X_S \in \mf X$ for all $S \in \mf x$ it follows that $\mu_A^U(\mf X) = \mf x$. Moreover
		\begin{displaymath}
			(UH)\bigpars{\mf X, \iota^U_B(y)} = \inf_{\substack{R \in \mf X\\T \in \iota_B^U(y)}} \sup_{\substack{\mf y \in R\\t \in T}} H(\mf y, t) = \inf_{R \in \mf X} \sup_{\mf y \in R} H(\mf y, y) \geq \inf_{R \in \mf X} v = v
		\end{displaymath}
		where the inequality follows from the fact that every $R \in \mf X$ intersects $X_A$. We conclude that
		\begin{displaymath}
			H(\mf x, y) = H\bigpars{\mu_A^U(\mf X), y} \geq (UH)\bigpars{\mf X, \iota_B^U(y)} \geq v
		\end{displaymath}
		where the first inequality follows from the assumption on $H$. This completes the proof.
	\end{proof}
	The following lemma was used in the proof above; it is a straightforward generalisation of Proposition~1.8.29 of \cite{Lowen97}.
	\begin{lemma} \label{minimax}
		Let $\V$ be a completely distributive lattice, and let $\map fA\V$ be any function. If $\mf x$ is an ultrafilter on $A$ then
		\begin{displaymath}
			\sup_{S \in \mf x} \inf_{s \in S} fs = \inf_{S \in \mf x} \sup_{s \in S} fs.
		\end{displaymath}
	\end{lemma}
	\begin{proof}
		The inequality $\leq$ follows directly from the fact that $\inf_{s \in S} fs \leq \sup_{t \in T} ft$ for all $S, T \in \mf x$. To show $\geq$ assume that $v \ll \inf_{S \in \mf x} \sup_{s \in S} fs$ for some $v \in \V$. Hence for all $S \in \mf x$ there is $s \in S$ with $v \leq fs$. Writing $T \dfn \set{t \in A \mid v \leq ft}$ we thus have $S \isect T \neq \emptyset$ for all $S \in \mf x$. We conclude $T \in \mf x$, from which $v \leq \sup_{S \in \mf x} \inf_{s \in S} fs$ follows.
	\end{proof}

	\section{Modular \texorpdfstring{$T$}{T}"/graphs} \label{modular T-graphs}
	Let $\V$ be a quantale and let $T$ be a lax monad on the thin equipment $\enRel\V$ of $\V$"/relations. In this section we consider $\V$"/categories equipped with compatible $(T, \V)$"/graph structures as follows. Applying the $2$"/functor $\Mod$ (\propref{2-functor Mod}) to $T$ we obtain a normal lax monad $\Mod T$ on the thin equipment $\enProf\V = \Mod(\enRel\V)$ of $\V$"/profunctors between $\V$"/categories (\exref{V-profunctors}). Following Section~4 of \cite{Tholen09} we call $(\Mod T)$"/graphs \emph{modular} $(T,\V)$"/graphs and write $\ModGph{(T, \V)} \dfn \Gph{(\Mod T)}$. A modular $(T, \V)$"/graph $A$ is a $\V$"/category $A = (A, \bar A)$ equipped with a $T$"/graph structure $\hmap\alpha{TA}A$ that is a $\V$"/profunctor, see the lemma below. A morphism of modular $(T, \V)$"/graphs $\map fAC$ is simultaneously a $\V$"/functor $\map f{(A, \bar A)}{(C, \bar C)}$ of $\V$"/categories as well as a morphism $\map f{(A, \alpha)}{(C, \gamma)}$ of $(T, \V)$"/graphs. As in the non"/modular case \eqref{subcategories} we have subcategories
	\begin{displaymath}
		\ModCat{(T, \V)} \into \ModUGph{(T, \V)} \into \ModRGph{(T, \V)} \into \ModGph{(T, \V)}
	\end{displaymath}
	consisting of modular $(T, \V)$"/categories, modular unitary $(T, \V)$"/graphs and modular right unitary ones.
  
  Generalising to lax monads $T$ on an arbitrary thin equipment $\K$, we shall call $(\Mod T)$"/graphs and $(\Mod T)$"/categories \emph{modular} $T$"/graphs and \emph{modular} $T$"/categories respectively, while we write $\ModGph T \dfn \Gph{(\Mod T)}$ and $\ModCat T \dfn \enCat{(\Mod T)}$. It is shown in \cite{Tholen09} that $\enCat T$ forms a full reflective subcategory of $\ModCat T$, via the embedding
  \begin{equation} \label{normalised T-category}
  	\map N{\enCat T}{\ModCat T}\colon (A, \alpha) \mapsto \bigpars{A, \alpha(\iota_A, \id), \alpha}
  \end{equation}
  whose left adjoint is the forgetful functor. We will follow \cite{Cruttwell-Shulman10} in calling a modular $T$"/category $A = (A, \bar A, \alpha)$ \emph{normalised} if it lies in the image of $N$, that is $\bar A = \alpha(\iota_A, \id)$.
  
  Before describing modular $\V$"/valued (pre"/)closure spaces we state a couple of lemmas that describe relations between monoid structures and $T$"/graph structures on a single object.
  \begin{lemma} \label{modular T-graph structure}
  	Let $T$ be a lax monad on a thin equipment $\K$ and let $A = (A, \bar A)$ be a monoid in $\K$. A bimodule $\hmap\alpha{TA}A$ (\defref{monoid}) forms a modular $T$"/graph structure on $A$ in $\Mod(\K)$ precisely if its underlying horizontal morphism forms a $T$"/graph structure on $A$ in $\K$. In that case
  	\begin{enumerate}[label=\textup{(\alph*)}]
  		\item both the modular $T$"/graph $(A, \bar A, \alpha)$ and the $T$"/graph $(A, \alpha)$ are left unitary;
  		\item $(A, \bar A, \alpha)$ is right unitary precisely if $(A, \alpha)$ is;
  		\item $(A, \bar A, \alpha)$ is a modular $T$"/category precisely if $(A, \alpha)$ is a $T$"/category.
  	\end{enumerate}
  \end{lemma}
  \begin{proof}
  	The main assertion states that the existences of the unitor cells (see \defref{graph}) for $(A, \bar A, \alpha)$ and $(A, \alpha)$ are equivalent, which amounts to proving that
  	\begin{displaymath}
  		\bar A \leq \alpha(\iota_A, \id) \qquad \iff \qquad 1_A \leq \alpha(\iota_A, \id).
  	\end{displaymath} 
  	The implication $\implies$ follows from the unit axiom $1_A \leq \bar A$ for monoids (see \defref{monoid}), while $\Leftarrow$ is shown by
  	\begin{displaymath}
  		\bar A = 1_A \hc \bar A \leq \alpha(\iota_A, \id) \hc \bar A = \iota_{A*} \hc \alpha \hc \bar A \leq \iota_{A*} \hc \alpha = \alpha(\iota_A, \id)
  	\end{displaymath}
  	where the second inequality follows from the fact that $\alpha$ is a bimodule.
  	
  	For part (a) notice that $(A, \bar A, \alpha)$ is left unitary because $\Mod T$ is normal. That $(A, \alpha)$ is left unitary is shown by
  	\begin{displaymath}
  		T1_A \hc \alpha \leq T\bar A \hc \alpha \leq \alpha.
  	\end{displaymath}
  	
  	Part (b) is a direct consequence of the fact that the cells expressing right unitarity for $(A, \bar A, \alpha)$ and $(A, \alpha)$ respectively coincide. The same holds for the cells expressing associativity, so that part (c) follows too.
  \end{proof}
  
  In the setting of a $\Set$"/monad $T$ laxly extended to $\V$"/relations the following was proved as Lemma~1 of \cite{Tholen09}. The proof given there, which uses a result analogous to \propref{lax functors preserve restrictions}, applies verbatim to our setting.
  \begin{lemma}[Tholen] \label{compatibility of modular structure}
  	Let $T$ be a lax monad on a thin equipment $\K$. Consider both a monoid structure $\hmap{\bar A}AA$ and a $T$"/category structure $\hmap\alpha{TA}A$ on a single object $A$. The triple $(A, \bar A, \alpha)$ forms a modular $T$"/category precisely if $\bar A \leq \alpha(\iota_A, \id)$.
  \end{lemma}
  
  \begin{example} \label{modular closure space}
  	Recall from \exref{closure space} that the notions of left unitary $\V$"/valued pseudoclosure space and $\V$"/valued preclosure space coincide. Hence, by \lemref{modular T-graph structure}, the notions of modular $\V$"/valued pseudoclosure space and \emph{modular $\V$"/valued preclosure space} coincide as well: both consist of a $\V$"/category $A = (A, \bar A)$ equipped with a $\V$"/relation $\map\delta{PA \times A}\V$ that satisfies reflexivity and modularity axioms:
  	\begin{enumerate}
  		\item[(R)] $k \leq \delta(\set x, x)$;
  		\item[(M)] $\bigpars{\inf_{t \in T}\sup_{s \in S} A(s, t)} \tens \delta(T, x) \tens A(x, y) \leq \delta(S, y)$,
  	\end{enumerate}
  	for all $x, y \in A$ and $S, T \in PA$. We write $\ModPreCls\V \dfn \ModGph{(P, \V)}$ for the category of modular $\V$"/valued preclosure spaces, whose morphisms are both $\V$"/functors and continuous maps.
  	
  	$\ModPreCls\V$ contains as a full subcategory the category $\ModCls\V$ of \emph{modular $\V$"/valued closure spaces} $A = (A, \bar A, \delta)$, with $\delta$ a $\V$"/valued closure space structure on $A$. In this case, by \lemref{compatibility of modular structure}, the modularity axiom above reduces to
  	\begin{enumerate}
  		\item[(M')] $A(x, y) \leq \delta(\set x, y)$
  	\end{enumerate}
  	for all $x, y \in A$. Under the embedding \eqref{normalised T-category} any $\V$"/valued closure space $A = (A, \delta)$ gives rise to a normalised modular $\V$"/valued closure space $NA$, whose $\V$"/category structure is given by $A(x, y) \dfn \delta(\set x, y)$.
  	
  Taking $\V = \2$ in the above we obtain the notion of a \emph{modular preclosure space}: a preclosure space $A$ equipped with an ordering $\leq$ satisfying $\upset \overline{\upset S} \subseteq \bar S$ for all $S \subseteq A$, where $\upset S = \set{x \in A \mid \exists s \in S\colon s \leq x}$ is the upset generated by $S$. For a \emph{modular closure space} $A$ the latter condition reduces to $x \leq y$ implies $y \in \overline{\set x}$ for all $x, y \in A$. Any closure space $A$ can be regarded as a normalised modular closure space by equipping it with the \emph{specialisation order}: $x \leq y \defeq y \in \overline{\set x}$, that is $\overline{\set x} = \upset x$ for all $x \in A$.
	\end{example}
	
	\begin{example}	\label{modular (U, V)-category}
		A modular $(U, \V)$"/graph is a $\V$"/category $A = (A, \bar A)$ equipped with a $\V$"/valued pseudotopological space structure $\hmap\alpha{UA}A$ (\exref{pseudotopological space}) satisfying the modularity axiom
		\begin{enumerate}
			\item[(M)] $\bigpars{\inf_{\substack{S \in \mf x \\ T \in \mf y}}\sup_{\substack{s \in S \\ t \in T}} A(s, t)} \tens \alpha(\mf y, x) \tens A(x, y) \leq \alpha(\mf x, y)$
		\end{enumerate}
		for all $\mf x, \mf y \in UA$ and $x, y \in A$. We will call $A = (A, \bar A, \alpha)$ a \emph{modular $\V$"/valued pseudotopological space}. If the $\V$"/valued convergence relation $\alpha$ corresponds to a $\V$"/valued topological space structure under \thmref{V-valued topological space}, so that $(A, \alpha)$ is a modular $(U, \V)$"/category, then by \lemref{compatibility of modular structure} the modularity axiom above reduces to
		\begin{enumerate}
			\item[(M')] $A(x, y) \leq \alpha(\iota x, y)$
		\end{enumerate}
		for all $x, y \in A$, where $\iota x$ is the principal ultrafilter on $x$. \thmref{modular V-valued topological space} below shows that the correspondence of \thmref{V-valued topological space} restricts to one between modular $(U, \V)$"/categories and \emph{modular $\V$"/valued topological spaces}, by which we mean modular $\V$"/valued closure spaces $A = (A, \bar A, \delta)$ whose structure $\V$"/relations $\delta$ preserve finite joins \eqref{preserve finite joins}. \emph{Modular $\V$"/valued pretopological spaces} are defined analogously; that they correspond to modular unitary $(U, \V)$"/graphs is proved by \thmref{modular V-valued topological space} as well. As with $\V$"/valued closure spaces, any $\V$"/valued topological space $A = (A, \delta)$ induces a normalised modular $\V$"/valued topological space $NA$ with $\V$"/category structure $A(x, y) = \delta(\set x, y)$.
		
		In particular, by taking $\V = \2$ in the above we recover the notion of \emph{modular topological space} (Section~4 of \cite{Tholen09}): a topological space $A$, with ultrafilter convergence $\alpha$, equipped with an ordering $\leq$ that is contained in the specialisation order of $\alpha$, i.e.\ $x \leq y$ implies $(\iota x)\alpha y$ for all $x, y \in A$.
		
		Taking the Lawvere quantale $\V = \brks{0, \infty}$ instead, a \emph{modular approach space} $A$ (Section~6 of \cite{Tholen09}) is a generalised metric space $A = (A, \bar A)$ equipped with a point"/set distance $\map\delta{PA \times A}{\brks{0, \infty}}$ such that $\delta\bigpars{\set x, y} \leq A(x, y)$ for all $x, y \in A$.
	\end{example}
	
	We denote by $\ModTop\V$ the category of modular $\V$"/valued topological spaces (as defined in the example above), which forms a full subcategory of $\ModCls\V$. Likewise $\ModPreTop\V$ denotes the full subcategory of $\ModPreCls\V$ that consists of modular $\V$"/valued pretopological spaces.
	\begin{theorem} \label{modular V-valued topological space}
		Let $\V$ be a completely distributive quantale. The pair of adjunctions $\tilde \eps \ladj R$ described in \thmref{V-valued topological space} and \thmref{V-valued pretopological space} lift as shown in the diagrams below, where $N$ is as defined in \eqref{normalised T-category} and where forgetful functors are denoted by $U$. Except for the composites $N \of R$ and $\bar R \of N$, any two parallel composites between opposite corners coincide.
		\begin{displaymath}
			\begin{tikzpicture}[textbaseline]
				\matrix(m)[math35, row sep={4em,between origins}, column sep={9em,between origins}]{\ModCat{(U, \V)} & \ModCls\V \\ \enCat{(U, \V)} & \Cls\V \\};
				\path[map]	(m-1-1) edge[transform canvas={yshift=3pt}] node[above] {$\bar{\tilde\eps}$} (m-1-2)
														edge[transform canvas={xshift=-3pt}] node[left] {$U$} (m-2-1)
										(m-1-2)	edge[transform canvas={yshift=-3pt}] node[below] {$\bar R$} (m-1-1)
														edge[transform canvas={xshift=-3pt}] node[left] {$U$} (m-2-2)
										(m-2-1)	edge[transform canvas={xshift=3pt}] node[right] {$N$} (m-1-1)
														edge[transform canvas={yshift=3pt}] node[above] {$\tilde\eps$} (m-2-2)
										(m-2-2) edge[transform canvas={xshift=3pt}] node[right] {$N$} (m-1-2)
														edge[transform canvas={yshift=-3pt}] node[below] {$R$} (m-2-1);
				\path[white]	(m-1-1) edge node[black, yshift=0.1pt, scale=0.6] {$\bot$} (m-1-2)
															edge node[black, scale=0.6] {$\ladj$} (m-2-1)
											(m-1-2) edge node[black, scale=0.6] {$\ladj$} (m-2-2)
											(m-2-1) edge node[black, yshift=0.1pt, scale=0.6] {$\bot$} (m-2-2);
			\end{tikzpicture} \qquad \begin{tikzpicture}[textbaseline]
				\matrix(m)[math35, row sep={4em,between origins}, column sep={10em,between origins}]{\ModUGph{(U, \V)} & \ModPreCls\V \\ \UGph{(U, \V)} & \PreCls\V \\};
				\path[map]	(m-1-1) edge[transform canvas={yshift=3pt}] node[above] {$\bar{\tilde\eps}$} (m-1-2)
														edge node[left] {$U$} (m-2-1)
										(m-1-2)	edge[transform canvas={yshift=-3pt}] node[below] {$\bar R$} (m-1-1)
														edge node[right] {$U$} (m-2-2)
										(m-2-1)	edge[transform canvas={yshift=3pt}] node[above] {$\tilde\eps$} (m-2-2)
										(m-2-2) edge[transform canvas={yshift=-3pt}] node[below] {$R$} (m-2-1);
				\path[white]	(m-1-1) edge node[black, yshift=0.1pt, scale=0.6] {$\bot$} (m-1-2)
											(m-2-1) edge node[black, yshift=0.1pt, scale=0.6] {$\bot$} (m-2-2);
			\end{tikzpicture}
		\end{displaymath}
		Leaving $\V$"/category structures unchanged, the lifts above establish isomorphisms of categories
		\begin{displaymath}
			\ModCat{(U, \V)} \iso \ModTop\V \qquad \text{and} \qquad \ModUGph{(U, \V)} \iso \ModPreTop\V.
		\end{displaymath}
	\end{theorem}
	\begin{proof}
		Remember that the functors $\tilde\eps$ and $R$ leave underlying sets $A$ unchanged while they act on structure $\V$"/relations $\hmap\alpha{UA}A$ and $\hmap\delta{PA}A$ by $\alpha \mapsto \eps_A \hc \alpha$ and $\delta \mapsto \eps_A \lhom \delta$ respectively. Since the lifts of $\tilde\eps$ and $R$ leave $\V$"/category structures unchanged, it suffices to check that the latter assignments preserve modularity with respect to any given $\V$"/category structure $\hmap{\bar A}AA$ on $A$. That the first assignment does follows easily from \propref{algebraic morphism}(b), when applied to $J = \bar A$. To see that the second does too let $\hmap\delta{PA}A$ be any $\V$"/valued preclosure space structure on $A$ that is modular with respect to $\bar A$, i.e.\ $P\bar A \hc \delta \hc \bar A \leq \delta$. Then, writing $\eps = \eps_A$,
		\begin{multline*}
			U\bar A \hc (\eps \lhom \delta) \hc \bar A \leq \eps \lhom \bigpars{\eps \hc U\bar A \hc (\eps \lhom \delta) \hc \bar A} \\
			= \eps \lhom \bigpars{P\bar A \hc \eps \hc (\eps \lhom \delta) \hc \bar A} \leq \eps \lhom (P\bar A \hc \delta \hc \bar A) \leq \eps \lhom \delta
		\end{multline*}
		where the first two inequalities are given by the unit and counit of the adjunction $\eps \hc \dash \ladj \eps \lhom \dash$ while the equality follows from \propref{algebraic morphism}(b). We conclude that $\eps_A \lhom \delta$ is again modular.
		
		It remains to show the commutativity of the diagrams. It is clear that any two parallel composites containing $U$ coincide, leaving us to prove that
		\begin{displaymath}
			(N \of \tilde\eps)(A, \alpha) = (\bar{\tilde\eps} \of N)(A, \alpha)
		\end{displaymath}
		in the left"/hand diagram, for any $(U, \V)$"/category structure $\alpha$ on $A$. For this it suffices to check that the $\V$"/category structures coincide, which is shown by
		\begin{displaymath}
			(\eps \hc \alpha)(\iota_A^P, \id) = \iota_{A*}^P \hc \eps \hc \alpha = \iota_{A*}^U \hc \alpha = \alpha(\iota_A^U, \id),
		\end{displaymath}
		where the second equality follows from \propref{algebraic morphism}(a).
	\end{proof}

	\section{\texorpdfstring{$T$}{T}"/cocomplete \texorpdfstring{$T$}{T}"/graphs} \label{T-complete T-graphs}
	Let $T$ be a lax monad on a thin equipment $\K$. To be able to generalise the maximum theorem to Kan extensions $\map lBM$ between $T$"/graphs, we need either the Kan extension $l$ itself or its target $M$ to be `well"/behaved'. In \secref{Kan extension section} well"/behaved Kan extensions were described: they are the ones that satisfy the Beck"/Chevalley condition. By well"/behaved $T$"/graphs we mean `$T$"/cocomplete' ones as follows.
	\begin{definition} \label{complete}
		Let $T$ be a lax monad on a thin equipment $\K$. A $T$"/graph $A = (A, \alpha)$ is called \emph{$T$"/cocomplete} if $\alpha = a_*$ in $\K$ for some morphism $\map a{TA}A$.
	\end{definition}
	Applying the above to the induced lax monad $T = \Mod(T)$ on $\Mod(\K)$, by the lemma below a modular $T$"/graph $A = (A, \bar A, \alpha)$ is $T$"/cocomplete whenever $\alpha = \bar A(a, \id)$ in $\K$, for some morphism $\map a{TA}A$. We will see in \exref{Lawvere quantale as a modular approach space} below that the isomorphisms of \thmref{modular V-valued topological space} fail to preserve $T$"/cocompleteness in general.
	
	The closely related but different notion of `$T$"/cocompleteness' for $(T, \V)$"/categories $A = (A, \alpha)$, that is considered in Section~III.5.4 of \cite{Hofmann-Seal-Tholen14}, can be rephrased in terms of the above as follows: $A$ is `$T$"/cocomplete' whenever its corresponding normalised modular $(T, \V)$"/category $NA = \bigpars{A, \alpha(\iota, \id), \alpha}$ is $T$"/cocomplete in our sense, that is $\alpha = \alpha(\iota \of a, \id)$ for some function $\map a{TA}A$.
	\begin{lemma}
		Let $T$ be a lax monad on a thin equipment $\K$ and $A = (A, \bar A, \alpha)$ a modular $T$"/graph. If $\alpha = \bar A (a, \id)$ in $\K$ for some morphism $\map a{TA}A$ then $a$ is a homomorphism of monoids $(TA, T\bar A) \to (A, \bar A)$ so that $\alpha = a_*$ in $\Mod \K$.
	\end{lemma}
	\begin{proof}
		By \lemref{horizontal cells} the existence of the structure cell exhibiting $a$ as a homomorphism of monoids can be deduced from
		\begin{displaymath}
			T\bar A \hc a_* = T\bar A \hc a_* \hc 1_A \leq T\bar A \hc a_* \hc \bar A = T\bar A \hc \alpha \leq \alpha = \bar A(a, \id),
		\end{displaymath}
		where we use the unit axiom for $\bar A$ and the fact that $\alpha$ is a bimodule.
	\end{proof}
	
	\begin{example} \label{complete modular closure space}
		Let $P$ be the powerset monad extended to $\V$"/relations taking values in a quantale $\V$. A modular $\V$"/valued preclosure space $A = (A, \bar A, \delta)$ (see \exref{modular closure space}) is $P$"/cocomplete whenever for each subset $S \in PA$ there exists a tacitly chosen \emph{generic point} $x_0 \in A$ such that
		\begin{displaymath}
			\delta (S, y) = A(x_0, y).
		\end{displaymath}
		for all $y \in A$.
		
		For a modular preclosure space $A = (A, \leq, S \mapsto \bar S)$ (\exref{modular closure space}) the above reduces to $\bar S = \upset x_0$, that is all closed subsets of $A$ are principal upsets. For the normalised modular closure space $NA$ induced by a closure space $A$ the latter means $\bar S = \overline{\set{x_0}}$, that is every closed subset of $A$ contains a generic point.
	\end{example}
	
	\begin{example} \label{complete modular approach space}
		Let $U$ be the ultrafilter monad extended to $\V$"/relations taking values in a completely distributive quantale $\V$. A modular $\V$"/valued pretopological space $A$, regarded as a $\V$"/category $(A, \bar A)$ equipped with a convergence $\V$"/profunctor \mbox{$\map\alpha{UA \times A}\V$} (see \exref{modular (U, V)-category}), is $U$"/cocomplete whenever for each $\mf x \in UA$ there exists a tacitly chosen \emph{generic point} $x_0 \in A$ such that
		\begin{displaymath}
			\alpha(\mf x, y) = A(x_0, y)
		\end{displaymath}
		for all $y \in A$. In Section~III.5.6 of \cite{Hofmann-Seal-Tholen14} $U$"/cocomplete normalised modular topological spaces are characterised in terms of `irreducible' closed subsets and, generalising this, in Section~III.5.9 $U$"/cocomplete normalised modular approach spaces are characterised in terms of `irreducible' continuous maps.
	\end{example}
	
	The remainder of this section describes how every completely distributive quantale $\V$ can itself be regarded as a modular $\V$"/valued topological space, that is both normalised and $U$"/cocomplete. With this aim in mind, let $P$ be the powerset monad extended to $\V$"/relations taking values in a completely distributive quantale $\V$. We consider the vertical part $(\Mod P)_\textup v$ of the induced monad $\Mod P$ (see \propref{2-functor Mod}) on the thin equipment $\Mod(\enRel\V) = \enProf\V$ of $\V$"/profunctors. Writing again $P \dfn (\Mod P)_\textup v$, this is the powerset monad on the category $\enCat\V = (\enProf\V)_\textup v$ of $\V$"/categories, whose Eilenberg"/Moore category $(\enCat\V)^P$ of algebras consists of $\V$"/categories $(A, \bar A)$ equipped with a $P$"/algebra structure map $\map a{PA}A$ that is a $\V$"/functor.
		
		We will consider the images of such algebras under the composite functor
	\begin{equation} \label{composite}
		\pars{\enCat\V}^P \xrar C \ModCls\V \xrar{\bar R} \ModCat{(U, \V)}
	\end{equation}	
	where $\bar R$ is given in \thmref{modular V-valued topological space} and $C$ is the ``composition functor'' described in Section~4 of \cite{Tholen09}, when applied to the powerset monad. This composite maps a $P$"/algebra $A = (A, \bar A, a)$ to the modular $(U, \V)$"/category $(\bar R \of C)(A) = (A, \bar A, \alpha)$ whose $\V$"/valued convergence relation is given by
	\begin{displaymath}
		\hmap\alpha{UA}A\colon (\mf x, x) \mapsto \inf_{S \in \mf x} \bar A\bigpars{a(S), x}.
	\end{displaymath}
	
	The examples below describe two types of images under the composite \eqref{composite}. Both depend on the fact that a complete lattice $A$ (for instance $A = \V$) admits two algebra structures over the powerset monad $P$ on $\Set$, given by
	\begin{equation} \label{complete lattice algebra structure}
		\map{a_{\inf}}{PA}A\colon S\mapsto \inf S \qquad \text{and} \qquad \map{a_{\sup}}{PA}A\colon S\mapsto \sup S
	\end{equation}
	respectively.
	\begin{example} \label{Scott topology}
		Let $A$ be a complete lattice. It is easily checked that the $P$"/algebra structure $a_{\inf}$ above is an order preserving map $\map{a_{\inf}}{(PA, P{\leq})}{(A, \leq)}$, so that we may regard $(A, \leq, a_{\inf})$ as an algebra in $\pars{\enCat\2}^P$. Applying the composite functor \eqref{composite}, where $\V = \2$, we obtain a modular topology (\exref{modular (U, V)-category}) on $A$ whose ultrafilter convergence relation we denote by $\hmap{\alpha_{\inf}}{UA}A$; it is given by
		\begin{displaymath}
			\mf x \alpha_{\inf} x \quad \defeq \quad \sup_{S \in \mf x} \inf S \leq x
		\end{displaymath}
		for all $\mf x \in UA$ and $x \in A$. Dually the $P$"/algebra structure $a_{\sup}$ given in \eqref{complete lattice algebra structure} induces a modular topology on the complete lattice $\rev A = (A, \geq)$ that is obtained by reversing the order on $A$.
		
		The proposition below shows that if $A$ is completely distributive then the topology corresponding to the convergence relation $\alpha_{\inf}$ is the \emph{Scott topology} \cite{Scott72} as follows. The open subsets $O \subseteq A$ are the downsets satisfying
		\begin{equation} \label{Scott open}
			\Forall_{D \in \DnDir A} \inf D \in O \implies D \isect O \neq \emptyset
		\end{equation}
		where $D$ ranges over all \emph{down"/directed} subsets of $A$: a subset $D \subseteq A$ is down"/directed whenever it is non"/empty and every finite subset of $D$ has a lower bound in $D$, that is for all $x, y \in D$ there is $z \in D$ with $z \leq x$ and $z \leq y$.
		
		For example the Scott topology on $\brks{-\infty, \infty}$, with respect to the reversed order $\geq$, consists of the open subsets of the form $(x, \infty ]$, where $x \in \brks{0, \infty}$, together with $\brks{-\infty, \infty}$ itself. A function $\map fA{\brks{-\infty, \infty}}$ that is continuous with respect to this topology is called \emph{lower semi"/continuous}; see e.g.\ Section~IV.8 of \cite{Berge59} or Section~2.10 of \cite{Aliprantis-Border06}. Dually, equipping $\brks{-\infty, \infty}$ with the Scott topology with respect to the natural order $\leq$ instead, we obtain the notion of \emph{upper semi"/continuous} function $\map fA{\brks{-\infty, \infty}}$.
		
		On $\2 = \set{\bot \leq \top}$ the Scott topology coincides with the \emph{Sierpi\'nski topology}, which has $\set{\bot}$ as its only non"/trivial open subset.
	\end{example}
	\begin{proposition}
		Let $A$ be a completely distributive complete lattice. The topology corresponding to the convergence relation $\hmap{\alpha_{\inf}}{UA}A$, described in the example above, is the Scott topology. 
	\end{proposition}
	\begin{proof}[Sketch of the proof.]
		By \lemref{minimax} we have $\sup_{S \in \mf x} \inf S = \inf_{S \in \mf x} \sup S$ in the definition of $\alpha_{\inf}$ so that, by \exref{ultrafilter convergence}, the topology corresponding to $\alpha_{\inf}$ is given as follows: $O \subseteq A$ is open precisely if the equivalent conditions below hold.
		\begin{displaymath}
			\Forall_{\substack{\mf x \in UA \\ x \in O}} \bigpars{\sup_{S \in \mf x} \inf S \leq x \implies O \in \mf x} \qquad \iff \qquad \Forall_{\substack{\mf x \in UA \\ x \in O}} \bigpars{\inf_{S \in \mf x} \sup S \leq x \implies O \in \mf x}
		\end{displaymath}
		
		To see that this implies that $O$ satisfies \eqref{Scott open} consider, for any down"/directed subset $D \subseteq A$ with $\inf D \in O$, any ultrafilter $\mf x \in UA$ generated by the proper filter base $\set{\downset x \isect D}_{x \in D}$. Since $\downset x \in \mf x$ for all $x \in D$ it follows that $\inf_{S \in \mf x} \sup S \leq \inf D$, so that $O \in \mf x$ by the above. Because $D \in \mf x$ we conclude that $D \isect O \neq \emptyset$.
		
		Conversely if $O \subseteq A$ is a downset satisfying \eqref{Scott open} then it satisfies the equivalent conditions above. Indeed for any $\mf x \in UA$ and $x \in O$ with $\inf_{S \in \mf x} \sup S \leq x$, the set $D = \set{\sup S}_{S \in \mf x}$ is down"/directed with $\inf D \leq x \in O$. As $O$ is a downset $\inf D \in O$ follows so that $D \isect O \neq \emptyset$ by \eqref{Scott open}. Hence $\sup S \in O$ for some $S \in \mf x$ which, again because $D$ is a downset, implies $S \subseteq O$; we conclude that $O \in \mf x$.
	\end{proof}
	\begin{example} \label{V as a modular approach space}
		Here we consider the $P$"/algebra structure $a_{\inf}$ given by \eqref{complete lattice algebra structure} on a completely distributive quantale $\V = A$. Since the inner hom $\dashcirc$ of $\V$ (see \exref{V-profunctors}) is contravariant in the first variable and an $\inf$"/map in the second, we have
		\begin{displaymath}
			(P\V_{\dashcirc})(S, T) = \inf_{t \in T} \sup_{s \in S} s \dashcirc t \leq (\inf S) \dashcirc (\inf T)
		\end{displaymath}
		for all $S, T \in P\V$, showing that $a_{\inf}$ is a $\V$"/functor $\map{a_{\inf}}{P\V_{\dashcirc}}{\V_{\dashcirc}}$. Hence $(\V_{\dashcirc}, a_{\inf})$ forms a $P$"/algebra in $\pars{\enCat\V}^P$ so that, under the composite functor \eqref{composite}, $\V_{\dashcirc}$ becomes a modular $(U, \V)$"/category (\exref{modular (U, V)-category}), whose $\V$"/valued convergence relation we denote by $\map{\nu_{\inf}}{U\V \times \V}\V$; it is given by
		\begin{displaymath}
			\nu_{\inf}(\mf x, x) = (\sup_{S \in \mf x} \inf S) \dashcirc x.
		\end{displaymath}
		Notice that this defines a modular $(U, \V)$"/category structure on $\V_{\dashcirc}$ that is both normalised (see \eqref{normalised T-category}) and $U$"/cocomplete (\defref{complete}).

		Analogous to the above, the dual $P$"/algebra structure $a_{\sup}$ on $\V$ given by \eqref{complete lattice algebra structure} froms a $\V$"/functor $\map{a_{\sup}}{P\V_{\circdash}}{\V_{\circdash}}$ and thus induces a modular $(U, \V)$"/category structure $\nu_{\sup}$ on $\V_{\circdash}$, that is given by $\nu_{\sup}(\mf x, x) = (\inf_{S \in \mf x} \sup S) \circdash x$.
	\end{example}
	\begin{example} \label{Lawvere quantale as a modular approach space}
		By applying the previous example to the Lawvere quantale $\brks{0, \infty}$, equipped with the generalised metric $\brks{0, \infty}_{\dashcirc}(x, y) = y \ominus x$ (\exref{non-expansive relation}), we obtain the metric convergence relation $\hmap{\nu_{\sup}}{U\brks{0, \infty}}{\brks{0, \infty}}$ given by 
		\begin{displaymath}
			\nu_{\sup}(\mf x, x) = x \ominus (\inf_{S \in \mf x}\sup S)
		\end{displaymath}
		for all $\mf x \in U\brks{0, \infty}$ and $x \in \brks{0, \infty}$. Dually, equipping $\brks{0, \infty}$ with the reversed metric $\brks{0, \infty}_{\circdash}(x, y) = x \ominus y$ instead, we obtain the metric convergence relation $\nu_{\inf}$ that is given by
		\begin{displaymath}
			\nu_{\inf}(\mf x, x) = (\sup_{S \in \mf x}\inf S) \ominus x.
		\end{displaymath}
	
	Under the isomorphisms of \thmref{modular V-valued topological space} the above metric convergence relations correspond to the point"/set distances given by
		\begin{displaymath}
			\delta_{\sup}(S, x) = \begin{cases}
				x \ominus (\sup S) & \text{if $S \neq \emptyset$;} \\
				\infty & \text{if $S = \emptyset$}
			\end{cases} \quad \text{and} \quad \delta_{\inf}(S, x) = \begin{cases}
				(\inf S) \ominus x & \text{if $S \neq \emptyset$;} \\
				\infty & \text{if $S = \emptyset$}				
			\end{cases}
		\end{displaymath}
		respectively, for all $S \in PA$ and $x \in A$. The first of these is used throughout \cite{Lowen97}, see Examples~1.8.33 therein. While the metric convergence relations $\nu_{\sup}$ and $\nu_{\inf}$ are $U$"/cocomplete, notice that both point"/set distances $\delta_{\sup}$ and $\delta_{\inf}$ fail to be $P$"/cocomplete.
		
		Proving that $\nu_{\sup}$ corresponds to $\delta_{\sup}$ amounts to showing that $\delta_{\sup}(S, x) = (\eps_{\brks{0, \infty}} \hc \nu_{\sup})(S, x)$ for all $S$ and $x$, where $\eps_{\brks{0, \infty}}$ is given in \propref{algebraic morphism}. If $S = \emptyset$ this follows from $\emptyset^\sharp = \emptyset$. If $S \neq \emptyset$ then, because $S \mapsto x \ominus (\sup S)$ preserves binary joins, the argument given in the second paragraph of the proof of Theorem~3.6 of \cite{Lai-Tholen17a} can be applied without change.
	\end{example}
	
	\begin{remark}
		For a commutative and completely distributive quantale $\V$, Clementino and Hofmann describe in \cite{Clementino-Hofmann04} a general construction that extends a `suitable' monad $T$ on $\Set$ to the thin equipment $\enRel\V$. In \cite{Clementino-Hofmann09} they show that in this setting $\V$ admits a $T$"/algebra structure whose structure map is a $\V$"/functor, thus generalising \exref{V as a modular approach space} above in the case that $\V$ is commutative.
	\end{remark}
	
	\section{Horizontal \texorpdfstring{$T$}{T}"/morphisms} \label{horizontal T-morphisms}
	The following definition generalises the notions of hemicontinuity for relations between topological spaces (see Section~VI.1 of \cite{Berge59} or Section~17.2 of \cite{Aliprantis-Border06}) to notions of `open' and `closed' horizontal morphism between $T$"/graphs. In \defref{open closed vertical morphism} below these notions are extended to vertical morphisms.
	
	Each of the generalisations of the maximum theorem given in the next section involves a Kan extension along either an open or closed horizontal morphism. Some of these generalisations provide conditions ensuring that the Kan extension itself is an open or closed morphism.
	\begin{definition} \label{open closed horizontal morphism}
		Let $T$ be a lax monad on a thin equipment $\K$. A horizontal morphism $\hmap JAB$ between $T$"/graphs $A = (A, \alpha)$ and $B = (B, \beta)$ is called
		\begin{enumerate}[label=-]
			\item \emph{$T$"/open} if $\alpha \hc J \leq TJ \hc \beta$;
			\item \emph{$T$"/closed} if $TJ \hc \beta \leq \alpha \hc J$.
		\end{enumerate}
	\end{definition}
	Notice that if $T$ is a lax monad on the thin equipment $\enRel\V$ of relations taking values in a quantale $\V$, and $T \dfn \Mod(T)$ is the induced lax monad on $\enProf\V$, then $T$"/open/$T$"/closed horizontal morphisms in $\enProf\V$ are precisely those $\V$"/profunctors whose underlying $\V$"/relations are $T$"/open/$T$"/closed.
	
	\begin{example} \label{open/closed relations}
		Let $P$ be the powerset monad extended to $\V$"/relations (\exref{powerset monad}). A $\V$"/relation $\hmap JAB$ between $\V$"/valued pseudoclosure spaces $A = (A, \delta)$ and $B = (B, \zeta)$ (\exref{closure space}) is $P$"/open if
		\begin{enumerate}
			\item[(O)] $\displaystyle\delta(S, x) \tens J(x, y) \leq \sup_{T \in PB} \bigpars{\inf_{t \in T}\sup_{s \in S} J(s, t)} \tens \zeta(T, y)$ 
		\end{enumerate}
		for all $S \in PA$, $x \in A$ and $y \in B$. Dually $J$ is $P$"/closed if
		\begin{enumerate}
			\item[(C)] $\displaystyle\bigpars{\inf_{t \in T}\sup_{s \in S}J(s, t)} \tens \zeta(T, y) \leq \sup_{x \in A} \delta(S, x) \tens J(x,y)$
		\end{enumerate}
		for all $S \in PA$, $T \in PB$ and $y \in B$.
		
		It is straightforward to show that if $\hmap JAB$ is discrete, that is $\im J \subseteq \set{\bot, k}$, while $B$ is a $\V$"/valued preclosure space (\exref{closure space}), then the axioms above reduce to
		\begin{enumerate}
			\item[(O')] $\delta(S, x) \tens J(x, y) \leq \zeta(J_k S, y)$;
			\item[(C')] $\zeta(J_k S, y) \leq \displaystyle\sup_{z \in \rev J_k y} \delta(S, z)$,
		\end{enumerate}
		for all $S \in PA$, $x \in A$ and $y \in B$. Here $\hmap{J_k}AB$ is the ordinary relation defined by $xJ_ky \defeq J(x, y) = k$ for all $x \in A$ and $y \in B$.
		
		Choosing $\V = \2$ in the above, the proposition below shows that a relation $\hmap JAB$ between closure spaces is open precisely if, for any open $O \subseteq B$, the preimage $\rev JO$ is $P$"/open in $A$. Dually it is easy to show that $J$ is $P$"/closed precisely if, for every closed $V \subseteq A$, the image $JV$ is closed in $B$.
	\end{example}
	
	\begin{proposition} \label{P-open}
		Let $P$ be the powerset monad on $\Rel$. For a relation $\hmap JAB$ between closure spaces the following are equivalent:
		\begin{enumerate}[label=\textup{(\alph*)}]
			\item $J$ is $P$"/open;
			\item $J\bar S \subseteq \overline{JS}$ for all $S \subseteq A$;
			\item $\rev JO$ is open in $A$ for all $O \subseteq B$ open.
		\end{enumerate}
	\end{proposition}
	\begin{proof}
		(a) $\iff$ (b) follows immediately from axiom (O') in the previous example.
		
		(b) $\implies$ (c). Assume that $O \subseteq B$ is open but its preimage $\rev JO \subseteq A$ is not, that is $\overline{A - \rev JO} \nsubseteq A - \rev JO$. Thus $\overline{A - \rev JO} \isect \rev JO \neq \emptyset$ or, equivalently, $J\overline{A - \rev JO} \isect O \neq \emptyset$. But part (b) implies
		\begin{displaymath}
			J\overline{A - \rev JO} \subseteq \overline{J(A - \rev JO)} \subseteq \overline{B - O} = B - O,
		\end{displaymath}
		contradicting the latter.
		
		(c) $\implies$ (b). Assuming (c), suppose that (b) does not hold, i.e.\ for some $S \subseteq A$ we have $J\bar S \nsubseteq \overline{JS}$ or, equivalently, $\bar S \isect \rev J(B - \overline{JS}) \neq \emptyset$. But this is contradicted by
		\begin{displaymath}
			\bar S = \overline{A - (A - S)} \subseteq \overline{A - \rev J(B - JS)} \subseteq \overline{A - \rev J(B - \overline{JS})} = A - \rev J(B - \overline{JS}),
		\end{displaymath}
		where the last equality follows from part (c).
	\end{proof}
	
	The following theorem describes open and closed $\V$"/relations $\hmap JAB$ between $(U, \V)$"/categories $A$ and $B$ in terms of the corresponding $\V$"/valued topological space structures on $A$ and $B$. In order to state it we need the following definition.
	\begin{definition} \label{U-compact}
		Let $U$ be the ultrafilter monad on $\enRel\V$, where $\V$ is a completely distributive quantale. Given a $(U, \V)$"/graph $A = (A, \alpha)$ and a set $B$ we will call a $\V$"/relation $\hmap JAB$ \emph{$U$"/compact} whenever $(UJ)(\id, \iota_B) \leq \alpha \hc J$, that is
	\begin{displaymath}
		\inf_{S \in \mf x} \sup_{s \in S} J(s, y) \leq \sup_{x \in A} \alpha(\mf x, x) \tens J(x, y)
	\end{displaymath}
	for all $\mf x \in UA$ and $y \in B$.
	\end{definition}
	If $\hmap JAB$ is discrete, i.e.\ $\im J \subseteq \set{\bot, k}$, then the condition above reduces to
	\begin{displaymath}
		k \leq \sup_{x \in \rev J_k y} \alpha(\mf y, x)
	\end{displaymath}
	for all $y \in B$ and $\mf y \in UA$ with $\rev J_k y \in \mf y$. In particular if $\V = \2$, so that $A$ is a pseudotopological space and $J$ is an ordinary relation, this means that every ultrafilter on $\rev Jy$ converges to some $x \in \rev Jy$; that is, for each $y \in B$ the preimage $\rev Jy$ is compact in $A$.
	\begin{theorem} \label{P and U horizontal morphism}
		Let $\V$ be a completely distributive quantale and let $U$ and $P$ be the ultrafilter and powerset monads on $\enRel\V$. Consider $(U, \V)$"/graphs $A = (A, \alpha)$ and $B = (B, \beta)$ as well as their induced $\V$"/valued preclosure space structures $\delta = \eps_A \hc \alpha$ and $\zeta = \eps_B \hc \beta$; see \propref{algebraic functor}. For a $\V$"/relation $\hmap JAB$ the following hold:
		\begin{enumerate}[label=\textup{(\alph*)}]
			\item if $J$ is $U$"/open as a $\V$"/relation of $(U, \V)$"/graphs then it is $P$"/open as a $\V$"/relation $\hmap J{(A, \delta)}{(B, \zeta)}$ of $\V$"/valued preclosure spaces;
			\item if $J$ is $U$"/closed as a $\V$"/relation of $(U, \V)$"/graphs then it is both $U$"/compact, in the sense above, as well as $P$"/closed as a $\V$"/relation $\hmap J{(A, \delta)}{(B, \zeta)}$ of $\V$"/valued preclosure spaces.
		\end{enumerate}
		The converse of \textup{(a)} holds as soon as $B$ is unitary and $U(UJ \hc \beta) = U^2J \hc U\beta$; the converse of \textup{(b)} holds whenever $A$ is a $(U, \V)$"/category and $U(\alpha \hc J) = U\alpha \hc UJ$.
	\end{theorem}
	\begin{proof}
		Part (a). Suppose that $J$ is $U$"/open as a $\V$"/relation between $(U, \V)$"/graphs, that is $\alpha \hc J \leq UJ \hc \beta$. Using \propref{algebraic morphism}(b) we then have
		\begin{displaymath}
			\delta \hc J = \eps_A \hc \alpha \hc J \leq \eps_A \hc UJ \hc \beta = PJ \hc \eps_B \hc \beta = PJ \hc \zeta,
		\end{displaymath}
		showing that $J$ is $P$"/open as a $\V$"/relation between $\V$"/valued preclosure spaces. For the converse assume that $B$ is unitary and that $U(UJ \hc \beta) = U^2J \hc U\beta$. It follows that
		\begin{multline*}
			\bigpars{U(UJ \hc \beta)}(\id, \iota^U_B) = U^2J \hc U\beta \hc \iota^{U*}_B \leq U^2J \hc \beta(\mu^U_B, \id) \\
			= U^2J \hc \mu^U_{B*} \hc \beta \leq \mu^U_{A*} \hc UJ \hc \beta = (UJ \hc \beta)(\mu^U_A, \id),
		\end{multline*}
		where the inequalities follow from $B$ being unitary and from applying \lemref{horizontal cells} to the naturality cell of $\mu$ at $J$. Thus by \propref{algebraic morphism}(c) we have $\eps_A \lhom (\eps_A \hc UJ \hc \beta) = UJ \hc \beta$. Using this, assuming that $J$ is $P$"/open, it follows that
		\begin{multline*}
			\alpha \hc J \overset{\text{(i)}}\leq \eps_A \lhom (\eps_A \hc \alpha \hc J) = \eps_A \lhom (\delta \hc J) \leq \eps_A \lhom (PJ \hc \zeta) \\
			= \eps_A \lhom (PJ \hc \eps_B \hc \beta) \overset{\text{(ii)}}= \eps_A \lhom (\eps_A \hc UJ \hc \beta) = UJ \hc \beta,
		\end{multline*}
		where (i) is given by the unit of $\eps_A \hc \dash \ladj \eps_A \lhom \dash$ and (ii) follows from \propref{algebraic morphism}(b). This shows that $J$ is $U$"/open.
		
		Part (b). Assume that $J$ is $U$"/closed as a $\V$"/relation between $(U, \V)$"/graphs, that is $UJ \hc \beta \leq \alpha \hc J$. Then $J$ is $U$"/compact:
		\begin{displaymath}
			(UJ)(\id, \iota^U_B) = UJ \hc \iota^{U*}_B \leq UJ \hc \beta \leq \alpha \hc J,
		\end{displaymath}
		where the first inequality follows from applying \lemref{horizontal cells} to the unit cell of $\beta$. That $J$ is $P$"/closed as a $\V$"/relation between the $\V$"/valued preclosure spaces $(A, \delta)$ and $(B, \zeta)$ is shown by
		\begin{displaymath}
			PJ \hc \zeta = PJ \hc \eps_B \hc \beta = \eps_A \hc UJ \hc \beta \leq \eps_A \hc \alpha \hc J = \delta \hc J
		\end{displaymath}
		where the second identity follows from \propref{algebraic morphism}(b). For the converse assume that $J$ is $U$"/compact and $P$"/closed while $A$ is a $(U, \V)$"/category and $U(\alpha \hc J) = U\alpha \hc UJ$. Using $U$"/compactness of $J$ and the associativity axiom for $A$ it follows that
		\begin{displaymath}
			U(\alpha \hc J)(\id, \iota^U_B) = U\alpha \hc UJ \hc \iota^{U*}_B \leq U\alpha \hc \alpha \hc J \leq \alpha(\mu_A^U,\id) \hc J = (\alpha \hc J)(\mu_A^U, \id),
		\end{displaymath}
		so that $\eps_A \lhom (\eps_A \hc \alpha \hc J) = \alpha \hc J$ by \propref{algebraic morphism}(c). We conclude that
		\begin{multline*}
			UJ \hc \beta \leq \eps_A \lhom (\eps_A \hc UJ \hc \beta) = \eps_A \lhom (PJ \hc \eps_B \hc \beta) \\
			= \eps_A \lhom (PJ \hc \zeta) \leq \eps_A \lhom (\delta \hc J) = \eps_A \lhom (\eps_A \hc \alpha \hc J) = \alpha \hc J,
		\end{multline*}
		showing that $J$ is $U$"/closed.
	\end{proof}
	
	\begin{example} \label{hemi-continuous relation}
		Let us choose $\V = \2$ in the previous theorem, so that it applies to a relation $\hmap JAB$ between topological spaces. Combined with the descriptions of $P$"/open and $P$"/closed relations between topological spaces given in \exref{open/closed relations}, as well as the description of $U$"/compact relations following \defref{U-compact}, we find that $J$ is $U$"/open precisely if it is \emph{lower hemi"/continuous} in the classical sense, see e.g.\ Section~VI.1 of \cite{Berge59}, while $J$ is $U$"/closed precisely if its reverse \mbox{$\hmap{\rev J}BA$} is \emph{upper hemi"/continuous}. More precisely, the notions of $U$"/open and $U$"/closed relation describe lower/upper hemi"/continuity in terms of ultrafilter convergence. For a closely related description of hemi"/continuity in terms of nets see Section~17.3 of \cite{Aliprantis-Border06}.
	\end{example}
	
	In the definition below the notions of openness and closedness are extended to vertical morphisms. In the case that $T$ is a lax monad on $\enRel\V$ this recovers the notions of `open' and `proper' morphism between $(T, \V)$"/categories, as studied in Section~V.3 of \cite{Hofmann-Seal-Tholen14}, although there $T$ is not required to be normal.
	\begin{definition} \label{open closed vertical morphism}
		Let $T$ be a normal lax monad on a thin equipment $\K$. A morphism $\map fAC$ of $T$"/graphs $A = (A, \alpha)$ and $C = (C, \gamma)$ is called
		\begin{enumerate}[label=-]
			\item \emph{$T$"/open} if its conjoint $\hmap{f^*}CA$ is $T$"/open, that is $\gamma(\id, f) \leq (Tf)^* \hc \alpha$;
			\item \emph{$T$"/closed} if its companion $\hmap{f_*}AC$ is $T$"/closed, that is $\gamma(Tf, \id) \leq \alpha \hc f_*$.
		\end{enumerate}
	\end{definition}
	We remark that, in rewriting the inequalities of \defref{open closed horizontal morphism} into those above, we use the fact that $T$ is normal, so that it preserves companions and conjoints. We shall only describe open and closed morphisms in $\ModCat{(T, \V)}$ (see \secref{modular T-graphs}) where either $T = P$ is the powerset monad or $T = U$ is the ultrafilter monad. For a description of open and closed morphisms in $\enCat{(U, \V)}$ we refer to Section~V.3.4 of \cite{Hofmann-Seal-Tholen14}.
	\begin{example} \label{open/closed maps between modular V-valued closure spaces}
		Let $P$ be the powerset monad on $\enProf\V$. A morphism $\map fAC$ of modular $\V$"/valued closure spaces $A = (A, \bar A, \delta)$ and $C = (C, \bar C, \zeta)$ is $P$"/open if
		\begin{displaymath}
			\zeta(T, fx) \leq \sup_{S \in PA} \bigpars{\inf_{s \in S} \sup_{t \in T} C(t, fs)} \tens \delta(S, x)
		\end{displaymath}
		for all $T \in PC$ and $x \in A$; dually $f$ is $P$"/closed if
		\begin{displaymath}
			\zeta(fS, z) \leq \sup_{x \in A} \delta(S, x) \tens C(fx, z)
		\end{displaymath}
		for all $S \in PA$ and $z \in C$. If $\V$ is completely distributive so that the ultrafilter $U$ extends to $\enProf\V$ as well (see \exref{ultrafilter monad}) then, by \thmref{P and U horizontal morphism}, a morphism $\map fAC$ of modular $\V$"/valued topological spaces is $U$"/open precisely if it is $P$"/open, while it is $U$"/closed precisely when it is $P$"/closed and its companion $\hmap{f_*}AC$ is $U$"/compact.
	\end{example}
	\begin{example} \label{open/closed maps between ordered topological spaces}
		Taking $\V = \2$ in the previous example, a monotone continuous map $\map fAC$ between modular topological spaces (\exref{modular (U, V)-category}) is $U$"/open if $\downset fO \subseteq C$ is open for all $O \subseteq A$ open; it is $U$"/closed whenever it is $P$"/closed, that is $\upset fV \subseteq C$ is closed for all $V \subseteq A$ closed, while $\inv f(\downset z) \subseteq A$ is compact for all $z \in C$.
		
		In case $C = \brks{-\infty, \infty}$ is equipped with the Scott topology with respect to $\geq$ (\exref{Scott topology}) then the first two of the conditions above weaken the classical notions of open and closed maps $\map fA{\brks{-\infty, \infty}}$: $f$ is $U$"/open means \mbox{$fO \subseteq \brks{-\infty, \infty}$} does not have a minimum, for all $O \subseteq A$ open, while $f$ is $P$"/closed means $fV \subseteq \brks{-\infty, \infty}$ has a maximum, for all closed $V \subseteq A$. The third condition above means that $\inv f(\brks{z, \infty}) \subseteq A$ is compact for all $z \in \brks{-\infty, \infty}$; functions $\map fA{\brks{-\infty, \infty}}$ with this property are called \emph{upper semi"/compact} or \emph{sup"/compact}, see e.g.\ Section~1 of \cite{Feinberg-Kasyanov15}.
	\end{example}
	
	We close this section with a couple of remarks.
	\begin{remark}
		Let $T$ be a lax monad on a thin equipment $\K$. Notice that the horizontal composite $J \hc H$ of $T$"/open horizontal morphisms $\hmap JAB$ and \mbox{$\hmap HBE$} is again $T$"/open. In fact $T$"/graphs, $T$"/morphisms, $T$"/open horizontal morphisms and the cells between them in $\K$ form a `thin double category' $T\text-\mathsf{Opn}$. While $T\text-\mathsf{Opn}$ has all companions $f_*$, the conjoint $f^*$ of a $T$"/morphism $f$ will in general not be $T$"/open, but $T$"/closed instead. If the monad $T$ preserves horizontal composition strictly then we are able to compose $T$"/closed horizontal morphisms as well, so that they form the horizontal morphisms of a thin double category $T\text-\mathsf{Cls}$.
	\end{remark}
	\begin{remark}
		Let $T$ be a lax monad on $\enRel\V$. Weakening the notion of modular $(T, \V)$"/category considered in our \secref{modular T-graphs}, in Section~5 of \cite{Tholen09} an `open $\V$"/structured $(T, \V)$"/category' $A$ is defined to be a $(T, \V)$"/category $(A, \alpha)$ equipped with a $\V$"/category structure $\hmap{\bar A}AA$ that is $T$"/open in our sense. Similarly `closed $\V$"/structured $(T, \V)$"/categories' $A$ are defined to be triples $(A, \bar A, \alpha)$ with $(A, \alpha)$ a $(T, \V)$"/category and $(A, \bar A)$ a $\V$"/category, such that $\bar A$ is $T$"/closed with respect to $\alpha$ and $T(\alpha \hc \bar A) = T\alpha \hc T\bar A$.
	\end{remark}
	
	\section{Generalisations of the maximum theorem} \label{maximum theorem}
	We are now ready to state and prove generalisations of the maximum theorem for Kan extensions of $T$"/morphisms between $T$"/graphs. Starting with right Kan extensions the first of these generalisations, \thmref{maximum theorem for right Kan extensions into a complete graph} below, assumes that the target of the Kan extension is $T$"/cocomplete (\defref{complete}), while \thmref{maximum theorem for right Kan extensions satisfying Beck-Chevalley} instead assumes a Kan extension that satisfies the Beck"/Chevalley condition (\thmref{Beck-Chevalley}). Similarly left Kan extensions are considered in \thmref{maximum theorem for left Kan extensions into a T-complete T-graph} and \thmref{maximum theorem for left Kan extensions satisfying Beck-Chevalley}.
	\begin{theorem} \label{maximum theorem for right Kan extensions into a complete graph}
		Let $T$ be a normal lax monad on a thin equipment $\K$. Let $\hmap JAB$ be a $T$"/open horizontal morphism between $T$"/graphs and $\map eBM$ a $T$"/morphism into a $T$"/cocomplete $T$"/graph $M$. The right Kan extension $\map rAM$ of $e$ along $J$ in $\K$, if it exists, is a $T$"/morphism.
	\end{theorem}
	\begin{proof}
		We write $\hmap\alpha{TA}A$, $\hmap\beta{TB}B$ and $\hmap{m_*}{TM}M$ for the horizontal structure morphisms of $A$, $B$ and $M$ respectively; because $M$ is $T$"/cocomplete the last of these is the companion of a vertical morphism $\map m{TM}M$. Consider the composite of cells on the left"/hand side below, where $T\eps$ denotes the `$T$"/image' of the cell $\eps$ definining $r$ and where the cells denoted $\leq$, from top to bottom, exist because $J$ is $T$"/open, $e$ is a $T$"/morphism, and $m_*$ is the companion of $m$.
		\begin{displaymath}
			\begin{tikzpicture}[textbaseline]
				\matrix(m)[math35]{TA & A & B \\ TA & TB & B \\ TM & TM & M \\ M & M & M \\};
				\path[map]	(m-1-1) edge[barred] node[above] {$\alpha$} (m-1-2)
										(m-1-2) edge[barred] node[above] {$J$} (m-1-3)
										(m-2-1) edge[barred] node[above] {$TJ$} (m-2-2)
														edge node[left] {$Tr$} (m-3-1)
										(m-2-2) edge[barred] node[above] {$\beta$} (m-2-3)
														edge node[right, inner sep=1.5pt] {$Te$} (m-3-2)
										(m-2-3) edge node[right] {$e$} (m-3-3)
										(m-3-1) edge node[left] {$m$} (m-4-1)
										(m-3-2) edge[barred] node[below] {$m_*$} (m-3-3)
														edge node[left] {$m$} (m-4-2);
				\path				(m-1-1) edge[eq] (m-2-1)
										(m-1-3) edge[eq] (m-2-3)
										(m-3-1) edge[eq] (m-3-2)
										(m-3-3) edge[eq] (m-4-3)
										(m-4-1) edge[eq] (m-4-2)
										(m-4-2) edge[eq] (m-4-3);
				\draw				($(m-1-1)!0.5!(m-2-3)$) node[rotate=-90] {$\leq$}
										($(m-2-1)!0.5!(m-3-2)$) node {$T\eps$}
										($(m-2-2)!0.5!(m-3-3)$) node[rotate=-90] {$\leq$}
										($(m-3-2)!0.5!(m-4-3)$) node[rotate=-90] {$\leq$};
			\end{tikzpicture} = \begin{tikzpicture}[textbaseline]
				\matrix(m)[math35]{TA & A & B \\ TM & & \\ M & M & M \\};
				\path[map]	(m-1-1) edge[barred] node[above] {$\alpha$} (m-1-2)
														edge node[left] {$Tr$} (m-2-1)
										(m-1-2) edge[barred] node[above] {$J$} (m-1-3)
														edge node[right] {$r$} (m-3-2)
										(m-1-3) edge node[right] {$e$} (m-3-3)
										(m-2-1) edge node[left] {$m$} (m-3-1);
				\path				(m-3-1) edge[eq] (m-3-2)
										(m-3-2) edge[eq] (m-3-3);
				\draw				($(m-1-1)!0.5!(m-3-2)$) node[rotate=-90] {$\leq$}
										($(m-1-2)!0.5!(m-3-3)$) node {$\eps$};
			\end{tikzpicture}
		\end{displaymath}
		By the universal property of $\eps$ the composite on the left factors as shown. Composing this factorisation with the appropriate cell among the pair of cells that defines $m_*$, we obtain the cell that exhibits $r$ as a $T$"/morphism.
	\end{proof}
	
	\begin{example}
		If $T = P$ is the powerset monad on the thin equipment $\K = \enProf\2$ of modular relations (\exref{modular relation}), so that $A$, $B$ and $M$ in the theorem above are modular preclosure spaces (\exref{modular closure space}), then the categorical proof above reduces to the following elementary proof. As $M$ is assumed to be $P$"/cocomplete its closed subsets are principal upsets $\upset z$ (\exref{complete modular closure space}), so that for the continuity of $r$ it suffices to show that $\inv r(\upset z)$ is closed for all $\upset z \subseteq M$ closed.
		
		To see this first notice that from the definition $rx = \inf_{y \in Jx} ey$ of $r$ (\exref{Kan extensions in ordered sets}) it follows that
		\begin{displaymath}
			S \subseteq \inv r(\upset z) \quad \iff \quad JS \subseteq \inv e(\upset z)
		\end{displaymath}
		for any $S \subseteq A$. Using this we find
		\begin{align*}
			S \subseteq \inv r(\upset z) \quad &\iff \quad JS \subseteq \inv e(\upset z) \\
			&\iff \quad \overline{JS} \subseteq \inv e(\upset z) && \text{(because $e$ is continuous)} \\
			&\implies \quad J\bar S \subseteq \inv e(\upset z) && \text{(because $J$ is $P$"/open; see \propref{P-open})} \\
			&\iff \quad \bar S \subseteq \inv r(\upset z),
		\end{align*}
		so that taking $S = \inv r(\upset z)$ here proves the continuity of $r$.
	\end{example}
	
	\begin{example} \label{right Kan extension into extended real line}
		Consider topological spaces $A$ and $B$ and let $\hmap JAB$ be a lower hemi"/continuous relation, that is $J$ is $P$"/open in our sense (see \exref{open/closed relations} and \exref{hemi-continuous relation}). Let $\map eB{\brks{-\infty, \infty}}$ be a lower semi"/continuous map, that is $e$ is continuous with respect to the Scott topology on $\brks{-\infty, \infty}$ with the reverse order $\geq$ (see \exref{Scott topology}). Regarding $A$ and $B$ as topological spaces with discrete orders, the previous theorem asserts that the right Kan extension $\map rA{\brks{-\infty, \infty}}$ of $e$ along $J$, given by
		\begin{displaymath}
			rx = \sup_{y \in Jx} ey
		\end{displaymath}
		for all $x \in A$, is lower semi"/continuous. This recovers partly Theorem~1 of Section~VI.3 of \cite{Berge59} (or Lemma~17.29 of \cite{Aliprantis-Border06}), where more general maps of the form \mbox{$\map e{A \times B}{\brks{-\infty, \infty}}$} are treated; see the final remark of the Introduction.
	\end{example}
	
	\begin{example}
		Recall from \exref{Lawvere quantale as a modular approach space} that the canonical normalised modular approach space structure on the Lawvere quantale $\brks{0, \infty}$ is given by the point"/set distance
		\begin{displaymath}
			\delta_{\sup}(S, x) = \begin{cases}
				x \ominus (\sup S) & \text{if $S \neq \emptyset$;} \\
				\infty & \text{otherwise,}
			\end{cases}
		\end{displaymath}
		where $\ominus$ denotes truncated difference. Consider a metric relation $\hmap JAB$ between approach spaces that is $P$"/open (\exref{open/closed relations}) or, equivalently by \thmref{P and U horizontal morphism}, $U$"/open, as well as a continuous map \mbox{$\map eB{\brks{0, \infty}}$}. Regarding $A$ and $B$ as modular approach spaces with discrete metrics, the previous theorem asserts that the right Kan extension $\map rA{\brks{0, \infty}}$ of $e$ along $J$, given by the suprema
		\begin{displaymath}
			rx = \sup_{y \in B} ey \ominus J(x, y)
		\end{displaymath}
		for all $x \in A$ (see \propref{Kan extension expression}), is continuous. The previous theorem likewise applies to right Kan extensions into $\brks{0, \infty}$ equipped with the `reversed' point"/set distance $\delta_{\inf}$ of \exref{Lawvere quantale as a modular approach space}.
	\end{example}	
	\begin{example}
		Analogous to the previous example, the above theorem applies to right Kan extensions of continuous maps $\map eB{\Delta_{\tn}}$ of modular probabilistic approach spaces. Here $\Delta_{\tn}$ is the space of distribution functions (\exref{Delta}), equipped with either of the probabilistic approach space structures that are decribed in \exref{V as a modular approach space} for $\V = \Delta_{\tn}$.
	\end{example}
	
	\begin{theorem}\label{maximum theorem for right Kan extensions satisfying Beck-Chevalley}
		Let $T$ be a normal lax monad on a thin equipment $\K$. Let $A$, $B$ and $M$ be $T$"/graphs, $\map eBM$ a $T$"/morphism and $\hmap JAB$ a $T$"/open horizontal morphism. The right Kan extension $\map rAM$ of $e$ along $J$ in $\K$, if it exists, is a $T$"/morphism whenever it satisfies the Beck"/Chevalley condition.
		
		Moreover, in that case $r$ is $T$"/closed as soon as both $e$ and $J$ are $T$"/closed, provided that $TJ \hc Te_* = T(J \hc e_*)$.
	\end{theorem}
	\begin{proof}
		We write $\hmap\alpha{TA}A$, $\hmap\beta{TB}B$ and $\hmap\nu{TM}M$ for the horizontal structure morphisms of $A$, $B$ and $M$. Consider the composite on the left"/hand side below, where $T\eps$ denotes the `$T$"/image' of the universal cell $\eps$ that defines $r$ and where the other two cells exhibit $J$ as a $T$"/open horizontal morphism and $e$ as a $T$"/morphism respectively.
		\begin{displaymath}
			\begin{tikzpicture}[textbaseline]
				\matrix(m)[math35]{TA & A & B \\ TA & TB & B \\ TM & TM & M \\};
				\path[map]	(m-1-1) edge[barred] node[above] {$\alpha$} (m-1-2)
										(m-1-2) edge[barred] node[above] {$J$} (m-1-3)
										(m-2-1) edge[barred] node[above] {$TJ$} (m-2-2)
														edge node[left] {$Tr$} (m-3-1)
										(m-2-2) edge[barred] node[above] {$\beta$} (m-2-3)
														edge node[right, inner sep=1.5pt] {$Te$} (m-3-2)
										(m-2-3) edge node[right] {$e$} (m-3-3)
										(m-3-2) edge[barred] node[below] {$\nu$} (m-3-3);
				\path				(m-1-1) edge[eq] (m-2-1)
										(m-1-3) edge[eq] (m-2-3)
										(m-3-1) edge[eq] (m-3-2);
				\draw				($(m-1-1)!0.5!(m-2-3)$) node[rotate=-90] {$\leq$}
										($(m-2-1)!0.5!(m-3-2)$) node {$T\eps$}
										($(m-2-2)!0.5!(m-3-3)$) node[rotate=-90] {$\leq$};
			\end{tikzpicture} = \begin{tikzpicture}[textbaseline]
				\matrix(m)[math35]{TA & A & B \\ TM & M & M \\};
				\path[map]	(m-1-1) edge[barred] node[above] {$\alpha$} (m-1-2)
														edge node[left] {$Tr$} (m-2-1)
										(m-1-2) edge[barred] node[above] {$J$} (m-1-3)
														edge node[right] {$r$} (m-2-2)
										(m-1-3) edge node[right] {$e$} (m-2-3)
										(m-2-1) edge[barred] node[below] {$\nu$} (m-2-2);
				\path				(m-2-2) edge[eq] (m-2-3);
				\draw				($(m-1-1)!0.5!(m-2-2)$) node[rotate=-90] {$\leq$}
										($(m-1-2)!0.5!(m-2-3)$) node {$\eps$};
			\end{tikzpicture}
		\end{displaymath}
		By assumption $r$ satisfies the Beck"/Chevalley condition so that, by the horizontal dual of \thmref{Beck-Chevalley}, the composite factors through $\eps$ as shown. This factorisation exhibits $r$ as a $T$"/morphism.
		
		Now assume that both $e$ and $J$ are $T$"/closed and that $T$ preserves the horizontal composite $J \hc e_*$. That $r$ is $T$"/closed is shown by
		\begin{multline*}
			\nu(Tr, \id) = Tr_* \hc \nu \overset{\text{(i)}}= T(J \hc e_*) \hc \nu = TJ \hc Te_* \hc \nu\\
			\overset{\text{(ii)}}\leq TJ \hc \beta \hc e_* \overset{\text{(iii)}}\leq \alpha \hc J \hc e_* \overset{\text{(i)}}= \alpha \hc r_*,
		\end{multline*}
		where the equalities marked (i) follow from the Beck"/Chevalley condition for $r$ while the inequalities marked (ii) and (iii) follow from $e$ and $J$ being $T$"/closed respectively.
	\end{proof}
	
	\begin{example}
		In the setting of \exref{right Kan extension into extended real line} assume that the relation $\hmap JAB$, besides being lower hemi"/continuous, is upper hemi"/continuous (see \exref{hemi-continuous relation}). Also assume that the right Kan extension $r$ of $e$ along $J$ satisfies the Beck"/Chevalley condition, that is the suprema defining $r$ are attained as maxima (\exref{Beck-Chevalley for ordered sets}). The second assertion of the previous theorem states that $r$ is $P$"/closed and upper semi"/compact (\exref{open/closed maps between ordered topological spaces}) whenever the map $e$ is. 
	\end{example}
	
	Next we turn to generalisations of the maximum theorem for left Kan extensions between $T$"/graphs.
	\begin{theorem} \label{maximum theorem for left Kan extensions into a T-complete T-graph}
		Let $T$ be a normal lax monad on a thin equipment $\K$. Let \mbox{$\hmap JAB$} be a $T$"/closed horizontal morphism between $T$"/graphs and let $\map dAM$ be a $T$"/morphism into a $T$"/cocomplete $T$"/graph $M = (M, m_*)$, where $\map m{TM}M$ (see \defref{complete}). The left Kan extension $\map lBM$ of $d$ along $J$ in $\K$, if it exists, is a $T$"/morphism whenever $m \of Tl$ is the left Kan extension of $m \of Td$ along $TJ$.
	\end{theorem}
	\begin{proof}
		Writing $\alpha$ and $\beta$ for the $T$"/graph structures of $A$ and $B$, consider on the left"/hand side below the composite of the cells exhibiting $J$ as a $T$"/closed horizontal morphism, $d$ as a $T$"/morphism, $l$ as a left Kan extension and $m_*$ as a companion.
		\begin{displaymath}
			\begin{tikzpicture}[textbaseline]
				\matrix(m)[math35]{TA & TB & B \\ TA & A & B \\ TM & M & M \\ M & M & M \\};
				\path[map]	(m-1-1) edge[barred] node[above] {$TJ$} (m-1-2)
										(m-1-2) edge[barred] node[above] {$\beta$} (m-1-3)
										(m-2-1) edge[barred] node[above] {$\alpha$} (m-2-2)
														edge node[left] {$Td$} (m-3-1)
										(m-2-2) edge[barred] node[above] {$J$} (m-2-3)
														edge node[right, inner sep=1.5pt] {$d$} (m-3-2)
										(m-2-3) edge node[right] {$l$} (m-3-3)
										(m-3-1) edge[barred] node[below] {$m_*$} (m-3-2)
														edge node[left] {$m$} (m-4-1);
				\path				(m-1-1) edge[eq] (m-2-1)
										(m-1-3) edge[eq] (m-2-3)
										(m-3-2) edge[eq] (m-3-3)
														edge[eq] (m-4-2)
										(m-3-3) edge[eq] (m-4-3)
										(m-4-1) edge[eq] (m-4-2)
										(m-4-2) edge[eq] (m-4-3);
				\draw				($(m-1-1)!0.5!(m-2-3)$) node[rotate=-90] {$\leq$}
										($(m-2-2)!0.5!(m-3-3)$) node {$\eta$}
										($(m-2-1)!0.5!(m-3-2)$) node[rotate=-90] {$\leq$}
										($(m-3-1)!0.5!(m-4-2)$) node[rotate=-90] {$\leq$};
			\end{tikzpicture} = \begin{tikzpicture}[textbaseline]
				\matrix(m)[math35]{TA & TB & B \\ TM & TM & \\ M & M & M \\};
				\path[map]	(m-1-1) edge[barred] node[above] {$TJ$} (m-1-2)
														edge node[left] {$Td$} (m-2-1)
										(m-1-2) edge[barred] node[above] {$\beta$} (m-1-3)
														edge node[right] {$Tl$} (m-2-2)
										(m-1-3) edge node[right] {$l$} (m-3-3)
										(m-2-1) edge node[left] {$m$} (m-3-1)
										(m-2-2) edge node[left] {$m$} (m-3-2);
				\path				(m-2-1) edge[eq] (m-2-2)
										(m-3-1) edge[eq] (m-3-2)
										(m-3-2) edge[eq] (m-3-3);
				\draw				($(m-1-2)!0.5!(m-3-3)$) node[rotate=-90] {$\leq$}
										($(m-1-1)!0.5!(m-2-2)$) node {$T\eta$};
			\end{tikzpicture}
		\end{displaymath}
		By assumption the first column in the right"/hand side above defines $m \of Tl$ as a left Kan extension so that the left"/hand side factors as shown. Composing this factorisation with the appropriate cell among the pair of cells that defines $m_*$, we obtain the cell that exhibits $l$ as a $T$"/morphism.
	\end{proof}
	\begin{theorem} \label{maximum theorem for left Kan extensions satisfying Beck-Chevalley}
		Let $T$ be a normal lax monad on a thin equipment $\K$. Let $A$, $B$ and $M$ be $T$"/graphs, $\map dAM$ a $T$"/morphism and $\hmap JAB$ a $T$"/closed horizontal morphism. The left Kan extension $\map lBM$ of $d$ along $J$ in $\K$, if it exists, is a $T$"/morphism whenever it satisfies the Beck"/Chevalley condition and $Td^* \hc TJ = T(d^* \hc J)$.
		
		Moreover, in that case $l$ is $T$"/open as soon as both $d$ and $J$ are $T$"/open.
	\end{theorem}
	\begin{proof}
		Denoting by $\alpha$, $\beta$ and $\nu$ the $T$"/graph structures of $A$, $B$ and $M$, consider on the left"/hand side below the composition of the cells exhibiting $J$ as a $T$"/closed horizontal morphism, $d$ as a $T$"/morphism and $l$ as the left Kan extension of $d$ along $J$.
		\begin{displaymath}
			\begin{tikzpicture}[textbaseline]
				\matrix(m)[math35]{TA & TB & B \\ TA & A & B \\ TM & M & M \\};
				\path[map]	(m-1-1) edge[barred] node[above] {$TJ$} (m-1-2)
										(m-1-2) edge[barred] node[above] {$\beta$} (m-1-3)
										(m-2-1) edge[barred] node[above] {$\alpha$} (m-2-2)
														edge node[left] {$Td$} (m-3-1)
										(m-2-2) edge[barred] node[above] {$J$} (m-2-3)
														edge node[right] {$d$} (m-3-2)
										(m-2-3) edge node[right] {$l$} (m-3-3)
										(m-3-1) edge[barred] node[below] {$\nu$} (m-3-2);
				\path				(m-1-1) edge[eq] (m-2-1)
										(m-1-3) edge[eq] (m-2-3)
										(m-3-2) edge[eq] (m-3-3);
				\draw				($(m-1-1)!0.5!(m-2-3)$) node[rotate=-90] {$\leq$}
										($(m-2-2)!0.5!(m-3-3)$) node {$\eta$}
										($(m-2-1)!0.5!(m-3-2)$) node[rotate=-90] {$\leq$};
			\end{tikzpicture} = \begin{tikzpicture}[textbaseline]
				\matrix(m)[math35]{TA & TB & B \\ TM & TM & M \\};
				\path[map]	(m-1-1) edge[barred] node[above] {$TJ$} (m-1-2)
														edge node[left] {$Td$} (m-2-1)
										(m-1-2) edge[barred] node[above] {$\beta$} (m-1-3)
														edge node[right, inner sep=1.5pt] {$Tl$} (m-2-2)
										(m-1-3) edge node[right] {$l$} (m-2-3)
										(m-2-2) edge[barred] node[below] {$\nu$} (m-2-3);
				\path				(m-2-1) edge[eq] (m-2-2);
				\draw				($(m-1-2)!0.5!(m-2-3)$) node[rotate=-90] {$\leq$}
										($(m-1-1)!0.5!(m-2-2)$) node {$T\eta$};
			\end{tikzpicture}
		\end{displaymath}
		Under the assumptions on $l$ it follows from \propref{image of Kan extension satisfying Beck-Chevalley} that $Tl$ is the left Kan extension of $Td$ along $TJ$, such that it satisfies the Beck"/Chevalley condition. Hence, by \thmref{Beck-Chevalley} the composite on the left factors through the `$T$"/image' of $\eta$ as shown, and this factorisation exhibits $l$ as a $T$"/morphism.
		
		That $l$ is $T$"/open whenever $d$ and $J$ are is shown by
		\begin{displaymath}
			\nu(\id, l) = \nu \hc l^* \overset{\text{(i)}}= \nu \hc d^* \hc J \overset{\text{(ii)}}\leq Td^* \hc \alpha \hc J \overset{\text{(iii)}}\leq Td^* \hc TJ \hc \beta = T(d^* \hc J) \hc \beta \overset{\text{(i)}}= Tl^* \hc \beta,
		\end{displaymath}
		where the equalities (i) follow from the Beck"/Chevalley condition for $l$ and the inequalities (ii) and (iii) follow from $d$ and $J$ being $T$"/open respectively.
	\end{proof}
	
	\begin{example}
		Let $\hmap JAB$ be an upper hemi"/continuous relation (\exref{hemi-continuous relation}) between topological spaces, such that $\rev Jy \subseteq A$ is non"/empty for each $y \in B$, and let $\map dA{\brks{-\infty, \infty}}$ be a upper semi"/continuous map (\exref{Scott topology}). Regarding $A$ and $B$ as discrete ordered sets, by the extreme value theorem (see e.g.\ \thmref{extreme value theorem for ordered closure spaces} below) the conditions on $J$ imply that the left Kan extension $\map lB{\brks{-\infty, \infty}}$ of $d$ along $J$ satisfies the Beck"/Chevalley condition, that is it is given by the maxima
		\begin{displaymath}
			ly = \max_{x \in \rev Jy} dx
		\end{displaymath}
		for all $y \in B$. Applying the previous theorem we find that $l$ is upper semi"/continuous, thus partly recovering Theorem 2 of Section~VI.3 of \cite{Berge59} (or Lemma~17.30 of \cite{Aliprantis-Border06}) which is stated in terms of the reverse $\hmap{\rev J}BA$. Moreover its second assertion means that $l$ is $U$"/open (\exref{open/closed maps between ordered topological spaces}) whenever $d$ is $U$"/open and $J$ is lower hemi"/continuous.
	\end{example}
	
	\begin{example}
		Let $\hmap JAB$ be a relation between topological spaces that is both lower and upper hemi"/continuous (\exref{hemi-continuous relation}), while $Jx \subseteq B$ is non"/empty for each $x \in A$, and suppose that $\map eB{\brks{-\infty, \infty}}$ is continuous. Since a function into $\brks{-\infty, \infty}$ is continuous precisely if it is both lower and upper semi"/continuous (\exref{Scott topology}), by combining \exref{right Kan extension into extended real line} and the previous example (the latter applied to $\hmap{\rev J}BA$ and $\map eB{\brks{-\infty, \infty}}$) we find that the extension \mbox{$\map mA{\brks{-\infty, \infty}}$} of $e$ along $J$, given by
		\begin{displaymath}
			mx = \max_{y \in Jx} ey
		\end{displaymath}
		for all $x \in A$, is continuous. This recovers the main assertion of Berge's maximum theorem, as stated in Section~4.3 of \cite{Berge59}.
	\end{example}
	
	\begin{example}
		Consider the Lawvere quantale $\brks{0, \infty}$ with the canonical normalised modular approach space structure with point"/set distance $\delta_{\sup}$ given in \exref{Lawvere quantale as a modular approach space}. Let $\map J{A \times B}{\brks{0, \infty}}$ be a $U$"/closed metric relation (see \thmref{P and U horizontal morphism}(b) and \exref{open/closed relations}) between approach spaces that is discrete, that is $\im J \subseteq \set{0, \infty}$, such that $\rev J_0 y \neq \emptyset$ for all $y \in B$, and suppose that $\map dA{\brks{0, \infty}}$ is continuous. Regarding $A$ and $B$ as modular approach spaces with discrete metrics, the left Kan extension $\map lB{\brks{0, \infty}}$ of $d$ along $J$ is defined on $y \in B$ by
		\begin{displaymath}
			ly = \inf_{x \in \rev J_0 y} dx;
		\end{displaymath}
		see \exref{Kan extensions in ordered sets}. By \thmref{extreme value theorem for modular V-valued pseudotopological spaces} below the conditions on $J$ ensure that $l$ satisfies the Beck"/Chevalley condition. It thus follows from the previous theorem that $l$ is continuous, while it is $P$"/open (\exref{open/closed maps between modular V-valued closure spaces}) whenever $d$ and $J$ are $P$"/open.
	\end{example}
	
	\begin{example}
		A result for left Kan extensions $\map lB{\Delta_{\tn}}$ between probabilistic approach spaces, analogous to the previous example, can be derived from the theorem above as well. In this case however the hypothesis $U(d^* \hc J) = Ud^* \hc UJ$ may not be satisfied: while the extension of the ultrafilter monad $U$ to $\enRel{\brks{0, \infty}}$ is strict (\exref{ultrafilter monad}), I do not know if its extension to $\enRel{\Delta_{\tn}}$ is too. Moreover, in general \thmref{extreme value theorem for modular V-valued pseudotopological spaces} applies only in the cases where the convolution product on $\Delta_{\tn}$ is induced by multiplication on $\brks{0, 1}$ or the \L ukasiewicz operation, see \propref{condition d for Delta}; it may fail to apply in the case of the frame operation $p \tn q = \min\set{p, q}$, see \exref{counter example condition d}.
	\end{example}
	
	\section{Generalisations of the extreme value theorem} \label{extreme value theorem section}
	In this last section we investigate Kan extensions that satisfy the Beck"/Chevalley condition (\thmref{Beck-Chevalley}). We treat two cases: the first concerning left Kan extensions between modular closure spaces (\exref{modular closure space}) and the second concerning a restricted class of left Kan extensions between modular $\V$"/valued pseudotopological spaces (\exref{modular (U, V)-category}). Starting with the former, the theorem below is a straightforward generalisation of Weierstra\ss' extreme value theorem (see e.g.\ Corollary~2.35 of \cite{Aliprantis-Border06}).
	
	A subset $S \subseteq A$ of a closure space $A$ is called \emph{compact} if for every family $(V_i)_{i \in I}$ of closed subsets of $A$ we have
	\begin{displaymath}
		\Bigpars{\Forall_{\substack{J \subseteq I\\ \text{$J$ is finite}}} S \isect \Isect_{j \in J} V_j \neq \emptyset} \quad \implies \quad S \isect \Isect_{i \in I} V_i \neq \emptyset.
	\end{displaymath}
	It is straightforward to check that this is equivalent to the definition of compact subsets in terms of finite open subcovers, which is often used in the case of topological spaces. In particular, continuous maps of closure spaces preserve compact sets.
	
	Recall that a subset $S$ of an ordered set $M$ is called \emph{up"/directed} whenever it is non"/empty and every finite subset of $S$ has an upper bound in $S$, that is for all $u, v \in S$ there is $w \in S$ with $u \leq w$ and $v \leq w$.
	
	\begin{theorem} \label{extreme value theorem for ordered closure spaces}
		Let $A$ and $M$ be modular closure spaces, $\map dAM$ a monotone continuous function and $\hmap JAB$ a modular relation into an ordered set $B$. The left Kan extension $\map lBM$ of $d$ along $J$, if it exists, satisfies the Beck"/Chevalley condition whenever $M$ is normalised and $d(\rev Jy) \subseteq M$ is compact and up"/directed for each $y \in B$.
	\end{theorem}
	\begin{proof}
		First recall from \exref{modular closure space} that all principal upsets $\upset z$ in $M$ are closed because $M$ is normalised. As described in \exref{Beck-Chevalley for ordered sets} we have to show that for each $y \in B$ the set $d(\rev Jy)$ has a maximum in $M$. To this end consider in $M$ the family of closed principal subsets
		\begin{displaymath}
			\bigpars{\upset z}_{z \in d(\rev Jy)}.
		\end{displaymath}
		We claim that $d(\rev Jy) \isect \Isect_{i = 1}^n \upset z_i \neq \emptyset$ for any finite sequence $z_1, \dotsc, z_n \in d(\rev J y)$. Indeed, since $d(\rev Jy)$ is up"/directed the set $\set{z_1, \dotsc, z_n}$ has an upper bound $z$. As $\upset z \subseteq \upset z_i$ for each $i$ we conclude $z \in d(\rev Jy) \isect \upset z \subseteq d(\rev Jy) \isect \Isect_{i = 1}^n \upset z_i$. By compactness of $d(\rev Jy)$ it follows that there exists some $w \in d(\rev J y) \isect \Isect_{z \in d(\rev Jy)} \upset z$. But this means $z \leq w$ for all $z \in d(\rev Jy)$, showing that $w$ is a maximum of $d(\rev Jy)$.
	\end{proof}
	
	The second generalisation of the extreme value theorem applies to a restricted class of left Kan extensions between modular $\V$"/valued pseudotopological spaces (\exref{modular (U, V)-category}). Remember that, for a $\V$"/profunctor $\hmap JAB$ and $v \in \V$, we write $\hmap{J_v}AB$ for the ordinary relation given by
	\begin{displaymath}
		x J_v y \quad \iff \quad v \leq J(x,y)
	\end{displaymath}
	for all $x \in A$ and $y \in B$.
	\begin{theorem} \label{extreme value theorem for modular V-valued pseudotopological spaces}
		Let $\V$ be a completely distributive quantale so that $\V = \V_{\dashcirc}$ is itself a modular $\V$"/valued topological space, see \exref{V as a modular approach space}. Consider a modular $\V$"/valued pseudotopological space $A$ and a $\V$"/category $B$, as well as a continuous $\V$"/functor $\map dA{\V_{\dashcirc}}$ and a $\V$"/profunctor $\hmap JAB$. The left Kan extension $\map lB{\V_{\dashcirc}}$ of $d$ along $J$ satisfies the Beck"/Chevalley condition whenever $J$ satisfies the following conditions:
		\begin{enumerate}[label=\textup{(\alph*)}]
			\item $J$ is discrete, that is $\im J \subseteq \set{\bot, k}$;
			\item $J$ is $U$"/compact, see \defref{U-compact};
			\item for each $y \in B$ the set $d(\rev J_ky)$ is up"/directed in $\V$;
			\item for each $y \in B$
			\begin{displaymath}
				k \leq \sup_{z \in d(\rev J_ky)} \bigpars{\sup d(\rev J_ky) \dashcirc z}.
			\end{displaymath}
		\end{enumerate}
	\end{theorem}
	Notice that, by \exref{Kan extensions in ordered sets}, condition (a) above means that the Kan extension $l$ is given by
	\begin{equation} \label{definition of ly}
		ly = \sup d(\rev J_ky)
	\end{equation}
	for all $y \in B$, so that condition (d) can be rewritten as
	\begin{displaymath}
		k \leq \sup_{z \in d(\rev J_ky)} \pars{ly \dashcirc z}
	\end{displaymath}
	for each $y \in B$. If $\V = \brks{0, \infty}$ then this inequality follows from the fact that the map $ly \dashcirc \dash$ preserves infima.
	
	After giving the proof of the theorem above, \exref{counter example condition a} below shows that condition~(a) on $J$ cannot be dispensed with. \propref{condition d for Delta} shows that condition~(d) holds for $\V = \Delta_\times$ and $\Delta_{\tn}$, where $\tn$ denotes the \L ukasiewicz operation (see \exref{Delta}). \exref{counter example condition d} then shows that condition~(d) does not generally hold in the case of $\V = \Delta_{\min}$.
	\begin{proof}
		Using \propref{Beck-Chevalley condition}, first notice that discreteness of $J$ means that the Beck"/Chevalley condition for $l$ reduces to the inequality
		\begin{equation}\label{Beck-Chevalley inequality}
			k \leq \sup_{x \in \rev J_ky} \V_{\dashcirc}(ly, dx)
		\end{equation}
		for all $y \in B$. While in general there might not be any $x \in \rev J_ky$ with $k \leq \V_{\dashcirc}(ly, dx) = ly \dashcirc dx$ notice that, if such a $x$ does exist then $ly \leq dx$ follows, so that the supremum $ly$ in \eqref{definition of ly} is attained as a maximum.
		
		Discreteness of $J$ also means that $U$"/compactness of $J$ (\defref{U-compact}) reduces to
		\begin{equation}\label{U-compactness}
			k \leq \sup_{x \in \rev J_ky} \alpha(\mf y, x)
		\end{equation}
		for all $y \in B$ and $\mf y \in UA$ with $\rev J_ky \in \mf y$, where $\hmap\alpha{UA}A$ is the $\V$"/valued convergence relation of $A$.
		
		Let us fix $y \in B$. Using the above any $\mf y \in UA$ with $\rev J_ky \in \mf y$ gives a lower bound for the right"/hand side of \eqref{Beck-Chevalley inequality} as follows, where $\hmap{\nu \dfn \nu_{\inf}}{U\V_{\dashcirc}}{\V_{\dashcirc}}$ is the $\V$"/valued convergence relation of $\V_{\dashcirc}$ (see \exref{V as a modular approach space}).
		\begin{align*}
			\sup_{x \in \rev J_k y} \V_{\dashcirc}(ly, dx) &\overset{\text{(i)}}= \sup_{x \in \rev J_k y} \nu\bigpars{(\iota \of l)(y), dx} \\
			&\overset{\text{(ii)}}\geq \sup_{\substack{x \in \rev J_k y \\ \mf x \in UA}} (U\V_{\dashcirc})\bigpars{(\iota \of l)(y), (Ud)(\mf x)} \tens \alpha(\mf x, x) \\
			&\geq (U\V_{\dashcirc})\bigpars{(\iota \of l)(y), (Ud)(\mf y)} \tens \sup_{x \in \rev J_ky} \alpha(\mf y, x) \\
			&\overset{\text{(iii)}}\geq (U\V_{\dashcirc})\bigpars{(\iota \of l)(y), (Ud)(\mf y)}
		\end{align*}
		 The equality denoted (i) here is a consequence of $\V_{\dashcirc}$ being normalised, see \eqref{normalised T-category}; (ii) follows from $(Ud)^* \hc \alpha \leq \nu(\id, d)$, which is obtained by applying \lemref{horizontal cells} to the cell exhibiting $d$ as a $U$"/morphism; (iii) follows from \eqref{U-compactness}.
		 
		We conclude that to prove the Beck"/Chevalley condition for $l$, that is \eqref{Beck-Chevalley inequality} holds, it suffices to construct an ultrafilter $\mf y$ on $\rev J_k y \subseteq A$ that satisfies
		\begin{equation} \label{reduced Beck-Chevalley inequality}
			k \leq (U\V_{\dashcirc})\bigpars{(\iota \of l)(y), (Ud)(\mf y)}.
		\end{equation}
		To construct the ultrafilter $\mf y$ notice that condition (c) on $J$ implies that the sets
		\begin{displaymath}
			X_z \dfn \rev J_k y \isect \inv d(\upset z),
		\end{displaymath}
		where $z$ ranges over $d(\rev J_ky)$, form a proper filter base; we choose $\mf y$ to be any ultrafilter containing all of them. From the definition of the $X_z$ it follows that $\mf y$ contains $\rev J_k y$ as well as the preimages $\inv d(\upset z)$, for each $z \in d(\rev J_ky)$. The latter implies that $\mf y$ contains the sets $\inv d\bigpars{\upset\bigpars{ly \tens (ly \dashcirc z)}}$ too: this is a consequence of the inequalities $ly \tens (ly \dashcirc z) \leq z$, which form the counit of the adjunction $(ly \tens \dash) \ladj (ly \dashcirc \dash)$.
		
		That $\mf y$ satisfies \eqref{reduced Beck-Chevalley inequality} is shown by
		\begin{align*}
			(U\V_{\dashcirc})\bigpars{(\iota \of l)(y), (Ud)(\mf y)} &= \sup\bigbrcs{v \in \V \mid \bigbrks{(\iota \of l)(y)} \bigpars{U(\V_\dashcirc)_v} \bigbrks{(Ud)(\mf y)}} \\
			&= \sup\bigbrcs{v \in \V \mid \inv d\bigpars{\upset(ly \tens v)} \in \mf y} \\
			&\geq \sup_{z \in d(\rev J_ky)} ly \dashcirc z \geq k
		\end{align*}
		where the four (in-)equalities are consequences of respectively the equivalent definition \eqref{equivalent definition UJ} of $U\V_{\dashcirc}$, the equivalences below, the discussion above and, finally, condition (d) on $J$. 
		\begin{align*}
			\bigbrks{(\iota \of l)(y)} \bigpars{U(\V_\dashcirc)_v} \bigbrks{(Ud)(\mf y)} \quad &\iff \quad (\V_\dashcirc)_v(ly) \in (Ud)(\mf y) \\
			&\iff \quad \set{x \in A \mid v \leq ly \dashcirc dx} \in \mf y \\
			&\iff \quad \set{x \in A \mid ly \tens v \leq dx} \in \mf y \\
			&\iff \quad \inv d\bigpars{\upset(ly \tens v)} \in \mf y
		\end{align*}
		This completes the proof.
	\end{proof}
	
	\begin{example} \label{counter example condition a}
		To see that condition (a) of the theorem above, that is discreteness of $\hmap JAB$, cannot be left out we consider the Sierpi\'nski space $\2 = \set{\bot, \top}$ (see \exref{Scott topology}) as a normalised modular approach space (\exref{modular (U, V)-category}), by taking its image under the composite of embeddings
		\begin{displaymath}
			\Top \xrar I \mathsf{App} \xrar N \mathsf{ModApp}.
		\end{displaymath}
		The functor $N$ here is described in \exref{modular (U, V)-category}, while the functor $I$ maps any topological space $A$, with closure operation $S \mapsto \bar S$, to the approach space $IA = (A, \delta)$ with point"/set distance given by
		\begin{displaymath}
			\delta(S, x) = \begin{cases}
				0 & \text{if $x \in \bar S$;} \\
				\infty & \text{otherwise.}
			\end{cases}
		\end{displaymath}
		
		Regarding the Lawvere quantale $\brks{0, \infty}$ as equipped with its canonical modular point"/set distance $\delta_{\sup}$ (\exref{Lawvere quantale as a modular approach space}), let $\map d\2{\brks{0,\infty}}$ and $\hmap J\2*$, where $* = \set*$ denotes the singleton approach space, be defined by
		\begin{displaymath}
			d(\bot) = 2, \qquad d(\top) = 0, \qquad J(\bot, *) = 0 \qquad \text{and} \qquad J(\top, *) = 1.
		\end{displaymath}
		It is straightforward to check that $d$ is a non"/expansive continuous map of modular approach spaces, that $J$ is a $U$"/compact modular metric relation, and that their left Kan extension \mbox{$\map l*{\brks{0,\infty}}$} is given by (see \propref{Kan extension expression})
		\begin{displaymath}
			l(*) = \min\bigbrcs{d(\bot) + J(\bot, *), d(\top) + J(\top, *)} = \min\set{2, 1} = 1.
		\end{displaymath}
		On the other hand we have
		\begin{multline*}
			\max\bigbrcs{\bigpars{d(\bot) \ominus l(*)} + J(\bot, *), \bigpars{d(\top) \ominus l(*)} + J(\top, *)} \\
			= \textstyle\max\set{2 \ominus 1 + 0, 0 \ominus 1 + 1} = 1 > 0,
		\end{multline*}
		so that by \exref{Beck-Chevalley for metric spaces} $l$ fails to satisfy the Beck"/Chevalley condition.
	\end{example}
	
	The following proposition shows that condition (d) of the previous theorem always holds in the cases $\V = \Delta_\times$ and $\V = \Delta_{\tn}$, the quantale of distance distribution functions equipped with multiplication $\tens$ given by convolution with respect to either multiplication $\times$ or the \L ukasiewicz operation on $\brks{0,1}$; see \exref{Delta}. Remember that the unit $\map k{\brks{0, \infty}}{\brks{0, 1}}$ of $\Delta_{\tn}$ is given by $k(t) = 1$ for all $t > 0$.
	\begin{proposition} \label{condition d for Delta}
		Consider the quantale $\Delta_{\tn}$ where either $\tn = \times$ or $\tn$ is the \L ukasiewicz operation, see \exref{Delta}. For any up"/directed set $\Phi \subseteq \Delta_{\tn}$:
		\begin{displaymath}
			\sup_{\phi \in \Phi} (\sup \Phi \dashcirc \phi) = k.
		\end{displaymath}
	\end{proposition}
	\begin{proof}
		Let us write $\sigma \dfn \sup\Phi$. Clearly if $\sigma = 0$, the bottom element of $\Delta_{\tn}$, then $\Phi = \set 0$ so that the identity above reduces to $0 \dashcirc 0 = k$ which immediately follows from the definition of $\dashcirc$. 
		
		Hence we assume $\sigma > 0$. For each $u \in (0, \infty)$ and $p \in (0, 1)$ let us denote by $\pi_{(u, p)}$ the distance distribution function given by
		\begin{displaymath}
			\pi_{(u, p)}(s) \dfn \begin{cases}
				0 & \text{if $s \leq u$;} \\
				p & \text{if $s > u$.}
			\end{cases}
		\end{displaymath}
		Because
		\begin{displaymath}
			\sup_{\substack{u \in (0, \infty) \\ p \in (0, 1)}} \pi_{(u, p)} = k \qquad \text{and} \qquad \sigma \dashcirc \psi = \sup \set{\chi \in \Delta_{\tn} \mid \sigma \tens \chi \leq \psi}
		\end{displaymath}
		(see \exref{V-profunctors}), it suffices to prove that for every $u \in (0, \infty)$ and every $p \in (0, 1)$ there is a $\psi \in \Phi$ with
		\begin{displaymath}
			\sigma \tens \pi_{(u, p)} \leq \psi.
		\end{displaymath}
		Unpacking the convolution product $\tens$ and using the left"/continuity of $\sigma$, we find that this means
		\begin{equation} \label{condition on psi}
			\sigma(s \ominus u) \tn p \leq \psi(s)
		\end{equation}
		for all $s \in \brks{0, \infty}$, where $\ominus$ denotes truncated difference.
		
		To show that we can find such $\psi \in \Phi$ for any $u \in (0, \infty)$ and $p \in (0, 1)$ we write
		\begin{displaymath}
			B \dfn \sigma(\infty) = \sup_{s \in (0, \infty)} \sigma(s) > 0
		\end{displaymath}
		and, for each $n \geq 1$,
		\begin{displaymath}
			t_n = \inf \set{s \in \brks{0, \infty} \mid \sigma(s \ominus u/2) > B \tn^n p}
		\end{displaymath}
		where
		\begin{displaymath}
			B \tn^n p \dfn \overbrace{(\dotsb((B \tn p) \tn p)\dotsb\tn p)}^\text{$n$ times}.
		\end{displaymath}
		As $u/2 > 0$ we have $\sigma(t_n) > B \tn^n p$ so that, for each $n \geq 1$, there must be some $\phi_n \dfn \phi \in \Phi$ with $\phi(t_n) > B \tn^n p$.
		
		Now for $n = 1$ we have for all $s \in \brks{t_1, \infty}$:
		\begin{displaymath}
			\phi_1(s) \geq \phi_1(t_1) > B \tn p \geq \sigma(s \ominus u) \tn p,
		\end{displaymath}
		where the last inequality follows from the definition of $B$. By definition of $t_n$ we have $\sigma(s \ominus u/2) \leq B \tn^n p$ for all $s < t_n$ so that, for all $n \geq 1$ and all $s \in [t_{n+1}, t_n)$ we have:
		\begin{displaymath}
			\phi_{n+1} (s) \geq \phi_{n+1}(t_{n+1}) > B \tn^{n+1} p \geq \sigma(s \ominus u/2) \tn p \geq \sigma(s \ominus u) \tn p.
		\end{displaymath}
		Finally, below we will show that in both cases of $\tn$ there is an integer $N \geq 1$ such that
		\begin{displaymath}
			\phi_N(s) \geq \sigma(s \ominus u) \tn p
		\end{displaymath}
		for all $s \in [0, t_N]$. We have thus found a finite number of distance distribution functions $\phi_1, \dotsc, \phi_N$ in $\Phi$ such that for each $s \in \brks{0, \infty}$ there is some $\phi_i$ with
		\begin{displaymath}
			\phi_i(s) \geq \sigma(s \ominus u) \tn p.
		\end{displaymath}
		As $\Phi$ is up"/directed it contains an upper bound $\psi$ of $\set{\phi_1, \dotsc, \phi_N}$. From the above it follows that $\psi$ satisfies \eqref{condition on psi} for all $s \in \brks{0, \infty}$, thus concluding the proof.
		
		It remains to show the existence of $N$. In the case that $\tn$ is the \L ukasiewicz operation, we take $N$ to be the minimal $n$ such that $B \tn^n p = 0$. Then $\sigma(t_N) = 0$ by the left"/continuity of $\sigma$ so that for all $s \in [0, t_N]$:
		\begin{displaymath}
			\phi_N(s) \geq 0 = \sigma(t_N) \tn p \geq \sigma(s \ominus u) \tn p.
		\end{displaymath}
		
		In the case that $\tn = \times$ we define $l \in [0, \infty)$ and $N \geq 1$ by
		\begin{displaymath}
			l \dfn \max\set{s \in \brks{0, \infty} \mid \sigma(s) = 0} \qquad \text{and} \qquad N \dfn \min\set{n \geq 1 \mid B \times p^n < \sigma(l + u/2)};
		\end{displaymath}
		that these extrema exist follows from the left"/continuity of $\sigma$ and the fact that $\sigma(l + u/2) > 0$. By definition of $t_N$ we have $t_N \leq l + u$. Hence for all $s \in \brks{0, t_N}$ we have $s \ominus u \leq l$ so that $\sigma(s \ominus u) = 0$. Thus
		\begin{displaymath}
			\phi_N(s) \geq 0 = \sigma(s \ominus u) \times p. \qedhere
		\end{displaymath}
	\end{proof}
	
	\begin{example} \label{counter example condition d}
		In the case of $\V = \Delta_{\min}$ (see \exref{Delta}), the assertion of the previous proposition is false in general. To see this consider the up"/directed set $\Phi \subset \Delta_{\min}$ consisting of the distance distribution functions
		\begin{displaymath}
			\map{\phi_i}{\brks{0, \infty}}{\brks{0, 1}}\colon \phi_i(t) = \begin{cases}
				t & \text{if $t \leq i$;} \\
				i & \text{if $t > i$,}
			\end{cases}
		\end{displaymath}
		where $i$ ranges over the real numbers in $(0, \frac12)$, so that $\sup \Phi = \phi_{\frac 12}$. We claim that the assertion of the previous proposition fails, that is $\sup_{i \in (0, \frac 12)} (\phi_{\frac 12} \dashcirc \phi_i) < k$, where $k$ is the unit of $\Delta_{\min}$. To see this remember that (see \exref{V-profunctors})
		\begin{equation} \label{condition d dashcirc}
			\phi_{\frac 12} \dashcirc \phi_i = \sup\set{\chi \in \Delta_{\min} \mid \phi_{\frac 12} \tens \chi \leq \phi_i}.
		\end{equation}
		For each $\chi \in \Delta_{\min}$ with $\phi_{\frac 12} \tens \chi \leq \phi_i$ we have
		\begin{displaymath}
			i = \phi_i(\infty) \geq (\phi_{\frac 12} \tens \chi)(\infty) = \min\bigbrcs{\phi_{\frac 12}(\infty), \chi(\infty)} = \min\bigbrcs{\textstyle\frac 12, \chi(\infty)},
		\end{displaymath}
		where the second equality follows easily from the definition of the convolution product $\tens$, see \exref{Delta}. Since $i < \frac 12$ we conclude that $\chi(\infty) \leq i < \frac 12$ for all $\chi$ in \eqref{condition d dashcirc}. Because $k(\infty) = 1$ it follows that $\sup_{i \in (0, \frac 12)} (\phi_{\frac 12} \dashcirc \phi_i) < k$, as claimed.
	\end{example}
	
	\section*{Acknowledgements} \addcontentsline{toc}{section}{Acknowledgements}
	This paper has benefitted from discussions with Bob Par\'e, during my visit to Dalhousie University in August 2016. In particular his suggestions led to \thmref{Beck-Chevalley} above. I would like to thank Bob, as well as the @CAT-group for partly funding my visit. I thank the anonymous referee for helpful suggestions and the prompt review of this paper.
	
	\section*{References}
  \bibliographystyle{elsart-num-sort}
	\bibliography{../_other/main}
\end{document}